\documentclass[a4paper]{amsart}

\usepackage{lmodern}
\usepackage[T1]{fontenc}
\usepackage[english]{babel}
\usepackage{amsfonts}
\usepackage{amsthm}
\usepackage{amssymb}
\usepackage{mathcomp}
\usepackage{mathrsfs}
\usepackage[all]{xy}
\usepackage{graphicx}
\usepackage{epstopdf}
\usepackage[centertags]{amsmath}
\usepackage{vmargin}
\usepackage{tikz}
\usepackage{dsfont}
\usepackage{cases}

\definecolor{BleuTresFonce}{rgb}{0.215, 0.215, 0.36}
\definecolor{airforceblue}{rgb}{0.36, 0.54, 0.66}

\usepackage[colorlinks,final,hyperindex]{hyperref}
\hypersetup{citecolor=black, linkcolor=black}

\theoremstyle{plain}
\newtheorem{thm}{Theorem}
\newtheorem*{thm*}{Theorem}
\newtheorem{lemma}{Lemma}
\newtheorem{prop}{Proposition}
\newtheorem{cor}{Corollary}

\newtheorem*{theointro1}{Theorem 2}
\newtheorem*{theointro2}{Theorem 3}
\newtheorem*{theointro3}{Theorem 6}
\newtheorem*{theointro4}{Theorem 7}

\theoremstyle{definition}
\newtheorem{defin}{Definition}
\newtheorem*{nota}{\sc Notations}

\theoremstyle{remark}
\newtheorem{rmk}{\sc Remark}
\newtheorem{eg}{\sc Example}

\setpapersize{A4}

\newcommand{\id}{\mathrm{Id}}
\newcommand{\ii}{\mathrm{id}}

\newcommand{\Lin}{Lin}

\newcommand{\BijC}{\mathsf{Bij}_{C}}

\newcommand{\PP}{\mathcal{P}}
\newcommand{\Sy}{\mathbb{S}}
\newcommand{\Tree}{\mathsf{Tree}}

\newcommand{\poubelle}[1]{}

\title{From homotopy operads to infinity-operads}
\author{Brice Le Grignou}
\address{University  Nice Sophia Antipolis, CNRS,  LJAD, UMR 7351, 06100 Nice, France.}
\email{Brice.LE\_GRIGNOU@unice.fr}

\date{\today}



\thanks{The author was supported by the ANR HOGT and the ANR SAT grants.}

\begin{document}

\maketitle

\begin{abstract}
The goal of the present paper is to compare, in a precise way, two notions of operads up to homotopy which appear in the literature. Namely,  we construct a functor from the category of strict unital homotopy colored operads  to the category of infinity-operads. The former notion, that we make precise, is the operadic generalization of the notion of A-infinity-categories and the latter notion was defined by Moerdijk--Weiss in order to generalize the simplicial notion of infinity-category of Joyal--Lurie. This functor extends in two directions the simplicial nerve of Faonte--Lurie for A-infinity-categories and the homotopy coherent nerve of Moerdijk--Weiss for differential graded operads; it is also shown to be equivalent to a big nerve \`a la Lurie for differential graded operads. We prove that it satisfies some homotopy properties with respect to weak equivalences and fibrations; for instance, it is shown to be a right Quillen functor.
\end{abstract}

\setcounter{tocdepth}{1}
\tableofcontents

\section*{Introduction}

In Algebra, the structure relations hold strictly. This is for instance the case for sets, groups, rings, vector spaces, etc. So all of these examples are well encoded by categories. In this context, two objects are considered to represent equivalent notions if they are related by an isomorphism. However, in some areas of Mathematics, this equivalence relation is too strong and one would like to consider only weakly equivalent objects. For example, two categories are considered to be essentially the same if they are equivalent, two chain complexes give rise to the same homology groups when they are quasi-isomorphic, and two topological spaces have the same homotopy type if they are related by weak homotopy equivalences. In these examples, one can consider the Dwyer--Kan localization with respect to the class of weak equivalences. This provides us with a higher category structure made up of $2$-morphisms, which are morphisms between morphisms, and, in general, $n$-morphisms, which are morphisms between $(n-1)$-morphisms, for integers $n$. These $n$-morphisms are invertible for $n\geq2$ and they encode coherent higher homotopies. In this way, one can perform higher algebra in these $(\infty,1)$-categories, which can actually be either categories enriched in simplicial sets, like the Dwyer--Kan localization, or Joyal's quasi-categories, called $\infty$-categories by Jacob Lurie \cite{Lurie09}. Recall that a quasi-category is a simplicial set satisfying a lifting condition, which defines the composition of morphisms up to homotopy.\\

Categories enriched in chain complexes, called dg categories, are a kind of linear version of $\infty$-categories. In his book \cite{Lurie12}, Jacob Lurie presents several ways to interpret  dg categories as $\infty$-categories. Relaxing the associativity relation of the composition rule in a dg category leads to the definition of an $\mathcal{A}_{\infty}$-category, notion which now plays a key role in symplectic geometry \cite{FOOO09I}. In \cite{Faonte13}, Giovanni Faonte extends one of Lurie's constructions to the case of $\mathcal{A}_{\infty}$-categories and so builds a simplicial nerve of $\mathcal{A}_{\infty}$-categories. Let $[n]$ be the poset $0<1< \cdots <n$ canonically enriched into a dg category; the simplicial nerve of an $\mathcal{A}_{\infty}$-category $\mathsf{C}$ is defined by the following simplicial set:
\begin{equation*}
\mathrm{N}_{\mathcal{A}_{\infty}}(\mathsf{C})_n:=\mathrm{Hom}_{\mathcal{A}_{\infty}\textrm{-}\mathsf{cat}}([n],\mathsf{C})\ ,
\end{equation*}
for any integer $n$. Faonte's main result asserts that this actually forms a quasi-category.\\

Representations of associative algebras are made up of linear operators. To encode multi-linear operators, which are operators with many inputs but one output, one can use representations of operads. Furthermore, to encode multi-linear operators on a many components object, one can use reprensentations of colored operads. A colored operad $\mathscr{P}$ is the data of a set of colors and sets of multi-linear operations, with inputs and output labeled  by the colors. One can only compose operations with matching colors. Moreover, the action of the symmetric groups allows one to permute the inputs. Note that any category may be viewed as a colored operad, where the objects are the colors and where the maps are  operations of arity one.\\

On the one hand, the notion of quasi-category has an analogue in the world of operads which is called an $\infty$-operad; this notion has been developed by Ieke Moerijk and Ittay Weiss in \cite{MoerdijkWeiss07}. First one needs an operadic generalization of the notion of a simplicial set: it is defined as a contravariant functor from trees to sets and called a dendroidal set. In the same way as a quasi-category is a simplicial set satisfying a lifting property, an $\infty$-operad is a dendroidal set satisfying a similar lifting property. On the other hand, the differential graded notion of an $\mathcal{A}_{\infty}$-algebra has also been extended  by Pepijn Van der Laan \cite{VanDerLaan03} to a notion called homotopy operad. A homotopy operad is the data of operations, with zero or many inputs but one output, structured in a family of differential graded $\mathbb{K}$-modules. The group actions and composition maps are well defined but the latter ones are ``associative'' only up to a family of higher homotopies. Actually, a homotopy operad whose operations only have one input is an $\mathcal{A}_{\infty}$-algebra. \\

The present paper has the following two main goals. The first one is to define a suitable notion of homotopy colored operads, with a homotopy control of the unit, and which completes the following commutative diagram.
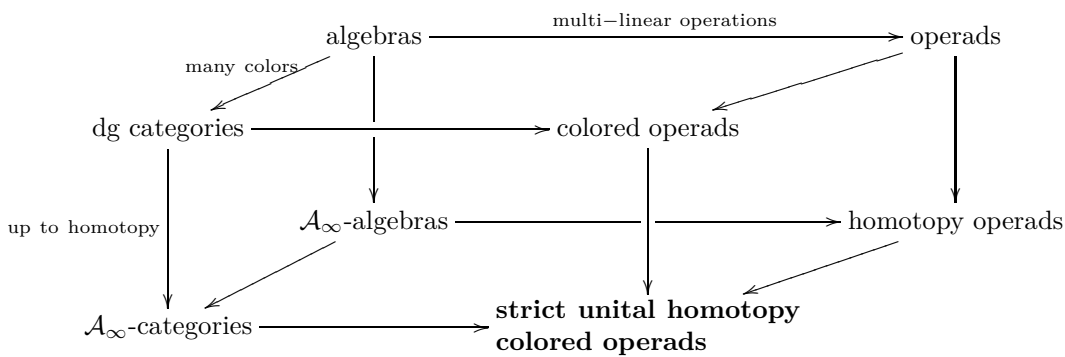
\begin{equation*}
\xymatrix@R=20pt@C=12pt{%
& \mathrm{algebras}  \ar@{->}[rr]^{\mathrm{multi-linear\ operations}}\ar@{->}[dl]_{\mathrm{many\ colors}}\ar@{->}[dd]|!{[d];[d]}\hole && \mathrm{operads} %
\ar@{->}[dl] \ar@{->}[dd]\\
\text {dg categories}\ar@{->}[rr]\ar@{->}[dd]_{\mathrm{up\ to\ homotopy}}&& \mathrm{colored} \  \mathrm{operads} \ar@{->}[dd]\\
& \mathcal{A}_{\infty}\textrm{-}\mathrm{algebras} \ar@{->}[rr]|!{[r];[r]}\hole \ar@{->}[dl] %
&& \mathrm{homotopy} \  \mathrm{operads} \ar@{->}[dl]\\
\mathcal{A}_{\infty}\textrm{-}\mathrm{categories}\ar@{->}[rr] && 
\begin{minipage}[c]{4cm} 
\bf strict unital homotopy 
colored operads
\end{minipage} 
}
\end{equation*}
The second goal is to extend the nerve of Faonte--Lurie to the operadic level, study this extension and compare it with other constructions which appear in the literature. More precisely, we build a nerve functor $\mathrm{N}^{\Omega}$ from the category of strict unital homotopy colored operads to dendroidal sets, such that for any $\mathcal{A}_{\infty}$-category $A$, the simplicial set induced by the dendroidal set $\mathrm{N}^{\Omega}(A)$ is equal to $\mathrm{N}_{\mathcal{A}_{\infty}}(A)$.

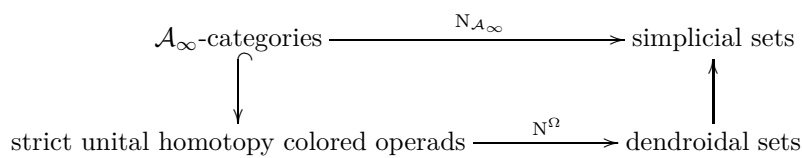
\begin{equation*}
\xymatrix{\mathcal{A}_{\infty}\textrm{-}\mathrm{categories} \ar@{^{(}->}[d] \ar[rr]^(0.51){\mathrm{N}_{\mathcal{A}_{\infty}}} && \mathrm{simplicial \ sets} \\
\text{strict unital homotopy colored operads} \ar[rr]^(0.65){\mathrm{N}^{\Omega}} && \text{dendroidal sets} \ar[u]}
\end{equation*}

This \textit{dendroidal nerve} is defined as follows. For any tree $T$, let us denote by $\mathbb{K}\Omega(T)$ the canonical algebraic colored operad induced by $T$ and consider it as a strict unital homotopy colored operad. The dendroidal nerve $\mathrm{N}^{\Omega}(\mathscr{P})$ of a strict unital homotopy colored operad $\mathscr{P}$ is defined by
$$
\mathrm{N}^{\Omega}(\mathscr{P})_T:=\mathrm{Hom}_{\mathsf{suOp}_{\infty}}(\mathbb{K}\Omega(T),\mathscr{P}) \ ,
$$
where $\mathsf{suOp}_{\infty}$ is the category of strict unital homotopy colored operads, with their morphisms, sometimes called $\infty$-morphisms. The following theorem gives the first comparison statement between the two worlds of homotopy operads and $\infty$-operads.

\begin{theointro1}
The dendroidal nerve of a strict unital homotopy colored operad is an $\infty$-operad.
\end{theointro1}

The dendroidal nerve $\mathrm{N}^{\Omega}$ generalizes the homotopy coherent nerve $\mathrm{hcN}$ of colored operads on chain complexes built by Ieke Moerdijk and Ittay Weiss in \cite{MoerdijkWeiss07}.

\begin{theointro2}
There is a canonical isomorphism
$$
\mathrm{hcN}(\mathscr{P}) \simeq \mathrm{N}^{\Omega}(\mathscr{P}) \ ,
$$
which is natural in dg colored operads $\mathscr{P}$.
\end{theointro2}

We prove that the dendroidal nerve functor satisfies some nice homotopy properties. For instance, if we endow the category of dendroidal sets with the Cisinski--Moerdijk model structure introduced in \cite{CisinskiMoerdijk11}, we can characterize the morphisms of strict unital homotopy colored operads whose image under the nerve $\mathrm{N}^{\Omega}$ are weak equivalences (resp. fibrations). As a consequence, if we consider the category of colored operads on chain complexes over a field of characteristic zero with the model structure introduced in \cite{Caviglia14}, then we can show that the homotopy coherent nerve $\mathrm{hcN}$ is a right Quillen functor.

\begin{theointro3}
There is a Quillen adjunction
$$
\xymatrix{\mathsf{dgOp} \ar@<1ex>[r]^{\mathrm{hcN}} & \mathsf{dSet} \ar@<1ex>[l]^{W_!^{dg}} \ .}
$$
\end{theointro3}

Finally, from a category enriched in chain complexes, one can truncate the mapping spaces and apply to them the Dold--Kan construction to obtain a simplicial category. Then, one can apply the simplicial nerve of simplicial categories and get a quasi-category. This procedure called the big nerve is introduced by Lurie in \cite[Section 1.3.1]{Lurie12}. He shows that this nerve is equivalent to the simplicial nerve $\mathrm{N}_{\mathcal{A}_{\infty}}$. The big nerve and the latter result can be extended to the operadic level.

\begin{theointro4}
There exists a transformation of functors $\alpha^*$ from the big nerve of dg operads to the homotopy coherent nerve such that, for any dg colored operad $\mathscr P$, the morphism $\alpha^*(\mathscr P): \mathrm{N}^{\mathrm{big}}_{\mathrm{dg}}(\mathscr P) \rightarrow \mathrm{hcN}(\mathscr P)$ is a weak equivalence of dendroidal sets.
\end{theointro4}

\paragraph*{\bf Layout}
In the first section of this paper, we recall the definitions of colored operads, dendroidal sets, and $\infty$-operads; we provide the reader with more details on the relations between trees and operads. 
In the second section, we define colored generalizations of the notions of cooperads, homotopy operads and strict unital homotopy operads.
The third section is the main part of this paper. There, we introduce the dendroidal nerve which goes from the category of strict unital colored homotopy operads to the category of dendroidal sets. We prove that its image actually lands in the category of $\infty$-operads and we show that its restriction to dg operads is equal to the homotopy coherent nerve of Moerdijk--Weiss. Then we investigate its homotopical behavior and show that the homotopy coherent nerve is a right Quillen functor.
In the fourth section, we introduce the big nerve of dg colored operads which extends the big nerve of dg categories of Lurie. Finally, we show that it is pointwise equivalent to the homotopy coherent nerve.
In the appendix, we prove a equivalence between two notions of operations spaces of $\infty$-operads.\\

\paragraph*{\bf Acknowledgements} I would like to thank Damien Calaque for leading me to the idea of this paper, Clemens Berger for sharing his clear approach to colored operads and my advisor Bruno Vallette for his precious advice and careful review of this paper. I am very grateful to Ieke Moerdijk for sharing with me some of his ideas about the homotopy theory of dendroidal sets.\\

\paragraph*{\bf Conventions}
\begin{itemize}
\item[$\triangleright$] We work over a unital commutative ring $\mathbb{K}$.

\item[$\triangleright$] We denote by $\mathsf{Set}$ the category of sets, $\Delta$ the category of finite sets $\{0,\ldots,n\}$, for $n \geq 0$, with monotonic functions and $\mathsf{sSet}$ the category of simplicial sets.

\item[$\triangleright$] We denote by $\mathsf{gr}\textsf{-}\mathsf{Mod}$ the category of graded $\mathbb{K}$-modules, $\mathsf{dg}\textsf{-}\mathsf{Mod}$ the category of chain complexes of $\mathbb{K}$-modules, and $\mathsf{dg}\textsf{-}\mathsf{Mod}^{\geq 0}$ the category of nonnegatively graded chain complexes of $\mathbb{K}$-modules. These three categories are endowed with their canonical symmetric monoidal structure.

\item[$\triangleright$] For any chain complex $V$, the \textit{suspension} of $V$ is the shifted chain complex $sV$ such that $(sV)_n:=V_{n-1}$ and $d_{sV}(sv):=-sd_V(v)$. 

\item[$\triangleright$] If $x$ is a homogeneous element of a graded $\mathbb{K}$-module, then its degree is denoted by $|x|$.

\item[$\triangleright$] Let $\mathsf{gr}\textsf{-}\mathsf{Mod}^{\mathsf{deg}}$ be the category of graded $\mathbb{K}$-modules with graded morphisms.


\item[$\triangleright$]  Finally, we let $[n]$ denote the poset $0<1<\cdots <n$ and $\underline{n}$ the set $\{1,\ldots,n\}$.
\end{itemize}



\section{Recollections on colored operads and dendroidal sets}

In this section, we recall the notions of colored operads, dendroidal sets and $\infty$-operads. These concepts are intimately related to the properties of trees. First, we make more precise the appendix of \cite{BergerMoerdijk07} defining colored operads as monoids. This clear presentation will allow us to introduce the relevant new operadic notions in the next section. 
Then, we give a short survey on dendroidal sets and $\infty$-operads together with their homotopical properties after the original references \cite{MoerdijkWeiss07}, \cite{CisinskiMoerdijk11} and \cite{CisinskiMoerdijk13a}.

\subsection{Colored operads as monoids}
We give a monoidal definition of colored operads over a symmetric monoidal category $(\mathsf{E}, \otimes , \mathds{1}_{\mathsf{E}})$ with colimits preserved by the monoidal product. This notion can be found in the appendix of the paper \cite{BergerMoerdijk07} by Berger--Moerdijk. We start working over a fixed \textit{set of colors} $C$.

\begin{defin}[The groupoid $\BijC$]
Let $\BijC$ be the category whose objects are pairs $(\chi:X \to C ; c)$, where $c$ is a color in ${C}$ and $\chi$ is a function from a finite set $X$ to the set $C$. A morphism from $(c;\chi:X \rightarrow C)$ to $(\chi':Y \rightarrow C ; c)$ consists of a a bijection $\beta: X \rightarrow Y$ such that the following diagram commutes 
$$
\xymatrix{X \ar[rd]_{\chi} \ar[rr]^{\beta} && Y \ar[ld]^{\chi'} \\
&  C \ .
}
$$
There are no morphisms between objects $(\chi ; c)$ and $(\chi' ; c')$ when $c \neq c'$.
\end{defin}

\begin{defin}[$(C,\mathbb{S})${-module}] 
A $(C,\mathbb{S})$\textit{-module} is an $\mathsf{E}$-presheaf on $\mathsf{Bij}_{C}$, i.e. a functor from  the category $\mathsf{Bij}_{C}^{op}$ to $\mathsf{E}$. The category of $(C,\mathbb{S})$-modules is the category of $\mathsf{E}$-presheaves on $\mathsf{Bij}_{C}$. We denote it by $(C,\mathbb{S})\textsf{-}\mathsf{Mod}$.
\end{defin}

\paragraph{\sc Notation} Let $\mathcal{V}$ be a $(C,\mathbb{S})$-module, let $c$, $c_1$, \ldots, $c_m$ be elements of $C$ and let $\phi$ the function from $\underline{n}$ to $C$ which sends $i$ to $c_i$. We will sometimes write $\mathcal{V}(c_1, \ldots,c_m ; c)$ to denote $\mathcal{V}(\phi ; c)$.\\

One can interpret a $(C,\mathbb{S})$-module as a collection of operations with one output and zero or many inputs labeled by colors. For example, let $\mathcal V$ be a $(C,\mathbb{S})$-module, let $\chi:X \rightarrow C$ be a function and let $c$ be an element of $C$. Then $\mathcal V(\chi ; c)$ represents the operations whose inputs are the elements $x \in X$  colored by $\chi(x)$ and  whose output is colored by $c$. To compose operations of two $(C,\mathbb{S})$-modules with respect to the colors, we introduce the following  {composite product}.

\begin{defin}[Composite product] Let $\mathcal V$ and $\mathcal W$ be two $(C,\mathbb{S})$-modules. Their \textit{composite product} is the $(C,\mathbb{S})$-module $\mathcal V \circ \mathcal W$ which sends every object $(\chi: X \rightarrow C ; c)$ of the category $\mathsf{Bij}_{C}$ to the following colimit: $\mathcal V \circ\mathcal  W(\chi:X \rightarrow C ; c) := $
$$
 \coprod_{k \geq 1} \left( \coprod_{\psi:\, \underline{k}  \rightarrow C} \  \coprod_{\alpha:\, X \rightarrow \underline{k}} \mathcal V(\psi ; c) \otimes \mathcal W(\chi_{|\alpha^{-1}(1)} ; \psi(1)) \otimes \cdots \otimes\mathcal  W(\chi_{|\alpha^{-1}(k)} ; \psi(k)) \right) _{\mathbb{S}_k}\ , 
$$
\noindent where the second coproduct is taken over all  functions $\psi:\underline{k} \rightarrow C$ and where the third coproduct is taken over all functions $\alpha: X \rightarrow \underline{k}$. This colimit is a coset under the following actions of the symmetric groups $\mathbb{S}_k$, for integers $k \geq 1$. 
\begin{itemize}
\item[$\triangleright$] A permutation $\sigma$ in $\mathbb{S}_k$ induces an isomorphism $(\psi \sigma^{-1} ; c) \rightarrow (\psi ; c)$ in the groupoid $\mathsf{Bij}_{C}$ and so an isomorphism $\mathcal V(\sigma): \mathcal V(\psi ; c) \rightarrow  \mathcal V(\psi \sigma^{-1} ; c)$ in $\mathsf{E}$.

\item[$\triangleright$] This permutation also induces an isomorphism $\sigma^*$ from $\mathcal W(\phi_{|\alpha^{-1}(1)} ; \psi(1)) \otimes \cdots  \allowbreak \otimes \allowbreak  \mathcal{W}(\allowbreak \phi_{|\alpha^{-1}(k)} ; \psi(k)) $ to $  \mathcal W(\phi_{|\alpha^{-1}(\sigma^{-1}(1))} ; \psi(\sigma^{-1}(1))) \otimes \cdots \otimes \mathcal W(\phi_{|\alpha^{-1}(\sigma^{-1}(k))} ; \psi(\sigma^{-1}(k)))$ through the symmetric monoidal structure of $\mathsf{E}$.
\end{itemize}

\noindent The global action of $\mathbb{S}_k$ is given by the isomorphisms $V(\sigma) \otimes \sigma^*$, 
 for every permutation $\sigma \in \mathbb{S}_k$.
\end{defin}

Let $I_{C}$ be the $(C,\mathbb{S})$-module defined by $I_{C}(c;c):=I_{C}(* \mapsto c ; c)= \mathds{1}_{\mathsf{E}}$ , for any $c\in C$, and by $I_{C}(\chi:X \rightarrow C ; c)= \emptyset$ (the initial object in the category $\mathsf{E}$) if the cardinal of $X$ is not $1$ or if the image of $\chi$ is different from the color $\{c\}$ of the output.

\begin{prop}
The category of $(C,\mathbb{S})$-modules together with the composite product $\circ$ and the unit object $I_{C}$ forms a monoidal category.
\end{prop}

\begin{proof}
The proof is similar to the classical case of non-colored operads, see \cite[Section~5.1]{LodayVallette12}.
\end{proof}

\begin{defin}[$C$-colored operad]
A $C$-\textit{colored operad}  is a monoid  $(\PP,\gamma, \eta)$ in the monoidal category of $(C,\mathbb{S})$-modules: the composition map $\gamma: \PP \circ \PP \rightarrow \PP$ is associative and the map $\eta : I_{C} \rightarrow \PP$ is a unit.
\end{defin}

\subsection{Morphisms of colored operads}  We provide here a detailed definition of the suitable notion of morphism between two operads colored over possibly different sets of colors. From now on, we consider the set of colors to be part of the data and we work over varying sets of colors. 

\begin{defin}[Colored $\mathbb S$-module]\label{defcolor}
A \textit{colored} $\mathbb{S}$\textit{-module} $(C,\mathcal{V})$ is made up of  a set $C$ of colors and a $(C,\mathbb{S})$-module $\mathcal V$.
\end{defin}

We will define morphisms of $(C,\mathbb{S})$-modules in an analogous way as morphisms of presheaves on topological spaces are defined. Let us recall that an $\mathsf{E}$-presheaf on a topological space $X$ is a $\mathsf{E}$-presheaf over the category $\mathsf{Open}_X$ of open subsets of $X$ with inclusions. Let $X$ and $Y$ be topological spaces and  let $\mathcal{F}$ and $\mathcal{G}$ be presheaves respectively on $X$ and $Y$.  Any continuous function $f:X \rightarrow Y$ induces a functor $F$ from $\mathsf{Open}_Y$ to $\mathsf{Open}_X$. The presheaf $f^{-1}\mathcal{G}$ on $X$ is defined by $f^{-1}\mathcal{G}(U):=\lim_{f(U)\subset V} \mathcal{G}(V)$ and the presheaf $f_*\mathcal{F}$ on $Y$ by $f_*\mathcal{F}(U):=\mathcal{F}(f^{-1}(U)) =\mathcal{F}(F(U))$. Furthermore, a morphism of presheaves on topological spaces from $(X,\mathcal{F})$ to $(Y,\mathcal{G})$ is the data of a continuous function $f$ from $X$ to $Y$, and hence a functor $F$ from $\mathsf{Open}_Y$ to $\mathsf{Open}_X$, together with a morphism of presheaves over $\mathsf{Open}_Y$ from $\mathcal{G}$  to $f_*\mathcal{F} $, or equivalently, a morphism of presheaves over $\mathsf{Open}_X$ from $f^{-1}\mathcal{G}$ to $\mathcal{F}$.

\begin{defin}[Pullback and pushforward of colored $\mathbb S$-module]
Let $\phi:C \rightarrow {D}$ be a function between two sets of colors and let $\mathcal V$ and  $\mathcal W$ be a $(C,\mathbb{S})$-module and a $({D},\mathbb{S})$-module respectively. We define $\phi^*\mathcal W$ to be the following $(C,\mathbb{S})$-module 
$$
\phi^*\mathcal W(\chi ; c):=\mathcal W(\phi \chi ; \phi(c))
$$
and we define $\phi_!\mathcal V$ to be the following $({D},\mathbb{S})$-module
$$
\phi_!\mathcal V(\rho ; d):=\coprod_{\rho=\phi \chi , \atop \phi(c)=d} \mathcal V(\chi ; c) \ ,
$$
for any  $\rho:X \rightarrow {D}$ and $d\in {D}$. The coproduct is taken over the colors $c$ in $C$ such that $\phi(c)=d$ and the functions $\chi: X \rightarrow C$ such that $\rho= \phi \chi$. 
\end{defin}

These two constructions are functorial.

\begin{lemma}
For any function $\phi:C \rightarrow {D}$, the functor $\phi^*$ is right adjoint to the functor $\phi_!$; equivalently there exist natural bijections 
$$
\mathrm{Hom}_{({D},\mathbb{S})\textsf{-}\mathsf{Mod}}(\phi_!\mathcal V, \mathcal W) \cong \mathrm{Hom}_{(C,\mathbb{S})\textsf{-}\mathsf{Mod}}(\mathcal V,\phi^*\mathcal W) \ .
$$
\end{lemma}

\begin{proof}
The proof is straightforward and left to the reader. 
\end{proof}

These functors behave well with respect to the composition of functions: for any functions $\phi: {B} \rightarrow C$ and $\psi: C \rightarrow {D}$, $(\psi \phi)^*=\phi^*  \psi^*$ and $(\psi \phi)_!=\psi_!  \phi_!$.

\begin{defin}[Morphism of colored $\mathbb{S}$-modules]\label{coloredsmod}
A \textit{morphism of colored $\mathbb{S}$-modules}  from $(C,\mathcal V) $ to $({D}, \mathcal W)$ amounts to  the data of  a function $\phi$ from $C$ to ${D}$ and a morphism $f^*$ of $(C,\mathbb{S})$-modules from $\mathcal V$ to $\phi^* \mathcal W$, or equivalently, a morphism $f_!$ of $({D},\mathbb{S})$-modules from $\phi_! \mathcal V$ to $\mathcal W$. Such a morphism, defined either by the couple $(\phi,f^*)$ or by the couple $(\phi,f_!)$, will be denoted simply by $f$. 
\end{defin}

A morphism $f:(C,\mathcal V) \rightarrow ({D},\mathcal W)$ of colored $\mathbb{S}$-modules is therefore the data of a function $\phi:C \rightarrow D$ and morphisms $f^*(\chi ; c): \mathcal V(\chi ; c) \rightarrow  \mathcal W(\phi  \chi ; \phi(c))$ for any object $(\chi:X \rightarrow C ; c)$ of $\BijC$ such that, for any bijection $\beta: Y \rightarrow X$, the following diagram commutes.
$$
\xymatrix@C=40pt{\mathcal V(c; \chi) \ar[r]^(0.42){f^*(\chi ; c)} \ar[d]_{\mathcal V(\beta)} &\mathcal W(\phi \chi ; \phi(c)) \ar[d]^{\mathcal W(\beta)}\ \  \\
\mathcal V(\chi \beta ; c) \ar[r]_(0.42){f^*(\chi \beta ; c)} &\mathcal  W(\phi \chi \beta ; \phi(c))\ .
}
$$

\begin{prop}
Colored $\mathbb{S}$-modules and their morphisms form a category, denoted by  $\mathbb{S}\textsf{-}\mathsf{Mod}$.
\end{prop}

\begin{proof}
One defines the composite of two morphisms $f=(\phi,f^*)$ and $g=(\psi, g^*)$ by 
$$g f:=(\psi \phi, \phi^*(g^*) f^*) \ .$$
\end{proof}

\begin{rmk}
For any set $C$, there is an inclusion of categories $(C,\mathbb{S})\textsf{-}\mathsf{Mod} \hookrightarrow \mathbb{S}\textsf{-}\mathsf{Mod}$ which sends a $(C,\mathbb{S})$-module $\mathcal V$ to the $\mathbb{S}$-module $(C,\mathcal  V)$. A morphism $f=(\phi,f^*)$ is in the image of this inclusion if and only if $\phi$ is the identity of the set $C$.
\end{rmk}

\begin{lemma}
For any function $\phi: C \rightarrow {D}$, the functor $\phi^*$ is lax monoidal, i.e. there are morphisms $\phi_{\mathcal V,\mathcal  W}^*:\phi^*\mathcal V \circ \phi^*\mathcal  W \rightarrow \phi^*(\mathcal V \circ \mathcal W)$, natural in $\mathcal V$ and $\mathcal W$ and $\phi_I^*: I_{C} \rightarrow \phi^*I_{{D}}$, satisfying associativity and unitality conditions, see \cite{MacLane98} for more details.
\end{lemma}

\begin{proof}
For any $({D},\mathbb{S})$-modules $\mathcal V$ and $\mathcal W$, the morphism $\phi_{\mathcal V,\mathcal W}^*$ is built from the following equality:
\begin{align*}
& \phi^*(\mathcal V)(\psi ; c) \otimes\phi^*(\mathcal W)(\chi_{|\alpha^{-1}(1)} ; \psi(1)) \otimes \cdots \otimes \phi^*(\mathcal W)(\chi_{|\alpha^{-1}(k)} ; \psi(k))=\\
& \mathcal V(\phi \psi ; \phi(c)) \otimes \mathcal W(\phi \chi_{|\alpha^{-1}(1)} ; \phi(\psi(1))) \otimes \cdots \otimes \mathcal W(\phi \chi_{|\alpha^{-1}(k)} ; \phi(\psi(k)))\ .
\end{align*}
\end{proof}

\begin{rmk}
The monoidal functor $\phi^*$ is strong if and only if the map $\phi$ is bijective. 
\end{rmk}

\begin{prop}
Let $\phi:C \rightarrow {D}$ be a function, and let $(\PP,\gamma,\eta)$ be a ${D}$-colored operad. The $(C,\mathbb{S})$-module $\phi^*\PP$ has a canonical structure of $C$-operad $(\phi^*\PP,\gamma^{\phi},\eta^{\phi})$ induced by the structure of ${D}$-operad on $\PP$.
\end{prop}

\begin{proof}
This is a corollary of the previous lemma since the image of a monoid under a lax monoidal functor is again a monoid. 
\end{proof}

\begin{defin}[The category of colored operads]\label{newdef}\leavevmode
\begin{itemize}
\item[$\triangleright$]
A \textit{colored operad} $\mathscr{P}=(C,\PP,\gamma, \eta)$  is the data of a set of colors $C$ and a $C$-colored operad $(\PP,\gamma, \eta)$.

\item[$\triangleright$]A \textit{morphism of colored operads} from $\mathscr{P}=(C,\PP,\gamma,\eta)$ to $\mathscr{Q}=({D},\mathcal Q,\nu,\theta)$ is a morphism of colored $\mathbb{S}$-modules $f=(\phi,f^*)$ such that $f^*$ is a morphism of $C$-operads from $(\PP,\gamma,\eta)$ to $(\phi^*\mathcal Q,\nu^{\phi},\theta^{\phi})$, i.e. a morphism of monoids in the monoidal category of $(C,\mathbb{S})$-modules.
\end{itemize}
\end{defin}

\begin{prop}\label{prop:CategoryOp}
Colored operads together with their morphisms form a category denoted $\mathsf{Op}$. 
 \end{prop}

\begin{proof} We first prove that the composite of two morphisms of colored operads is a morphism of colored operads. 
Let $\mathscr{P}=(C,\PP,\gamma,\eta)$, $\mathscr{Q}=({D},\mathcal Q,\nu,\theta)$ and $\mathscr{R}=(E, \mathcal R, \mu, \upsilon)$ be three colored operads, and let $f=(\phi,f^*) : \mathscr{P} \to \mathscr{Q}$ and $g=(\psi,g^*) : \mathscr{Q} \to \mathscr{R}$ be two morphisms of colored operads. Their composite is equal to 
$g f:=(\psi \phi, \phi^*(g^*) f^*)$. Since the functor $\phi^*$ is a lax monoidal functor, it preserves morphisms of monoids. So the composite $\phi^*(g^*) f^*$ is a morphism of $(C,\Sy)$-operads. 
The unit morphisms for the composite of morphisms of colored operads are the unit morphisms of the category of colored $\mathbb{S}$-modules.
\end{proof}

\begin{defin}[Restriction of colored operads to categories]
Let $C$ be a set of colors and let $j: \mathsf{Bij}_C(1) \hookrightarrow \mathsf{Bij}_C$ be the embedding of the full sub-category of $\mathsf{Bij}_C$ made up of objects $(\chi:X \rightarrow C ; c)$ where $X$ has one elements. It induces a restriction functor $j^*$ from the category of $C$-colored operads to the category of categories enriched over $\mathsf{E}$ and whose set of objects is $C$. It extends canonically to a functor $j^*: \mathsf{Op} \rightarrow \mathsf{Cat}_{\mathsf{E}}$ from colored operads enriched over $\mathsf{E}$ to small categories enriched over $\mathsf{E}$.
\end{defin}

\paragraph{\sc Convention} Since we will only work with colored operads throughout the text, we will often use  the simpler terminology of $\Sy$-modules and operads, i.e. dropping the understood adjective ``colored''.

\subsection{The category of trees}

The theory of operads is intrinsically related to the combinatorics of trees. So we begin this section with a precise definition of the notion of tree used here. The formalism of trees will be used in the next section to introduce the concept of dendroidal set \cite{MoerdijkWeiss07}, which is an operadic generalization of the concept of simplicial set. Finally, we will recall the definition of an $\infty$-operad, which is to dendroidal sets what $\infty$-categories are to simplicial sets, that is a weak Kan object. 

\begin{defin}[Graph]
A \textit{graph} is a quadruple $G=({V}, {F}, \upsilon, \rho)$ where ${V}$ is a finite set of elements called the \textit{vertices}, ${F}$ is a finite set of elements called the \textit{flags} or \textit{half-edges}, $\upsilon$ is a function from the set ${F}$ of flags to the set ${V}$ of vertices and $\rho$ is an involution of the set ${F}$. The orbits of this involution are called \textit{edges}. An edge is \textit{inner} if it contains two flags and  \textit{outer} otherwise. 
\end{defin}

\begin{eg}
For instance, two vertices linked by an edge is just a set of vertices with two objects $\{v_1,v_2\}$, a set of flags with two objects $\{f_1,f_2\}$, a function $\upsilon$ such that $\upsilon(f_i)=v_i$ and an involution $\rho$ such that $\rho(f_1)=f_2$. 
\end{eg}

\begin{defin}[Tree]
A \textit{(rooted) tree} $T=({V},{F}, \upsilon,\rho, r)$ is  a connected graph $({V},{F},\upsilon,\rho)$ with no cycles and a distinguished outer edge $r$ called the \textit{root}. The remaining outer edges are called \textit{leaves}. In this context, each vertex has one output edge and possibly many input edges (there can be no input edge). The number of inputs of a vertex is called its \textit{arity}.
\end{defin}
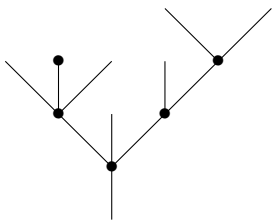
\begin{figure}[h]
\begin{tikzpicture}[scale=0.7]

\draw (2,1) node {$\bullet$} ;
\draw (1,2) node {$\bullet$} ;
\draw (3,2) node {$\bullet$} ;
\draw (4,3) node {$\bullet$} ;
\draw (1,3) node {$\bullet$} ;

\draw (2,0) -- (2,1) ;

\draw (2,1) -- (1,2) ;
\draw (2,1) -- (2,2) ;
\draw (2,1) -- (3,2) ;

\draw (1,2) -- (0,3) ;
\draw (1,2) -- (1,3) ;
\draw (1,2) -- (2,3) ;

\draw (3,2) -- (3,3) ;
\draw (3,2) -- (4,3) ;

\draw (4,3) -- (3,4) ;
\draw (4,3) -- (5,4) ;
\end{tikzpicture}
\caption{Example of a tree.}
\end{figure}

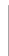
\begin{figure}[h]
\begin{tikzpicture}[scale=0.7]
\draw (0,0) -- (0,1) ;
\end{tikzpicture}
\caption{The \textit{trivial} tree with no vertex but one edge.}
\end{figure}

\begin{defin}[Sub-tree]
Let $T=({V},{F},\upsilon,\rho,r)$ be a tree. A \textit{non trivial sub-tree} $T'$ of $T$ is a non-empty subset ${V}'$ of ${V}$ which is connected, i.e. for any two vertices $v_1$ and $v_2$ of $V'$, there exists a path between them in the tree $T$ which only visits vertices in ${V}'$. This subset determines a new tree also denoted $T'$ which is the $5$-tuple $({V}',{F}',\upsilon',\rho',r')$ where $F'$ is the set of flags $f$ of ${F}$ such that $\upsilon(f)$ is in $V'$, the function $\upsilon'$ is the restriction of $\upsilon$ to the set ${F}'$ and for any flag $f$ in ${F}'$, then $\rho'(f)=\rho(f)$ if $\rho(f) \in {F}'$ and $\rho(f)=f$ otherwise. The root $r'$ is the edge of $T'$ which is the flag closest to the root $r$ of $T$.
\end{defin}

\begin{defin}[Partition of a tree]
Let $T=({V},{F},u,\rho,r)$ be a tree. A \textit{partition} of $T$ with no trivial component is the data of non-trivial sub-trees $T_1, \ldots, T_k$  with no common vertices such that their union contains every vertex of $T$. Moreover, we denote by $T/T_1\cdots T_n$ the tree obtained from $T$ by contracting into one vertex each sub-tree $T_i$.
\end{defin}

\begin{figure}[h]
\begin{tikzpicture}[scale=0.7]

\draw (2,1) node {$\bullet$} ;
\draw (1,2) node {$\bullet$} ;
\draw (3,2) node {$\bullet$} ;
\draw (4,3) node {$\bullet$} ;
\draw (1,3) node {$\bullet$} ;

\draw (2,0) -- (2,1) ;

\draw (2,1) -- (1,2) ;
\draw (2,1) -- (2,2) ;
\draw (2,1) -- (3,2) ;

\draw (1,2) -- (0,3) ;
\draw (1,2) -- (1,3) ;
\draw (1,2) -- (2,3) ;

\draw (3,2) -- (3,3) ;
\draw (3,2) -- (4,3) ;

\draw (4,3) -- (3,4) ;
\draw (4,3) -- (5,4) ;

\draw (9,1) node {$\bullet$} ;
\draw (8,2) node {$\bullet$} ;
\draw (10,2) node {$\bullet$} ;

\draw (9,0) -- (9,1) ;

\draw (9,1) -- (8,2) ;
\draw (9,1) -- (9,2) ;
\draw (9,1) -- (10,2) ;

\draw (8,2) -- (7.5,3) ;
\draw (8,2) -- (8.5,3) ;

\draw (10,2) -- (9.5,3) ;
\draw (10,2) -- (10,3) ;
\draw (10,2) -- (10.5,3) ;

\draw (2,0.9) circle (1);
\draw (1,2.9) circle (1.1);
\draw (3.8,3) circle (1.5);

\draw (2,5) node {$T=T_1 \sqcup T_2 \sqcup T_3$};
\draw (9,5) node {$T/T_1T_2T_3$};
\end{tikzpicture}
\caption{An example of a partitioned tree and its associated contraction.}
\end{figure}
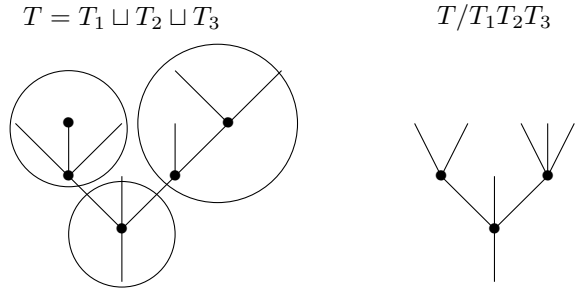

Following  Moerdijk--Weiss \cite{MoerdijkWeiss07}, we can consider the following set theoretical colored operad $\Omega(T)$ generated by any  tree $T$.

\begin{defin}[The operad $\Omega(T)$]\label{deftreeop}
For any tree $T$, $\Omega(T)$ is the operad colored by the set of edges of $T$, freely generated by the set of vertices of $T$. In details, $\Omega(T)(c;\chi : X \to C):=\{*\}$ if there is a (possibly trivial) sub-tree of $T$ with output $c$ and inputs $\chi(x_1), \ldots, \chi(x_n)$. Otherwise $\Omega(T)(c;\chi : X \to C):=\emptyset$. The composite map is given by the grafting of sub-trees inside $T$.
\end{defin}

\begin{defin}[The category of trees]
The category of trees, written $\mathsf{Tree}$ is made up of trees, as defined above; the morphisms from a tree $T$ to a tree $T'$ are the morphisms of operads from $\Omega(T)$ to $\Omega(T')$, i.e.
$$
\mathrm{Hom}_{\mathsf{Tree}}(T,T'):=\mathrm{Hom}_{\mathsf{Op}}(\Omega(T), \Omega(T'))\ .
$$
\end{defin}

Here are three families of simple morphisms of trees.  An \textit{outer coface} $\delta_v$ adds a new external vertex $v$, i.e. a vertex attached to at most one other vertex. An \textit{inner coface} $\delta_e$ introduces an inner edge $e$. A \textit{codegeneracy} $\sigma$ erase an arity $1$ vertex. The inner and outer cofaces, the codegeneracies and the isomorphisms generate all the morphisms of trees. More details can be found in the paper \cite{MoerdijkWeiss07}.

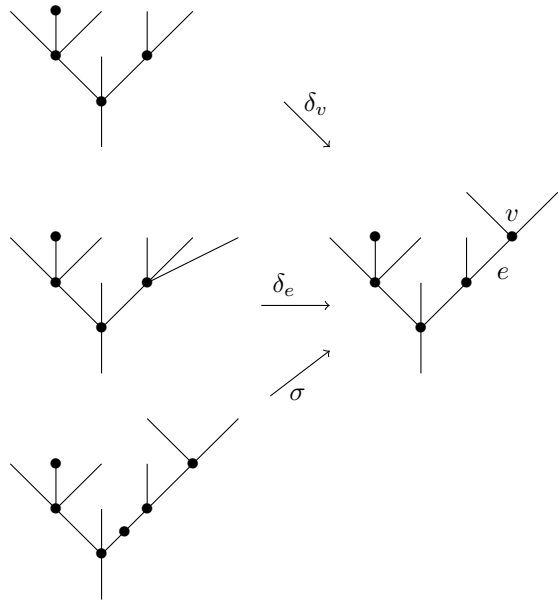
\begin{figure}[h]
\begin{tikzpicture}[scale=0.6]

\draw[->] (6,11)--(7,10);
\draw (6.7,10.5) node[above] {$\delta_v$} ;

\draw (2,11) node {$\bullet$} ;
\draw (1,12) node {$\bullet$} ;
\draw (3,12) node {$\bullet$} ;

\draw (1,13) node {$\bullet$} ;

\draw (2,10) -- (2,11) ;

\draw (2,11) -- (1,12) ;
\draw (2,11) -- (2,12) ;
\draw (2,11) -- (3,12) ;

\draw (1,12) -- (0,13) ;
\draw (1,12) -- (1,13) ;
\draw (1,12) -- (2,13) ;

\draw (3,12) -- (3,13) ;
\draw (3,12) -- (4,13) ;


\draw[->] (5.5,6.5)--(7,6.5);
\draw (6,6.5) node[above] {$\delta_e$} ;

\draw (2,6) node {$\bullet$} ;
\draw (1,7) node {$\bullet$} ;
\draw (3,7) node {$\bullet$} ;
\draw (1,8) node {$\bullet$} ;

\draw (2,5) -- (2,6) ;

\draw (2,6) -- (1,7) ;
\draw (2,6) -- (2,7) ;
\draw (2,6) -- (3,7) ;

\draw (1,7) -- (0,8) ;
\draw (1,7) -- (1,8) ;
\draw (1,7) -- (2,8) ;

\draw (3,7) -- (3,8) ;
\draw (3,7) -- (4,8) ;

\draw (3,7) -- (5,8) ;


\draw[->] (5.7,4.5)--(7,5.5);
\draw (6.3,4.2) node[above] {$\sigma$} ;

\draw (2,1) node {$\bullet$} ;
\draw (1,2) node {$\bullet$} ;
\draw (3,2) node {$\bullet$} ;
\draw (4,3) node {$\bullet$} ;
\draw (1,3) node {$\bullet$} ;
\draw (2.5,1.5) node {$\bullet$} ;

\draw (2,0) -- (2,1) ;

\draw (2,1) -- (1,2) ;
\draw (2,1) -- (2,2) ;
\draw (2,1) -- (3,2) ;

\draw (1,2) -- (0,3) ;
\draw (1,2) -- (1,3) ;
\draw (1,2) -- (2,3) ;

\draw (3,2) -- (4,3) ;
\draw (3,2) -- (3,3) ;

\draw (4,3) -- (3,4) ;
\draw (4,3) -- (5,4) ;


\draw (9,6) node {$\bullet$} ;
\draw (8,7) node {$\bullet$} ;
\draw (10,7) node {$\bullet$} ;
\draw (11,8) node {$\bullet$} ;
\draw (8,8) node {$\bullet$} ;

\draw (11,8.5) node {$v$};
\draw (10.8,7.2) node {$e$} ;

\draw (9,5) -- (9,6) ;

\draw (9,6) -- (8,7) ;
\draw (9,6) -- (9,7) ;
\draw (9,6) -- (10,7) ;

\draw (8,7) -- (7,8) ;
\draw (8,7) -- (8,8) ;
\draw (8,7) -- (9,8) ;

\draw (10,7) -- (10,8) ;
\draw (10,7) -- (11,8) ;

\draw (11,8) -- (10,9) ;
\draw (11,8) -- (12,9) ;
\end{tikzpicture}
\caption{Examples of an outer coface, an inner coface and a codegeneracy}
\end{figure}

\subsection{Dendroidal sets} We introduce here the concept of dendroidal set due to Moerdijk--Weiss \cite{MoerdijkWeiss07}, which is an operadic generalization of the concept of simplicial set.

\begin{defin}[Dendroidal set]
A \textit{dendroidal set} is presheaf on the category $\Tree$, i.e. it is a functor from the category $\mathsf{Tree}^{op}$ to the category $\mathsf{Set}$ of sets. We denote the category of dendroidal sets by $\mathsf{dSet}$. By analogy with simplices of simplicial sets, we call \textit{dendrices} of a dendroidal set $D$ the elements of the image sets $D_T$, for any tree $T$.
\end{defin}

Let $D$ be a dendroidal set. Any outer coface $\delta_v: T-\{v\} \rightarrow T$, any inner coface $\delta_e: T/e \rightarrow T$ and any codegeneracy $\sigma :T  \rightarrow T'$ give respectively an outer face $d_v: D_{T} \rightarrow D_{T-\{v\}}$, an inner face $d_e: D_T \rightarrow D_{T/e}$ and a degeneracy $s: D_{T'} \rightarrow D_T$.\\

We can consider the poset $[n]:=0<1<\cdots<n$ as the linear tree (ladder) with $n$ vertices and $n+1$ edges labeled by the set $\{0,\ldots ,n\}$ from bottom to top. Then we can consider the category $\Delta$ made up of sets $\{0, \ldots ,n\}$ and non-decreasing functions as a full sub-category of the category \textsf{Tree}. Let us denote $i:\Delta \rightarrow \mathsf{Tree}$ this embedding. It induces a restriction functor $i^*:\mathsf{dSet} \rightarrow \mathsf{sSet}$ which sends a dendroidal set $D:\mathsf{Tree}^{op} \rightarrow \mathsf{Set}$ to the following simplicial set:
$$
\xymatrix@M=12pt{i^*D:\Delta^{op} \ar@{^(->}[r]^i & \mathsf{Tree}^{op} \ar[r]^(0.53)D & \mathsf{Set}\ .}
$$
Moreover, the restriction functor $i^*$ has a left adjoint $i_!$ sending a simplicial set $S$ to the dendroidal $i_!S$ such that $(i_!S)_{[n]}=S_n$ for any linear tree $[n]$ and such that $(i_!S)_T= \emptyset$ for any tree $T$ which is not in the image of $i:\Delta \rightarrow \mathsf{Tree}$.\\

Dendroidal sets are thus a generalization of simplicial sets. Moreover, the simplicial sets $\Delta[n]$, $\partial \Delta[n]$ and $\Lambda^k[n]$ have dendroidal analogues which we define below.

\begin{defin}[The dendroidal sets $\Omega\lbrack T\rbrack$, $\partial\Omega\lbrack T \rbrack$, $\Lambda^e \lbrack T \rbrack$ and $\partial^{ext} \Omega\lbrack T \rbrack$]
For any tree $T$, let $\Omega[T]$ be the dendroidal set given by the Yoneda embedding of trees into dendroidal sets, that is 
$$
\Omega[T]:=\mathrm{Hom}_{\mathsf{Tree}}(-,T)\ .
$$
It has several canonical sub-objects that we will use.
\begin{itemize}
\item[$\triangleright$] Let $\partial \Omega[T]$ be the sub-dendroidal set of $\Omega[T]$ generated by all the faces (inner and outer).
$$
\partial\Omega[T]_{T'}:=\{f \in \mathrm{Hom}_{\mathsf{Tree}}(T',T)\ |\ \exists g ,\  \exists \delta ,\ f=\delta g \}\ ,
$$
where $g$ is a morphism of trees and $\delta$ is a coface targeting $T$. 
\item[$\triangleright$] For any inner edge $e$ of $T$, let $\Lambda^e[T]$ be the sub-dendroidal set of $\Omega[T]$ generated by all the faces (inner and outer) except the one corresponding to $e$, that is
$$
\Lambda^e[T]_{T'}:=\{f \in \mathrm{Hom}_{\mathsf{Tree}}(T',T)\ |\ \exists g ,\  \exists \delta \neq \delta_e,\ f=\delta g \}\ ,
$$
where $g$ is a morphism of trees and $\delta$ is a coface targeting $T$ which is different from the inner coface related to the edge $e$.
\item[$\triangleright$] Finally let $\partial^{ext} \Omega[T]$ be the the sub-dendroidal set of $\Omega[T]$ generated by all the outer faces, that is 
$$
\partial^{ext}\Omega[T]_{T'}:=\{f \in \mathrm{Hom}_{\mathsf{Tree}}(T',T)\ |\ \exists g,\  \exists \delta,\ f=\delta g \}\ ,
$$
where $g$ is a morphism of trees and $\delta$ is an outer coface targeting $T$.

\end{itemize} 

\end{defin}

The category of set-theoretical colored operads is canonically embedded in the category of dendroidal sets through the nerve functor $\mathrm{N}_d: \mathsf{Op} \rightarrow \mathsf{dSet}$ defined as follows:
$$
\mathrm{N}_d(\mathscr P)_T:= \mathrm{Hom}_{\mathsf{Op}}(\Omega(T),\mathscr P) \ .
$$
This functor has a left adjoint $\tau_d$ such that $\tau_d(\Omega[T])=\Omega(T)$.\\

\paragraph*{\sc Convention} For any dendroidal set $D$, the elements of the set $D_|$, where $|$ is the trivial tree, will be called the colors of the dendroidal set $D$. Moreover, if we denote the corolla with $m$ inputs by $C_m$, then the elements of $D_{C_m}$ will be called the arity $m$ operations of $D$. For such a corolla and any leaf $l$ of it, there is a face map $D_T \rightarrow D_|$, which gives the color of the input corresponding to $l$. The face map $D_T \rightarrow D_|$ corresponding to the root gives the color of the output. Indeed, for any set-theoretical colored operad $\mathscr P$, the set $\mathrm{N}_d(\mathscr P)_|$ is the set of colors of $\mathscr P$ and $\mathrm{N}_d(\mathscr P)_{C_m}$ are the operations of $\mathscr P$ of arity $m$.

\subsection{Infinity-operads}

We give here the definition of an $\infty$-operad, which is to dendroidal sets what
$\infty$-categories are to simplicial sets. Recall that an \textit{$\infty$-category} (or \textit{quasi-category} or \textit{weak Kan complex}) is a simplicial set $X$ such that for all $n \geq 2$ and $1 \leq k \leq n-1$ and every morphism $\Lambda^k[n] \rightarrow X$, there is a morphism $\Delta[n] \rightarrow X$ which lifts the horn inclusion $\Lambda^k[n] \rightarrow \Delta[n]$.
$$
\xymatrix@M=8pt{\Lambda^k[n] \ar@{^(->}[d] \ar[r] & X \\
\Delta[n] \ar@{-->}[ur]_{\exists}}
$$
Quasi-categories are models of $(\infty,1)$-categories where the objects are the $0$-vertices, the $1$-mor\-phi\-sms are the $1$-vertices and where the higher vertices encode higher homotopical data. Indeed, the above lifting property for $n=2$ and $k=1$ means that any two $1$-morphisms such that the target of the first one is the source of the second one can be composed up to homotopy. More details can be found in \cite{Lurie09}. The following notion of $\infty$-operad is an operadic generalization of this notion of $\infty$-category.

\begin{defin}[$\infty$-operad]
An \textit{$\infty$-operad}, or \textit{infinity-operad} in plain words, is a dendroidal set $D$ such that for every tree $T$ and any inner edge $e$ of $T$, every morphism from $\Lambda^e[T]$ to $D$ can be lifted to a morphism from $\Omega[T]$ to $D$:
$$
\xymatrix@M=8pt{
\Lambda^e[T]\ \  \ar[r] \ar@{^(->}[d] & D\\
\Omega[T] \ . \ar@{-->}[ur]_{\exists}
}
$$
\end{defin}

The ``horn`` condition of the definition for two-vertex trees means that two operations with compatible colors can be composed up to homotopy. Indeed, let $T$ be a tree with two vertices: it is made up of two sub-trees $T_1$ and $T_2$ joined by an edge $e$. Let $v_1$ (resp. $v_2$) be the unique vertex of the tree $T_1$ (resp. $T_2$).

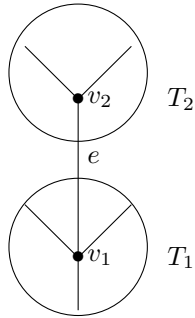
\vspace{1cm}
\begin{center}
\begin{tikzpicture}[scale=0.7]

\draw (2,1) node {$\bullet$} ;
\draw (2,4) node {$\bullet$} ;
\draw (2,1) node[right] {$v_1$} ;
\draw (2,4) node[right] {$v_2$} ;

\draw (2,0) -- (2,1) ;

\draw (2,1) -- (1,2) ;
\draw (2,1) -- (2,4) ;
\draw (2,1) -- (3,2) ;

\draw (2,4) -- (1,5) ;
\draw (2,4) -- (3,5) ;

\draw (2,1.2) circle (1.3) ;
\draw (2,4.5) circle (1.3) ;
\draw (3.5,1) node[right]{$T_1$};
\draw (3.5,4) node[right]{$T_2$};
\draw (2,2.9) node[right]{$e$};
\end{tikzpicture}

\end{center}

Let $T'$ be the corolla obtained from $T$ by contracting the edge $e$. The morphisms $\delta_{v_2}:T_1 \rightarrow T$, $\delta_{v_1}:T_2 \rightarrow T$ and $\delta_e:T' \rightarrow T$ are the three coface maps targeting the tree $T$. Let $x \in D_{T_1} $ and $y \in D_{T_2}$ such that $d_e (x)=d_e (y) \in D_{|}$. Through the two outer face maps, they determine a morphism $\Lambda^e[T] \rightarrow D$, which induces a morphism $\Omega[T] \rightarrow D$, i.e. a element $z$ of $D_T$. The inner face $d_e(z)$ can be thought of as the composition of $x$ with $y$ along the edge $e$. Moreover, the fact that the morphism from $\Omega[T]$ to $D$ extending $x$ and $y$ is not necessarily unique means that their composite is not necessarily strictly unique.

\subsection{The homotopy theory of dendroidal sets} In their paper \cite{CisinskiMoerdijk11}, Denis-Charles Cisinski and Ieke Moerdijk endow the category of dendroidal sets with a model structure in order to provide a homotopical interpretation for the notion of $\infty$-operad.

\begin{defin}[Categorical fibration]
A \textit{categorical fibration} (or \textit{isofibration}) is a functor $f:\mathsf{C} \rightarrow \mathsf{C}'$ such that, given any isomorphism
$\phi : c'_0 \rightarrow c'_1$ in $\mathsf{C}'$, and any object $c_1$ in $\mathsf{C}$ such that $f(c_1)=c'_1$, there exists an
isomorphism $\psi:c_0 \rightarrow c_1$ in $\mathsf{C}$, such that $f(\psi)=\phi$.
\end{defin}

\begin{thm}\cite[Theorem 2.4, Proposition 2.6]{CisinskiMoerdijk11}\label{thmcisinskimoerdijk}
The category of dendroidal sets is endowed with a model structure such that:
\begin{itemize}
\item[$\triangleright$] the fibrant objects are $\infty$-operads.
\item[$\triangleright$] the cofibrations are the normal monomorphisms, i.e. the monomorphisms $A \rightarrow B$ such that for every tree $T$ the action of the automorphism group $Aut(T)$ on $B_T - A_T$ is free. 
\item[$\triangleright$] the fibrations between fibrant objects, i.e. $\infty$-operads, are the morphisms $f$ such that $i^* \tau_d(f)$ is a  categorical fibration and such that $f$ satisfies the right lifting property with respect to the inner horn inclusions, i.e. the morphisms $\Lambda^e[T] \hookrightarrow \Omega[T]$ where $T$ is a tree and $e$ is an inner edge of $T$.
\end{itemize}
Furthermore, this model structure is left proper and combinatorial.
\end{thm}

\begin{rmk}
The restriction of this model structure to the category of simplicial sets, which is the slice category $\mathsf{dSet}/\Delta[0]$, is exactly the Joyal model structure. See \cite[2.2.5]{Lurie09}. Furthermore, in \cite[Proposition E.1.10]{Joyal00} Andr\'e Joyal shows that a model structure on a category is determined by its class of cofibrations and its class of fibrant objects.
\end{rmk}

We will need to have a precise description of the weak equivalences between $\infty$-operads.

\begin{defin}[Essentially surjective morphisms]
A morphism of dendroidal sets $A \rightarrow B$ is \textit{essentially surjective} if the induced functor $i^* \tau_d A \rightarrow i^*\tau_d B$ is essentially surjective, i.e. any object of $i^*\tau_d B$ is isomorphic to the image of an object of $i^*\tau_d A$.
\end{defin}

For any integers $m \geq 0$ and $n \geq 0$, let $C_{m,n}$ be the tree made up of a corolla having $m$ inputs and the linear tree with $n$ vertices $[n]$ under it; see Figure \ref{figuretree}. There is a canonical morphism of trees from $[n]$ to $C_{m,n}$. The data of such morphisms $[n] \rightarrow C_{m,n}$ gives us a cosimplicial object in the category of trees for any $m \geq 0$. 

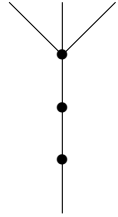
\begin{figure}[h]
\begin{tikzpicture}[scale=0.7]\label{figuretree}

\draw (1,1) node {$\bullet$} ;
\draw (1,2) node {$\bullet$} ;
\draw (1,3) node {$\bullet$} ;

\draw (1,0) -- (1,4) ;

\draw (1,3) -- (0,4) ;
\draw (1,3) -- (2,4) ;
\end{tikzpicture}
\caption{Example: the tree $C_{3,2}$.}
\end{figure}

Let us number the leaves of $C_m$ from $1$ to $m$. Let $A$ be a dendroidal set and let $c$, $c_1$, \ldots, $c_m$ be colors of $A$, i.e. the elements of the set $A_|$. Let $A^L(c_1,\ldots,c_m ; c)$ be the sub-simplicial set of $\mathrm{Hom}_{\mathsf{dSet}}(\Omega[C_{m,-}],A)$ whose $n$ simplices are the morphisms of $\mathrm{Hom}_{\mathsf{dSet}}(\Omega[C_{m,n}],A)$ such that their restriction to the leave $i$ is the color $c_i$ and their restriction to $[n]$ is the degeneracy of the color $c$.

\begin{defin}[Fully faithful morphisms]
A morphism of $\infty$-operads $f:P \rightarrow Q$ is \textit{fully faithful} if for any integer $m \geq 0$ and for any colors $c$, $c_1$, \ldots, $c_m$ of $P$, the morphism
$$
P^L(c_1,\ldots,c_m ; c) \rightarrow Q^L(f(c_1),\ldots,f(c_m) ; f(c))
$$
is a weak homotopy equivalence of simplicial set, i.e. a weak equivalence in the Quillen model structure on simplicial sets.
\end{defin}

\begin{prop}\cite[Proposition 4.17]{Heuts11}
A morphism of $\infty$-operads $f:P \rightarrow Q$ is a weak equivalence if and only it is essentially surjective and fully faithful.
\end{prop}

\begin{rmk}
In the paper \cite{CisinskiMoerdijk13a}, Cisinski--Moerdijk give an other definition of fully faithful morphisms: a morphism of $\infty$-operads $f:P \rightarrow Q$ is ``fully faithful`` if for any integer $m \geq 0$ and for any colors $c$, $c_1$, \ldots, $c_m$ of $P$, the morphism
$$
P(c_1,\ldots,c_m ; c) \rightarrow Q(f(c_1),\ldots,f(c_m) ; f(c))
$$
is a weak homotopy equivalence of simplicial set, where $P(c_1,\ldots,c_m ; c)$ is a simplicial set defined by the following collection of pullbacks
$$
\xymatrix{P(c_1,\ldots,c_m ; c)_n \ar[r] \ar[d] & \mathrm{Hom}_{\mathsf{dSet}}(\Omega[C_m],P^{(\Delta[n])}) \ar[d]\\
\Delta[0] \ar[r]_(0.3){(c,c_1,\ldots,c_m)} & \mathrm{Hom}_{\mathsf{dSet}}(\coprod_{i=0}^m \Delta[0], P^{(\Delta[n])})\ .}
$$
where $(P^{(\Delta[n])})_{n \in \mathbb{N}}$ is a Reedy fibrant replacement of $P$ which is defined in \cite[3.1]{CisinskiMoerdijk13a}. In the appendix we give a direct proof of the fact that $P(c_1,\ldots,c_m ; c)$ is homotopy equivalent to $P^L(c_1,\ldots,c_m ; c)$. This implies that the two definitions of the fully-faithful morphisms of $\infty$-operads are equivalent.
\end{rmk}

\subsection{The Dold Kan correspondance}\label{section:dold-kan}
We recall here the Dold--Kan correspondance between simplicial $\mathbb{K}$-modules and nonnegatively graded chain complex of $\mathbb{K}$-modules. See \cite[3.2]{GoerssJardine09} and \cite[Section 8.8]{MacLane95} for more details on this subject.\\ 

Let $X$ be a simplicial $\mathbb{K}$-module and let $V$ be chain complex of $\mathbb{K}$-modules.
\begin{itemize}
\item[$\triangleright$] Let $C(X)$ be the Moore complex of $X$, that is the nonnegatively graded chain complex such that $C(X)_n:=X_n$ for any integer $n \geq 1$, and whose differential $d:C(X)_n \rightarrow C(X)_{n-1}$ is given by the formula $d:=\sum_{i=0}^n (-1)^i d_i$ where $d_i$ is the $i^{th}$ face.\\

\item[$\triangleright$] Let $N(X)$ be the normalized Moore complex of  $X$, that is the sub-chain complex of $C(X)$ defined by $N(X):= \cap_{i=0}^{n-1} ker(d_i)$.\\

\item[$\triangleright$] Let $D(X)$ be the sub-chain complex of $C(X)$ generated by the images of the degeneracies $D(X):= \sum_{i=0}^{n-1} im(s_i)$. Besides, let $\pi(X)$ the functorial projection of $C(X)$ on $C/D(X):=C(X)/D(X)$. Then, the composite map $N(X) \hookrightarrow C(X) \twoheadrightarrow C/D(X)$ is a functorial isomorphism.\\

\item[$\triangleright$] Let $\Gamma(V)$ be the simplicial $\mathbb{K}$-module such that for any integer $n$, $\Gamma(V)_n:= \bigoplus_{[n] \twoheadrightarrow [p]} V_p$, where the sum is a taken over the maps of ordered sets $[n] \twoheadrightarrow [p]$ which are surjections. The faces and degeneracies are defined as follows. Let $\eta: [n] \twoheadrightarrow [p]$ be a surjection and $v \in V_p \subseteq \Gamma(V)_n$ an element in the $\eta$-summand of $\Gamma(V)_n$. 
\begin{itemize}
\item[$\triangleright$] For any codegeneracy $\sigma_i : [n+1] \twoheadrightarrow [n]$, the corresponding degeneracy $s_i(v)$ of $v$ is the element $v \in V_p \subseteq \Gamma(V)_{n+1}$ in the $\eta\sigma_i$-summand of $\Gamma(V)_{n+1}$.
\item[$\triangleright$] Let $\delta_i : [n-1] \hookrightarrow [n]$ be a coface. If $\eta\delta_i$ is still a surjection, then the corresponding face $d_i(v)$ of $v$ is the element $v \in V_p \subseteq \Gamma(V)_{n-1}$ in the $\eta\delta_i$-summand of $\Gamma(V)_{n-1}$. Otherwise, $\eta\delta_i$ can be uniquely factorized as the composition $[n-1] \twoheadrightarrow [p-1] \hookrightarrow [p]$ of an ordered surjection $\eta'$ followed by a coface $\delta_j$. If $j=p$, then $d_i(v):= (-1)^p d(v) \in V_{p-1}$ in the $\eta'$-summand of $\Gamma(V)_{n-1}$. Otherwise, $d_i(v):=0$.    
\end{itemize}
\end{itemize}

The functor $\Gamma$ from the category $\mathsf{dg}\text{-}\mathsf{Mod}^{\geq 0}$ of nonnegatively graded chain complexes to the category of simplicial $\mathbb{K}$-modules is both right adjoint and left adjoint to the functor $C/D$. Furthermore, these two functors give us an equivalence of categories. This is the Dold--Kan correspondence in the context of $\mathbb{K}$-modules.\\

The functor $C/D$ from the category of simplicial $\mathbb{K}$-modules to the category of nonnegatively graded chain complexes is lax symmetric monoidal through the Eilenberg--Zilber map. Furthermore, it is lax comonoidal through the Alexander--Whitney map; see \cite[Section 8.8]{MacLane95} for a description of these maps. Hence, its adjoint $\Gamma$ is lax monoidal. However, $\Gamma$ is not symmetric monoidal. 


\section{Strict unital homotopy colored operads}

In this section we introduce the new notion of homotopy colored operads with strict unit, which is the operadic generalization of the notion of $\mathcal{A}_\infty$-category \cite{FOOO09I}. For that purpose, we introduce the notions of colored cooperads, conilpotent cofree colored cooperads and coderivations, which are generalizations from the non-colored case in the framework developed in the previous section. The propositions are often proved in the same way as in the non-colored case, but we recall the key properties that  will be used later on.

\subsection{Colored cooperads} We first consider the category of colored cooperads.

\begin{defin}[Colored cooperads]\leavevmode
\begin{itemize}
\item[$\triangleright$]
For any set ${C}$, a \textit{${C}$-cooperad} is a comonoid $(\mathcal{C},\Delta,\varepsilon)$ in the category of $({C},\mathbb{S})$-modules.
\item[$\triangleright$]
More generally, a \textit{colored cooperad} is a quadruple $({C},\mathcal{C}, \Delta, \varepsilon)$, where ${C}$ is a set and where $(\mathcal{C}, \Delta,\varepsilon)$ is a ${C}$-cooperad.
\end{itemize}
\end{defin}

\begin{lemma}
For any function $\phi: {C} \rightarrow {D}$, the functor $\phi_!: ({C},\mathbb{S})\textsf{-}\mathsf{Mod} \rightarrow ({D},\mathbb{S})\textsf{-}\mathsf{Mod}$ is a lax comonoidal functor, i.e. there is a natural morphism $\phi_!(\mathcal{V} \circ \mathcal{W}) \rightarrow \phi_!\mathcal{V} \circ \phi_!\mathcal{W}$ and a morphism $\phi_!I_C \rightarrow I_D$ satisfying coherence conditions.
\end{lemma}

\begin{proof}
For any $({C},\mathbb{S})$-modules $\mathcal{V}$ and $\mathcal{W}$ and any object $(c';\chi': X \rightarrow {D})$ of $\mathsf{Bij}_{{D}}$, we have:
$$
\phi_!(\mathcal{V} \circ \mathcal{W})(\chi' ; c')= \coprod_{k \geq 1} \Big(\coprod_{c, \chi, \alpha, \upsilon}  \mathcal{V}(\upsilon ; c)\otimes \mathcal{W}(\chi_{|\alpha^{-1}(1)} ; \upsilon(1)) \otimes \cdots \otimes \mathcal{W}(\chi_{|\alpha^{-1}(k)} ; \upsilon(k)) \Big)_{\mathbb{S}_k}\ ,
$$
where the second coproduct is taken over the colors $c$ in ${C}$ such that $\phi(c)=c'$, the functions $\chi: X \rightarrow {C}$ such that $\phi \chi= \chi'$, the functions $\alpha \rightarrow \underline{k}$ and the functions $\upsilon: \underline{k} \rightarrow {C}$. Furthermore we have,
$$
(\phi_!\mathcal{V} \circ \phi_!\mathcal{W})(\chi' ; c')= \coprod_{k \geq 1} \Big(\coprod_{c, \chi, \alpha, \upsilon\atop  (c_1, \ldots ,c_k)} \mathcal{V}(\upsilon ; c)\otimes \mathcal{W}(\chi_{|\alpha^{-1}(1)} ; c_1) \otimes \cdots \otimes \mathcal{W}(\chi_{|\alpha^{-1}(k)} ; c_k) \Big)_{\mathbb{S}_k}\ ,
$$
where the second coproduct is taken over the colors $c$ in ${C}$ such that $\phi(c)=c'$, the functions $\chi: X \rightarrow {C}$ such that $\phi \chi= \chi'$, the functions $\alpha \rightarrow \underline{k}$, the functions $\upsilon: \underline{k} \rightarrow {C}$, and the $k$-tuples of colors $(c_1, \ldots ,c_k)$ such that $\phi(c_i)=  \phi(\upsilon(i))$. The map
\begin{align*}
 \{\upsilon: \underline{k} \rightarrow {C}\} &\rightarrow\{(\upsilon: \underline{k} \rightarrow {C},(c_1, \ldots ,c_k))\ |\  \phi(c_i)=\phi(\upsilon(i)), \ \forall i\}\\
 \upsilon &\mapsto \big(\upsilon,(\upsilon(1), \ldots ,\upsilon(k))\big)
\end{align*}
\noindent induces a monomorphism $\phi_!(\mathcal{V} \circ \mathcal{W})(\chi' ; c') \hookrightarrow (\phi_!\mathcal{V} \circ \phi_!\mathcal{W})(\chi' ; c')$ which satisfies the required properties.
\end{proof}

\begin{prop}
Let $\phi:{C} \rightarrow {D}$ be a function, and let $(\mathcal{C},\Delta,\varepsilon)$ be a ${C}$-cooperad. The $({D},\mathbb{S})$-module $\phi_!\mathcal{C}$ has a canonical structure of ${D}$-cooperad $(\phi_!\mathcal{C},\Delta_{\phi},\varepsilon_{\phi})$ induced by the structure of ${C}$-cooperad of $\mathcal{C}$.
\end{prop}

\begin{proof}
This is a corollary of the previous lemma since comonoids induce comonoids through lax comonoidal functors.
\end{proof}

\begin{defin}[Morphisms of colored cooperads]
A \textit{morphism of colored cooperads} from $\mathscr{C}=({C},\mathcal{C},\Delta,\varepsilon)$ to $\mathscr{D}=({D}, \mathcal{D}, \Delta', \varepsilon')$ is a morphism of $\mathbb{S}$-modules $f=(\phi,f_!)$ such that $f_!$ is a morphism of ${D}$-cooperads from $(\phi_!\mathcal{C},\Delta_{\phi},\varepsilon_{\phi})$ to $(\mathcal{D},\Delta',\varepsilon')$.
\end{defin}

\begin{prop}
Colored cooperads with their morphisms form a category denoted by $\mathsf{Coop}$.
\end{prop}

\begin{proof}
This proof is similar to the proof of Proposition~\ref{prop:CategoryOp} for colored operads.
\end{proof}

\subsection{Coaugmented colored cooperads} Throughout this section, $\mathsf{E}$ is an abelian monoidal category.

\begin{defin}[Coaugmented colored cooperad]
A \textit{coaugmented colored cooperad} $\mathscr{C}=({C},\mathcal{C}, \Delta, \allowbreak \varepsilon, \allowbreak u)$ is the data of a colored cooperad $({C},\mathcal{C}, \Delta, \varepsilon)$ together with a morphism of ${C}$-cooperads $u:I_{{C}} \rightarrow \mathcal{C}$.
\end{defin}

Since $\mathsf{E}$ is an abelian category, any coaugmented colored cooperad $\mathscr{C}$ has the form $\mathcal{C}=I_{{C}}\oplus \overline{\mathcal{C}}$ where $\overline{\mathcal{C}}$ is the kernel of the counit map $\varepsilon: \mathcal{C}  \rightarrow I_{{C}}$. Furthermore, the restriction $\Delta_{|\overline{\mathcal{C}}}$ of the coproduct $\Delta$ to $\overline{\mathcal{C}}$ is equal to 
$$
\Delta_{|\overline{\mathcal{C}}}= I_{{C}} \circ \id + \id \circ I_{{C}} +\overline{\Delta} \ ,
$$
where the map $\overline{\Delta}$ is made up of the image of $\Delta$ living in the summand of $\overline{\mathcal{C}} \circ (I_{{C}} \oplus \overline{\mathcal{C}})$ where $\overline{\mathcal{C}}$ appears more than once on the right-hand side of the composite product $\circ$. Moreover the restriction of the coproduct $\Delta$ to $I_{{C}}$ is the canonical morphism $I_{{C}} \rightarrow I_{{C}} \circ I_{{C}}$.

\begin{defin}[Morphisms of coaugmented colored cooperads]
A \textit{morphism} of coaugmented colored cooperads from $({C},\mathcal{C},\Delta,\varepsilon, u)$ to $({D},\mathcal{D},\Delta', \varepsilon', u')$ is a morphism of colored cooperads $f=(\phi,f_!)$, whose restriction to $I_{{C}}$ is equal to the identity:
$$
\xymatrix@M=9pt{I_{{C}}(c;c) = \mathds{1}_{\mathsf{E}} \ar[r]^(0.42){\id} & \mathds{1}_{\mathsf{E}}= I_{{D}}(\phi(c);\phi(c))\ .}
$$
This defines the category of coaugmented colored cooperads. 
\end{defin}

\subsection{The tree module and the free colored operad} The tree module is the underlying construction of the free colored operad and the cofree colored cooperad on an arbitrary colored $\mathbb{S}$-module.

\begin{defin}[Colored trees]
Let ${C}$ be a set. A \textit{${C}$-colored tree} $t=(T,\kappa)$ is the data of a tree $T=({V}, {F}, u,\rho,r )$ and a coloring function $ \kappa$ from the set of edges of $T$ to the set of colors ${C}$. A \textit{morphism} of ${C}$-colored trees from $t$ to $t'$ is a morphism of trees such that the induced function on edges commutes with the coloring functions.
\end{defin}

\begin{nota}
Note that trees are denoted by capital letters, whereas colored trees are denoted by small letters. Let $t=(T ,\kappa)$ be a ${C}$-colored tree and $\mathcal{V}$ be a $({C},\mathbb{S})$-module. For any vertex $v$ of $t$, we denote by $\mathrm{in}(v)$ and $\mathrm{out}(v)$ respectively the set of inputs and the one-point-set of the output of $v$. Then we will denote the object $\mathcal{V}(\kappa(\mathrm{out}(v));\kappa_{|\mathrm{in}(v)})$ simply by $\mathcal{V}(v)$.
\end{nota}

For any $({C},\mathbb{S})$-module $\mathcal{V}$  and any ${C}$-colored tree $t=(({V},{F},u,\rho,r),\kappa)$, we denote by $t(\mathcal{V})$ the following colimit in the category $\mathsf{E}$ 
\begin{equation}\label{w(t)}
\displaystyle{t(\mathcal{V}):= \Big(\coprod_{\phi:\underline{n} \rightarrow {V}} \mathcal{V}(\phi(1)) \otimes \cdots \otimes \mathcal{V}(\phi(n)) \Big)_{\mathbb{S}_n}}\ ,
\end{equation}
which is made up of all the possible ways of labeling the vertices of the tree $t$ with elements of $\mathcal{V}$.
This colimit is taken over the set of bijections from the set $\underline{n}$ to the set ${V}$ of vertices of $t$, $n$ being the number of vertices of $t$, modulo the action of $\mathbb{S}_n$.\\

For any object $(\chi:X \rightarrow {C} ; c)$ in the category $\mathsf{Bij}_{{C}}$, we consider the category $\mathsf{Tree}_{C}(\chi ; c)$ where the  objects are pairs $(t,\alpha)$, with $t$  a $C$-colored tree whose root is colored by $c$ and $\alpha$  a bijection from $X$ to the leaves of $t$ such that the following diagram commutes
$$
\xymatrix{X \ar[rr]^(0.44){\alpha} \ar[rd]_{\chi} &&  \mathrm{leaves}(t) \ar[ld]^{\kappa} \\
& {C} \  . }
$$
The morphisms in $\mathsf{Tree}_{C}(\chi ; c)$ from $(t,\alpha)$ to $(t',\alpha')$ are the isomorphisms of $C$-colored trees $\beta: t \rightarrow t'$ such that $\alpha' =\beta \alpha $ on the leaves. For any $({C},\mathbb{S})$-module, Formula~(\ref{w(t)}) induces a functor from the category $\mathsf{Tree}_{C}(\chi ; c)$ to the category $\mathsf{E}$, i.e. a diagram in $\mathsf{E}$.

\begin{defin}[Tree module]
For any $({C},\mathbb{S})$-module $\mathcal{V}$, the \textit{tree module} $\mathbb{T}\mathcal{V}$ is defined, for any object $(\chi ; c)$ of the category of $\mathsf{Bij}_{{C}}$, by the following colimit:
$$
\displaystyle \mathbb{T}\mathcal{V}(\chi ; c):= \mathop{\rm colim}_{(t,\alpha) \in \mathsf{Tree}_{C}(\chi ; c)}
t(\mathcal{V})\ .
$$
This construction is functorial in $(\chi ; c)$ and thus defines a $({C},\mathbb{S})$-module.
\end{defin}

For any ${C}$-colored tree $t$, we can consider the object $t(\mathcal{V})$ of $\mathsf{E}$ as a $({C},\mathbb{S})$-module by taking the above colimit only over the pairs $(t', \alpha)$ such that $t'$ is isomorphic, as a $C$-colored tree, to $t$, that is 
$$
t(\mathcal{V})(\chi ; c):=
\mathop{\rm colim}_{(t',\alpha) \in \mathsf{Tree}_{C}(\chi ; c)\atop t' \simeq t}
 t'(\mathcal{V})\ .
$$
Note that the coproduct $\coprod_{[t]} t(\mathcal{V})$ over the isomorphism classes $[t]$ of $C$-colored trees is isomorphic to the tree module of $\mathcal{V}$, 
$$
\mathbb{T}\mathcal{V} \cong \coprod_{[t]} t(\mathcal{V}) \ .
$$

Finally, the tree module $\mathbb{T}\mathcal{V}$ is functorial in $\mathcal{V}$ and thus defines an endofunctor $\mathbb{T}$ of the category $({C},\mathbb{S})\textrm{-}\mathsf{Mod}$. It canonically extends to an endofunctor $\mathbb{T}$ of the whole category of colored $\mathbb{S}$-modules. In the sequel, we will also work with the \textit{augmented tree module}  $\overline{\mathbb{T}}\mathcal{V}$  made up of non trivial trees.
$$
\overline{\mathbb{T}}\mathcal{V} \cong \coprod_{[t] \neq |} t(\mathcal{V}) \ .
$$

\begin{rmk}
Let $t=t_1 \sqcup \ldots \sqcup t_k$ be a partition of the $C$-colored tree $t$ into sub-trees. Let $\phi: C  \rightarrow D$ be a function and let $f(t_i)$ be a morphism of colored $\mathbb{S}$-modules over $\phi$ from $t_i(\mathcal{V})$ to $(D,\mathcal{W})$,  for any $i$. Then, it induces a morphism of colored $\mathbb{S}$-modules over $\phi$ from $t(\mathcal{V})$ to $(D,\mathcal{W})$, that we denote by $f(t_1) \otimes \cdots \otimes f(t_k)$.
\end{rmk}

\begin{prop}
The colored $\mathbb{S}$-module $\mathbb{T}\mathcal{V}$ has a canonical structure of a colored operad given by the grafting of trees. We denote this colored operad by $\mathbb{T}^o\mathcal{V}$. It gives rise to  a functor $\mathbb{T}^o: \mathbb{S}\textsf{-}\mathsf{Mod} \to \mathsf{Op}$ from the category of $\mathbb{S}$-modules to the category $\mathsf{Op}$ of colored operads, which is left adjoint to the forgetful functor $\mathsf{Op} \to \mathbb{S}\textsf{-}\mathsf{Mod}$ .
\end{prop}

\begin{proof}
The proof is similar to the  non-colored case, see \cite[Section~5.8]{LodayVallette12}. The extension of morphisms from $\mathcal V$ to $\mathbb{T}^o\mathcal{V}$ is given by the construction mentioned in the above remark.
\end{proof}

In plain words, the colored operad $\mathbb{T}^o\mathcal{V}$ is the \textit{free colored operad} on the colored $\mathbb{S}$-module $\mathcal{V}$.


\begin{nota}
If $f$ is a morphism of $\mathbb{S}$-modules from $({C},\mathbb{T}\mathcal{V})$ to $({D},\mathcal{W})$, then we denote by $f(t)$ the restriction of $f$ to $t(\mathcal{V})$, for any ${C}$-colored tree $t$.
\end{nota}

\subsection{Conilpotent colored cooperads} We suppose again that $\mathsf{E}$ is an abelian category.

\begin{defin}[Conilpotent colored cooperads]
A \textit{conilpotent} colored cooperad $\mathscr{C}=({C}, \mathcal{C}, \Delta, \varepsilon, u)$ is a coaugmented colored cooperad such that the images of any element under the right-hand side iterations of the decomposition map $\bar{\Delta}+ \id \circ I_C$  stabilize at some point.(We refer the reader to \cite[Section~$5.8$]{LodayVallette12} for more details in the non-colored case.) The full subcategory of the category of coaugmented colored cooperads made up of the conilpotent colored cooperads is denoted by $\mathsf{ConilCoop}$.
\end{defin}

\begin{prop}
For any $({C},\mathbb{S})$-module $\mathcal{V}$, the tree module $\mathbb{T}\mathcal{V}$ has a canonical structure of conilpotent ${C}$-cooperad given by the degrafting of trees. We denote this colored cooperad by $\mathbb{T}^c\mathcal{V}$. 
This defines a functor $\mathbb{T}^c : \mathbb{S}\textsf{-}\mathsf{Mod} \to \mathsf{ConilCoop}$ from the category of colored $\mathbb{S}$-modules to the category of conilpotent colored cooperads, which is right adjoint to the forgetful functor $\mathscr C \mapsto \overline{\mathcal{C}}$.
\end{prop}

\begin{proof} 
The proof works in the same way as in the non-colored case, see \cite[Section~$5.8$]{LodayVallette12}. 
For any conilpotent cooperad $\mathscr{C}=({C},\mathcal{C},\Delta,\varepsilon,u)$, there is a canonical morphism of conilpotent ${C}$-cooperads  $\delta : \mathscr{C} \to \mathbb{T}^c\mathcal{C}$  \cite[Proof of Theorem~$5.8.9$]{LodayVallette12}. 
The natural isomorphism
$$
\mathrm{Hom}_{\mathbb{S}\textsf{-}\mathsf{Mod}}(({C},\overline{\mathcal{C}}),({D},\mathcal{V})) \cong \mathrm{Hom}_{\mathsf{ConilCoop}}(\mathscr{C},({D},\mathbb{T}^c\mathcal{V}))\ ,
$$
for any conilpotent colored cooperad $\mathscr{C}$ and any $\mathbb{S}$-module $({D},\mathcal{V})$ is given as follows. Any morphism of $\mathbb{S}$-modules $f:({C},\overline{\mathcal{C}}) \rightarrow ({D},\mathcal{V})$ extends to a morphism of conilpotent colored cooperads $Rf$ from $\mathscr{C}$ to $\mathbb{T}^c\mathcal{V}$ using the following formula:
$$
Rf=(\mathbb{T}^cf)\,  \delta \ .
$$
\end{proof}
In plain words, the colored cooperad $\mathbb{T}^c\mathcal{V}$ is the \textit{cofree conilpotent colored cooperad} on the colored $\mathbb{S}$-module $\mathcal{V}$. In the case where $\mathscr{C}$ is a cofree conilpotent colored cooperad  $\mathscr{C}=({C},\mathbb{T}^c\mathcal{W})$, the adjoint morphism $Rf : \mathbb{T}^c\mathcal{W} \to \mathbb{T}^c\mathcal{V}$ is given by the more simple formula
\begin{equation}\label{adjoint1}
\displaystyle{Rf(t)= \sum_{t=t_1 \sqcup t_2 \sqcup \ldots \sqcup t_k}    f(t_1) \otimes f(t_2) \otimes \cdots \otimes f(t_k)}\ , 
\end{equation}
where the sum is taken over the partitions with no trivial component of the ${C}$-colored tree $t$.

\subsection{Derivations and coderivations}

In this paragraph, the monoidal category $\mathsf{E}$ is, most of the time, the category $\mathsf{gr}\textsf{-}\mathsf{Mod}$ of graded $\mathbb{K}$-modules with degree zero linear maps. It is a subcategory of the category $\mathsf{gr}\textsf{-}\mathsf{Mod}^{\mathsf{deg}}$ of graded $\mathbb{K}$-modules with graded morphisms. We will just allow ourselves to use morphisms of degree different from zero to build codifferentials.\\

Let $\mathcal{V}$ and $\mathcal{W}$ be two $({C},\mathbb{S})$-modules. By definition, $\mathcal{V}$, $\mathcal{W}$ and $\mathcal{V} \circ \mathcal{W}$ are $\mathsf{gr}\textsf{-}\mathsf{Mod}$-presheaves over $\mathsf{Bij}_{{C}}$. Let $f: \mathcal{V} \rightarrow \mathcal{V}$ and $g: \mathcal{W} \rightarrow \mathcal{W}$ be two endomorphisms of $\mathsf{gr}\textsf{-}\mathsf{Mod}^{\mathsf{deg}}$-presheaves over $\mathsf{Bij}_{{C}}$. For any homogeneous elements $x \in \mathcal{V}(\phi ; c)$ and $x_i \in \mathcal{W}(\psi_i ; \phi(i))$, for $1 \leq i \leq k$, we consider the following map:
$$
x \otimes x_1 \otimes \cdots \otimes x_k \mapsto f(x) \otimes x_1 \otimes \cdots \otimes x_k + \sum_{i=1}^n (-1)^{|g|(|x|+|x_1|+\cdots +|x_{i-1}|)} x \otimes x_1 \otimes \cdots \otimes g(x_i) \otimes \cdots \otimes x_k
$$
The collection of these maps can be lifted to a morphism 
$$
f\circ \id_{\mathcal W} + \id_{\mathcal V} \circ' g: \mathcal{V} \circ \mathcal{W} \rightarrow \mathcal{V} \circ \mathcal{W}
$$
of $\mathsf{gr}\textsf{-}\mathsf{Mod}^{\mathsf{deg}}$-presheaves over $\mathsf{Bij}_{{C}}$, which is a linearization of the morphism $f\circ g$.

\begin{defin}[Derivations, differentials, coderivations, and codifferentials]\leavevmode
\begin{itemize}

\item[$\triangleright$]
A \textit{derivation} of a colored operad $\mathscr{P}=({C},\mathcal{P},\gamma, \eta)$ is a morphism $d:\mathcal{P} \rightarrow \mathcal{P}$ of $\mathsf{gr}\textsf{-}\mathsf{Mod}^{\mathsf{deg}}$-presheaves over $\mathsf{Bij}_{{C}}$ such that $$\gamma \  (d \circ \id_{\mathcal P} +
\id_{\mathcal P}\circ' d)= d \  \gamma\ .$$
A \textit{differential} is a degree $-1$ square-zero derivation. 

\item[$\triangleright$]
A \textit{coderivation} of a colored cooperad $\mathscr{C}=({C},\mathcal{C},\Delta,\varepsilon)$ is a morphism $d:\mathcal{C} \rightarrow \mathcal{C}$ of $\mathsf{gr}\textsf{-}\mathsf{Mod}^{\mathsf{deg}}$-presheaves over $\mathsf{Bij}_{{C}}$ such that 
$$ (d \circ \id_{\mathcal C} +
\id_{\mathcal C}\circ' d)\  \Delta= \Delta \  d \ . $$
A \textit{codifferential} is a degree $-1$ square-zero coderivation. 

\end{itemize}  
\end{defin}

A dg colored operad is an operad equipped with a differential. Morphisms of dg colored operads are morphisms of graded colored operads commuting with the differentials. We denote this category by $\mathsf{dg}\textsf{-}\mathsf{Op}$. The same phenomenon holds for dg colored cooperads, i.e. colored cooperads equipped with a codifferential. For any coaugmented colored cooperad $\mathscr{C}=({C},\mathcal{C},\Delta,\varepsilon,u)$, we require moreover that coderivations satisfy $\varepsilon \  d=0$ and $d \ u=0$.\\

Coderivations on cofree colored cooperads are completely characterized by their projections onto their generators.

\begin{prop}\label{prop:Coder}
Let ${\gamma} : \overline{\mathbb{T}}\mathcal{V} \rightarrow \mathcal{V}$ be a graded morphism. There is a unique coderivation $d_{\gamma}$ on the cofree colored cooperad $\mathbb{T}^c\mathcal V$ which extends ${\gamma}$; it is given by the following  formula:
$$
{d_{\gamma}}(t) = \sum_{s \subset t}  \id \otimes \cdots \otimes  {\gamma}(s) \otimes \cdots \otimes \id \ ,   
$$
where the sum is taken over the non-trivial sub-trees $s$ of a colored tree $t$. In this context, the coderivation $d_{\gamma}$ squares to zero if and only if $\, {\gamma} \circ d_{\gamma}=0$.
\end{prop}

\begin{proof}
The proof is similar to the non-colored case, see \cite[Chapter~$6$]{LodayVallette12}.
\end{proof}

\begin{nota}
For a coderivation $d_{\gamma}$  on a cofree colored cooperad $\mathbb{T}^c\mathcal{V}$, we denote by ${\gamma}$ its projection onto $\mathcal{V}$.
\end{nota}

As we have already seen through Equation~(\ref{adjoint1}), a morphism of conilpotent cofree colored cooperads from $({C},\mathbb{T}^c\mathcal{V}) $ to $({D},\mathbb{T}^c\mathcal{W})$ is equivalent to the data of a morphism of $\mathbb{S}$-modules $f$ from $({C},\overline{\mathbb{T}}\mathcal{V})$ to $ ({D},\mathcal{W})$. If the cofree colored cooperads are equipped with codifferentials, then the following proposition gives the condition on $f$ under which $Rf$ is a morphism of dg cooperads.

\begin{prop}\label{key}
Let $({C},\mathcal{V})$ and $({D},\mathcal{W})$ be $\mathbb{S}$-modules, let $d_{\gamma}$ and $d_{\nu}$ be two codifferentials on $\mathbb{T}^c\mathcal{V}$ and $\mathbb{T}^c\mathcal{W}$ respectively, and let $f: ({C},\overline{\mathbb{T}}\mathcal{V}) \rightarrow ({D},\mathcal{W})$ be a morphism of $\mathbb{S}$-modules. Then $Rf$ is a morphism of coaugmented dg cooperads, i.e. it commutes with the codifferentials, if and only if
\begin{equation}\label{eqn:MORPH}
{\nu} Rf = f d_{\gamma} \ .
\end{equation}
\end{prop}

\begin{proof}
The proof is similar to the proof of Proposition~$10.5.3$ of \cite{LodayVallette12}. On the one hand, if $d_{\nu} Rf = Rf  d_{\gamma}$, then ${\nu} \  Rf= f \  d_{\gamma}$ as ${\nu}$ (resp. $f$) is the projection onto $W$ of $d_{\nu}$ (resp. $Rf$). On the other hand, since $Rf$ is given by Formula ($\ref{adjoint1}$) and since $d_\gamma$ is given by  Proposition~\ref{prop:Coder},  we have
$$
(Rf  d_{\gamma})(t)= \sum_{s \subset t} \sum_{t=t_1 \sqcup \ldots \sqcup s \sqcup \ldots \sqcup t_k} f(t_1) \otimes \cdots \otimes (f  d_{\gamma})(s) \otimes \cdots \otimes f(t_k)  \ ,
$$
 for any ${C}$-colored tree $t$.
We also have:
$$
(d_{\nu}  Rf)(t)= \sum_{s \subset t} \sum_{t=t_1 \sqcup \ldots \sqcup s \sqcup \ldots \sqcup t_k} f(t_1) \otimes \cdots \otimes ({\nu} Rf)(s) \otimes \cdots \otimes f(t_k)  \ .
$$
So, if ${\nu}  Rf= f  d_{\gamma}$, then $d_{\nu} Rf = Rf  d_{\gamma}$.
\end{proof}

\paragraph{\sc Convention}\label{convdg}
In the case where the category $\mathsf{E}$ is the category $\mathsf{dg}\textsf{-}\mathsf{Mod}$ of chain complexes, a codifferential on the cofree cooperad  $\mathbb{T}^c\mathcal{V}$ has the from $d_1+d_{\geq 2}$, where $d_1$ is the internal codifferential induced by the differential of $\mathcal{V}$ through the formula
$$
d_1(x_1 \otimes \cdots \otimes x_m)= \sum_{i=1}^n (-1)^{|x_1|+ \cdot\cdot\cdot +|x_{i-1}|} 
(x_1 \otimes \cdots \otimes d_{\mathcal V}(x_i) \otimes \cdots \otimes x_n )\ ,
$$
and where $d_{\geq 2}$ is an additional codifferential, which is nonzero only on $\mathbb{T}^{\geq 2}\mathcal{V}$, the summand made up of trees with at least two vertices. We refer the reader to \cite[Chapter~$6$]{LodayVallette12} for more details.

\subsection{The categories of homotopy colored operads}\label{subscteionhomotopyoperads}

The concept of homotopy operads in the differential  graded context was introduced by Pepijn Van der Laan in the non-colored case in the paper \cite{VanDerLaan03}. In order to compare this notion to $\infty$-operads, we need to extend it by including colors and adding a homotopy coherent unit. In this section, the category $\mathsf{E}$ is the category $\mathsf{dg}\textsf{-}\mathsf{Mod}$ of chain complexes.

\begin{defin}[Strict unital homotopy colored operads]\leavevmode
\begin{itemize}
\item[$\triangleright$]
A \textit{nonunital homotopy colored operad} $\mathscr{P}=({C},\mathcal{P}, {\gamma})$ is the data of a colored  $\mathbb{S}$-module $({C},\mathcal{P})$ and a codifferential $d_{\gamma}$ on the cofree colored cooperad $\mathbb{T}^c(s\mathcal{P})$ on the suspension $s\mathcal{P}$ of $\mathcal{P}$. 

\noindent
A \textit{strict unital homotopy colored operad} $\mathscr{P}=({C},\mathcal{P},{\gamma}, \eta)$ is the data of a nonunital homotopy operad $({C},\mathcal{P},{\gamma})$ together with a morphism $\eta: I_C \to \mathcal P$ of $\mathbb{S}$-modules called the \textit{unit}. For each color $c$ in ${C}$, we denote by $\ii_c$ the image of the unit $1_{\mathbb{K}}$ of the ground field $\mathbb{K}$ under the map $\eta(c;c)$ from $\mathbb{K} = I_{{C}}(c;c)$ to $\mathcal{P}(c;c)$. Furthermore we require that the unit satisfies the following homotopy coherences:
$$\quad \quad \quad 
\left\{\begin{array}{ll}
\gamma(s \ii)=0 \\
{\gamma}(t)(s \ii_c \otimes sp)=sp,&
\text{for colored trees}\  t \ \text{with 2 vertices}; \\
{\gamma}(t)(sp \otimes s\ii_c)=(-1)^{|p|}sp,&
\text{for colored trees}\  t \ \text{with 2 vertices}; \\
{\gamma}(t)(sp_1 \otimes \cdots \otimes s\ii_c \otimes \cdots \otimes sp_{n-1})=0, &\text{for colored trees}\  t \ \text{with at least 3 vertices}.
\end{array}\right.
$$
In the second (resp. the third) equation $s\ii_c$ (resp. $sp$) labels the vertex attached to the root of $t$.

\item[$\triangleright$]
A \textit{morphism of nonunital homotopy colored operads} $\mathcal{P}=({C},\mathcal{P}, {\gamma}) \rightsquigarrow \mathcal{Q}=({D},\mathcal{Q}, {\nu})$ is a morphism of coaugmented dg cooperads 
$R f : (\mathbb{T}^c(s\mathcal{P}), d_\gamma)  \to (\mathbb{T}^c(s\mathcal{Q}), d_\nu)$.

\noindent
A \textit{morphism of strict unital homotopy colored operads} from $\mathscr{P}=({C},\mathcal{P}, {\gamma}, \eta)$ to $\mathscr{Q}=({D},\mathcal{Q},{\nu},\theta)$ is a morphism $R f$ of nonunital homotopy colored operads such that its composite with the projection onto the generators $f: \overline{\mathbb{T}}(s\mathcal{P}) \rightarrow s\mathcal{Q}$ satisfies:
$$
\left\{\begin{array}{ll}
f(s\ii_c)=s\ii_{\phi(c)} & \\
f(sp_1 \otimes \cdots \otimes s\ii_c \otimes \cdots \otimes sp_{n-1})=0\ .& 
\end{array}\right.
$$
The category of strict unital homotopy colored operads is denoted by $\mathsf{su Op}_{\infty}$.
\end{itemize}
\end{defin}

\paragraph{\sc Interpretation} Let us unfold this definition a little bit. A nonunital homotopy colored operad $\mathscr{P}=({C},\mathcal{P}, {\gamma})$ can actually be viewed as a colored $\mathbb{S}$-module endowed with a partial composition ``associative up to higher homotopies''. To be precise, it is necessary to give an orientation to the trees. This orientation will allow us to deal with the signs inherent to the underlying symmetric monoidal structure of the category $\mathsf{dg}\textsf{-}\mathsf{Mod}$ and which come from the suspension. For example, let $t=(\{v_0,v_1,v_2\},{F}, u,\rho,r,\kappa)$ be a ${C}$-colored tree with three vertices represented in the following picture. 

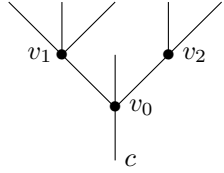
\begin{figure}[h]
\begin{tikzpicture}[scale=0.7]
\draw (2,1) node {$\bullet$} ;
\draw (2,1) node[right]{$\, v_0$} ;

\draw (1,2) node {$\bullet$} ;
\draw (1,2) node[left]{$v_1$} ;
\draw (3,2) node {$\bullet$} ;
\draw (3,2) node[right]{$\, v_2$} ;
\draw (2,0) -- (2,1) ;

\draw (2,1) -- (1,2) ;
\draw (2,1) -- (2,2) ;
\draw (2,1) -- (3,2) ;

\draw (1,2) -- (0,3) ;
\draw (1,2) -- (1,3) ;
\draw (1,2) -- (2,3) ;

\draw (3,2) -- (3,3) ;
\draw (3,2) -- (4,3) ;
\draw (2,0) node[right] {$c$} ;
\end{tikzpicture}
\caption{The tree $t$.}
\end{figure}

The set of leaves is denoted by $l$ and the color of the root is $c \in {C}$. The sub-tree containing $v_0$ and $v_1$ (respectively $v_0$ and $v_2$) is denoted by $t_1$ (resp. $t_2$) and its leaves by $l_1$ (resp. $l_2$). Let $x$ be an element of the $\mathbb{K}$-module $t(\mathcal{P})$. It is equal to a sum of elements $a_0 \otimes a_1 \otimes a_2$ where $a_i \in \mathcal{P}(v_i)$. The choice of such a representation of $x$ is related to the way we travel along the tree: in, this case, the path is $(v_0, v_1, v_2)$. This defines an orientation of the colored tree $t$. This orientation induces a morphism $\mathcal{P}(v_0) \otimes \mathcal{P}(v_1) \otimes \mathcal{P}(v_2) \rightarrow s\mathcal{P}(v_0) \otimes s\mathcal{P}(v_1) \otimes s\mathcal{P}(v_2)$ through the mapping $$a_0 \otimes a_1 \otimes a_2 \mapsto (-1)^{|a_1|}(sa_0) \otimes (sa_1) \otimes (sa_2)\ . $$ 
Then, applying ${\gamma}$ and the desuspension map $
s\mathcal{P} \rightarrow \mathcal{P}$ gives a morphism $$\gamma_3:\mathcal{P}(v_0) \otimes \mathcal{P}(v_1) \otimes \mathcal{P}(v_2) \rightarrow \mathcal{P}(l)$$
of degree $1$. Furthermore, the orientation that we have chosen for $t$ induces orientations on $t_1$, $t_2$, $t/t_1$ and $t/t_2$. Applying the same procedure produces degree $0$ morphisms respectively from $\mathcal{P}(v_0) \otimes \mathcal{P}(v_1)$ to $\mathcal{P}(l_1)$, from $\mathcal{P}(v_0) \otimes \mathcal{P}(v_2)$ to $\mathcal{P}(l_2)$, from $\mathcal{P}(l_1) \otimes \mathcal{P}(v_2)$ to $\mathcal{P}(l)$ and from $\mathcal{P}(l_2) \otimes \mathcal{P}(v_1)$ to $\mathcal{P}(l)$. We denote all of them by $\gamma_2$, since they amount to composing $2$ vertices labeled by $\mathcal P$. Let $\tau$ be the canonical isomorphism $\mathcal{P}(v_0) \otimes \mathcal{P}(v_1) \otimes \mathcal{P}(v_2) \simeq \mathcal{P}(v_0) \otimes \mathcal{P}(v_2) \otimes \mathcal{P}(v_1)$. The fact that $d_{\gamma}$ squares to zero implies:
$$
\gamma_{2} (\gamma_{2} \otimes \id) - \gamma_{2} (\gamma_{2} \otimes \id)\tau = \partial (\gamma_3)
$$
where $\partial (\gamma_3)= d_{\mathcal{P}} \  \gamma_3 + \gamma_3 \  d_{\mathcal{P}^{\otimes 3}}$. We interpret $\gamma_2$ as a partial composition; so the above equation shows that the parallel composition is not strictly associative but ``associative up to homotopy``; and this homotopy is precisely $\gamma_3$.
In the same way, $\gamma_3$ applied to trees with $3$ vertices one above another provides us with a homotopy for the sequential composite of $\gamma_2$. The other maps ${\gamma}(t)$, for bigger trees $t$, are higher homotopies. Indeed, given an orientation on a colored tree $t$, we have the following equation:
$$
\partial(\gamma(t))= \sum_{s \subset t} \pm \gamma(t/s) (\id \otimes \cdots \otimes  \gamma(s) \otimes \cdots \otimes \id)\ ,
$$
where the sum is taken over all the sub-trees $s$ of $t$, with at least $2$ vertices.
In strict unital homotopy colored operads, the composition is relaxed up to homotopy but the unit remains strict.\\

\paragraph{\sc Interpretation}  A morphism of nonunital homotopy colored operads from $\mathscr{P}=({C},\mathcal{P}, {\gamma}) $ to $\mathscr{Q}=({D},\mathcal{Q}, {\nu})$ is a morphism of colored $\mathbb{S}$-modules $f:\overline{\mathbb{T}}^c(s\mathcal{P}) \rightarrow s\mathcal{Q}$ such that ${\nu}  Rf=f d_{\gamma}$, according to Proposition~\ref{key}. Again, a choice of orientation of a colored tree $t$ with $n$ vertices $v_0, \ldots, v_{n-1}$ gives a morphism of graded $\mathbb{K}$-modules $f_n:\mathcal{P}(v_0) \otimes \cdots \otimes \mathcal{P}(v_{n-1}) \rightarrow \mathcal{Q}(l)$ and morphisms $\gamma_n: \mathcal{P}(v_0) \otimes \cdots \otimes \mathcal{P}(v_{n-1}) \rightarrow \mathcal{P}(l)$ and $\nu_n: \mathcal{Q}(v_0) \otimes \cdots \otimes \mathcal{Q}(v_{n-1}) \rightarrow \mathcal{Q}(l)$. In the case where the tree $t$ has two vertices, the fact that $f$ is a morphism of dg cooperads implies:
$$
\nu_2 (f_{1} \otimes f_1) - f_1 \gamma_2= \partial (f_2)\ ,
$$
where $\partial (f_2)=d_{\mathcal{Q}} \  f_2 + f_2 \  d_{\mathcal{P}^{\otimes 2}} $. Since $\gamma_2$ and $\nu_2$ are interpreted as composite maps, $f_{1}$ commutes with these compositions up to homotopy; and this homotopy is precisely $f_2$. The other maps $f(t)$ for bigger trees $t$ are the data of a higher homotopical control. A morphism of strict unital homotopy colored operads commutes with the composite maps up to higher homotopies but strictly  with the units.

\begin{prop}\label{prop:Comparison}\leavevmode
\begin{itemize}
\item[$\triangleright$]
Nonunital homotopy colored operad concentrated in arity one is the same notion as nonunital $\mathcal{A}_{\infty}$-category. Strict unital homotopy colored operad concentrated in arity one is the same notion as  $\mathcal{A}_{\infty}$-category. 

\item[$\triangleright$] The forgetful functor from strict unital colored operads to nonunital colored operads has a left adjoint which is an embedding of category of nonunital homotopy colored operads into the category of strict unital homotopy colored operads.

\item[$\triangleright$]
The category $\mathsf{dg}\textsf{-}\mathsf{Op}$ of differential graded colored operads embeds canonically into the category $\mathsf{suOp}_{\infty}$ of strict unital homotopy colored operads.
\end{itemize}

\end{prop}

\begin{proof}
The proof of the first point is straightforward with the various definitions from  \cite{FOOO09I}. For the second point, any nonunital colored operad $(C,\mathcal{P},\gamma)$ is sent to $(C,\mathcal{P} \oplus I_C,\gamma)$. For the third point, any dg colored operad $\mathscr{P}=({C},\mathcal{P},\gamma, \eta)$ can be seen as the strict unital homotopy colored operad $({C},\mathcal{P}, \widetilde{\gamma}, \widetilde{\eta})$ as follows.  The structure map $\widetilde{\gamma}$
is defined by $\widetilde{\gamma}(t):=0$ for colored trees $t$ with more than $3$ vertices, by 
$$
\widetilde{\gamma}(t)(sa_0 \otimes sa_1):=(-1)^{|a_0|} s \gamma(a_0 \otimes a_1)
$$
for colored trees $t$ with  $2$ vertices, and by $\widetilde{\gamma}(t):= d_{s \mathcal{P}}=-d_{\mathcal{P}}$
for colored trees $t$ with  $1$ vertex. 
\end{proof}

 In this context, a morphism of dg colored operads $\mathscr{P} \rightarrow \mathscr{Q}$ is a morphism of strict unital homotopy colored operads $Rf$ such that the corresponding morphism of colored $\mathbb{S}$-modules $f:\overline{\mathbb{T}}(s\mathcal{P}) \rightarrow s\mathcal{Q}$ vanishes on the trees with two vertices or more.

$$
\xymatrix@M=8pt{
\mathsf{dg}\textsf{-}\mathsf{cat}   \ar@{^(->}[r] \ar@{^(->}[d] &  \mathsf{dg}\textsf{-}\mathsf{Op} \ar@{^(->}[d]  \\
\mathcal{A}_\infty\textsf{-cat} \ar@{^(->}[r] & \mathsf{su}\mathsf{Op}_\infty} 
$$


\section{The dendroidal nerve of strict unital homotopy colored operads}

In this section, we introduce the dendroidal nerve of strict unital homotopy colored operads; we then show that its image actually produces an infinity-operad. So it provides us with a new functor which relates these two notions. To compare this dendroidal nerve with the  existing constructions, we prove that it extends both Faonte--Lurie's simplicial nerve of strict unital $\mathcal{A}_\infty$-categories and Moerdijk--Weiss' homotopy coherent nerve of dg operads. Then, we characterize the morphisms of strict unital homotopy colored operads whose images under the dendroidal nerve are respectively weak equivalences and fibrations for the Cisinski--Moerdijk model structure. Finally, we endow the category of dg colored operads with a model structure introduced in \cite{Caviglia14} and show that the homotopy coherent nerve is a right Quillen functor.\\

From now on, \underline{the word ``colored'' will often be understood}. For instance, we  call ``strict unital homotopy colored operads'' simply by  \textit{``su homotopy operads''}.

\subsection{Trees as operads}

Let $\Lin$  be the functor which associates to any set the free $\mathbb{K}$-module on it
$$ 
\Lin \ : \  X \in \mathsf{Set} \ \mapsto \ \bigoplus_{x \in X} \mathbb{K}.x \in \mathsf{Mod} \subset \mathsf{gr}\textsf{-}\mathsf{Mod} \subset \mathsf{dg}\textsf{-}\mathsf{Mod}\ .
$$

This $\mathbb{K}$-module can be considered as a differential graded $\mathbb{K}$-module concentrated in degree zero with trivial differential. Therefore, we view the functor $\Lin$ as mapping into the category $\mathsf{dg}\textsf{-}\mathsf{Mod}$. This functor is a strong symmetric monoidal functor. So it can be extended to a functor from the category of colored operads on sets to the category of differential graded operads.\\

We denote by $\mathbb{K}\Omega(T):=\Lin(\Omega(T))$ the image under this functor of the set-theoretical operad $\Omega(T)$ (Definition \ref{deftreeop}). The colors of $\mathbb{K}\Omega(T)$ are the edges of $T$ and $\mathbb{K}\Omega(T)(a;\chi : X \rightarrow \mathrm{edges}(T))=\mathbb{K}$ when $\chi$ is injective and when there is a sub-tree (possibly trivial) of $T$ with root $a$ and leaves $\chi(X)$. Since $\mathbb{K}\Omega(T)$ is a dg operad, it  can be consider as a su homotopy operad. Furthermore, the map $T \in \mathsf{Tree} \mapsto \Omega(T) \in \mathsf{Op}$ defines a codendroidal object in the category of set-theoretical operads. Therefore, the map $T \mapsto \mathbb{K}\Omega(T)$ is a codendroidal object in the category $\mathsf{dg}\textsf{-}\mathsf{Op}$ of dg operads and so in the category $\mathsf{suOp}_{\infty}$ of su homotopy operads.

\subsection{The dendroidal nerve}\label{subsecthedendroidalnerve}

The usual nerve of a category is a functor which associates, to any small category $\mathsf{C}$, the simplicial set 
$$\mathrm{N} (\mathsf{C})_n:=\mathrm{Hom}_{\mathsf{Cat}}([n], \mathsf{C})\, $$ where $\mathsf{Cat}$ is the category of small categories and where $[n]$ is the poset $0< \cdots < n$ viewed as a category. More generally, the nerve of an object $X$ in a category $\mathsf{A}$ associated to a functor $F: \Delta \rightarrow \mathsf{A}$ is the simplicial set $\mathrm{N}(X)_n:=\mathrm{Hom}_{\mathsf{A}}(F([n]),X)$. In the same way, a dendroidal nerve of an object $X$ associated to a functor $F: \mathsf{Tree} \rightarrow \mathsf{A}$ is the dendroidal set $$\mathrm{N}^{\Omega}(X)_T:=\mathrm{Hom}_{\mathsf{A}}(F(T),X)\ .$$ We apply this construction to the functor $T \in \mathsf{Tree} \mapsto \mathbb{K}\Omega(T) \in \mathsf{suOp}_{\infty}$.

\begin{defin}[Dendroidal nerve]
The \textit{dendroidal nerve} of su homotopy operads $\mathrm{N}^{\Omega}$ is the functor from the category $\mathsf{suOp}_{\infty}$ to the category $\mathsf{dSet}$ of dendroidal sets defined by the following formula:
$$
\mathrm{N}^{\Omega}(\mathscr{P})_T:=\mathrm{Hom}_{\mathsf{suOp}_{\infty}}(\mathbb{K}\Omega(T), \mathscr{P})\ ,
$$
for any tree $T$ and any su homotopy operad $\mathscr{P}$.
\end{defin}

Let us describe the dendrices of the dendroidal nerve of a su homotopy operad $\mathscr{P}=({C},\mathcal{P}, \gamma, \eta)$. We denote by $d_\nu$ the structural codifferential on $\mathbb{T}^c(s\mathbb{K}\Omega(T))$, which comes from the operad structure on $\mathbb{K}\Omega(T)$.

\begin{lemma}
A $T$-dendrex $ \mathbb{K}\Omega(T) \rightsquigarrow \mathscr{P}$ is equivalent to the following  data:
\begin{itemize}
\item[$\triangleright$]
an underlying function $\phi$  from the set of edges of $T$ to the set of colors ${C}$,
\item[$\triangleright$]
 maps of graded $\mathbb{S}$-modules $f(t):t(s\mathbb{K}\Omega(T)) \rightarrow s\mathcal{P}$ over the function $\phi$ for any  tree $t=T'/T_1 \cdots T_k$, which is the contraction of a sub-tree $T'$ of $T$ along a partition $T'=T_1 \sqcup \ldots \sqcup T_k$ and which is canonically colored by the set of edges of $T$,
\end{itemize}
 satisfying the following equations, for the same class of trees $t$,

\begin{equation}
\displaystyle{\sum_{t=t_1 \sqcup \ldots \sqcup t_l} \gamma \big(t/t_1 \cdots t_l\big)\Big( f\big(t_1\big) \otimes \cdots \otimes f\big(t_l\big) \Big)  = f d_\nu}(t)\ ,
\label{adjoint4}
\end{equation}

where the sum runs over the partitions with no trivial component of the colored tree $t$. 
\end{lemma}

\begin{rmk}
As the trees $t=T'/T_1 \cdots T_k$ are canonically colored by the set of edges of $T$, the set of such colored trees is canonically bijective with the set of partitioned sub-trees of $T$ under the inverse of the mapping $T'=T_1 \sqcup  \ldots \sqcup T_k \mapsto T'/T_1 \cdots T_k$.
\end{rmk}

\begin{proof}
On the one hand, a $T$-dendrex is a morphism of su homotopy operads $ \mathbb{K}\Omega(T) \rightsquigarrow \mathscr{P}$, which can be described as a morphism of $\mathbb{S}$-modules $f:\overline{\mathbb{T}}(s \mathbb{K}\Omega(T)) \rightarrow s\mathcal{P}$ satisfying Relation~(\ref{eqn:MORPH}). In particular, it gives us an underlying function $\phi$ from the set of edges of $T$ to the set of colors $C$ of $\mathscr P$ and morphisms $f(t)$ for any colored tree $t=T'/T_1 \cdots T_k$ satisfying Relation~(\ref{adjoint4}). On the other hand, let us consider a function $\phi$ and morphisms $\{f(t)\}_{t=T'/T_1 \cdots T_k}$ satisfying Relation~(\ref{adjoint4}). Let us fix:
\begin{itemize}
\item[$\triangleright$] for any colored tree $t$ with one vertex $v$ and two edges having the same color $e \in \text{edges}(T)$,
\begin{align*}
f(t):t\big(s\mathbb{K}\Omega(T)\big) & \rightarrow s\mathcal{P}\\
sv & \mapsto s\ii_{\phi(e)}\ ;
\end{align*} 
\item[$\triangleright$] for any other tree $t$, colored by the edges of $T$, and which is different from a contraction of a sub-tree of $T$, $f(t):=0$.
\end{itemize}
The data of the function $\phi$ and the maps $\{f(t)\}_t$ amounts to a morphism $f:\overline{\mathbb{T}}(s \mathbb{K}\Omega(T)) \rightarrow s\mathcal{P}$. Furthermore, since the morphisms $\{f(t)\}_{t=T'/T_1 \cdots T_k}$ satisfy Relation~(\ref{adjoint4}), then the morphism $f$ satisfy Relation~(\ref{eqn:MORPH}) and so is a $T$-dendrex of $\mathscr P$.
\end{proof}

The definition of morphisms of su homotopy operads induces the following description of the images of dendrices under the face and degeneracy maps of the dendroidal nerve $\mathrm{N}^{\Omega}(\mathscr{P})$.
Let $x=(\phi; \{f(t)\}_{t=T'/T_1 \cdots T_k})$ be a $T$-dendrex of the dendroidal nerve $\mathrm{N}^{\Omega}(\mathscr{P})$. 
For any outer vertex $v$ of $T$, the outer face $\delta_v(x)$ of $x$ is given by the restriction of $\phi$ to the set $\mathrm{edges}(T)\setminus \mathrm{in}(v)$ and by the restriction of the family $\{f(t)\}_{t=T'/T_1\cdots T_k}$ to the contractions of the sub-trees  $T'$ of $T-\{v\}$. For any inner edge $e$ of $T$, the corresponding inner face $\delta_e(x)$ of $x$ is given by the restriction of $\phi$ to the set $\mathrm{edges}(T)\setminus \{e\}$ and by the restriction of the family $\{f(t)\}_{t=T'/T_1\cdots T_k}$ to the partitioned sub-trees  $T'$ of $T$ such that the edge $e$ is inside one of the trees $T_i$. Finally, let $e$ be an edge of $T$ and let $T_{\sigma}$ be the tree obtained from $T$ replacing $e$ by two edges $e_1$ and $e_2$ separated by a vertex $v$. 
There is a codegeneracy $s :  T_\sigma \to T$ sending $T_\sigma$ to $T$.
\begin{center}

\begin{tikzpicture}[scale=0.7]
\draw (2.5,0.5)--(2.5,1.5);
\draw (2.5,1) node[right] {$e$};
\draw (1.5,1) node {$\mapsto$};
\draw (0,0)--(0,2);
\draw (0,1) node[right] {$v$};
\draw (0,0.4) node[right] {$e_1$};
\draw (0,1.6) node[right] {$e_2$};
\draw (0,1) node {$\bullet$};
\draw (0,2.5) node[above] {$T_{\sigma}$};
\draw (1.5,2.9) node {$\to$};
\draw (2.5,2.6) node[above] {$T$};
\end{tikzpicture}

\end{center}
The corresponding degeneracy $\sigma(x)$ of $x$ is a $T_{\sigma}$-dendrex $(\phi_{\sigma};\{f_{\sigma}(t)\}_t)$ of $\mathrm{N}^{\Omega}(\mathscr{P})$ described as follows. On the one hand, we have $\phi_{\sigma}(e_1)=\phi_{\sigma}(e_2)=\phi(e)$ and $\phi_{\sigma}(a)=\phi(a)$ for the other edges $a$ of $T_{\sigma}$ which can be considered as edges of $T$. On the other hand, let  $t=T'/T_1\cdots T_k$ be a contracted sub-tree of $T_{\sigma}$.
\begin{itemize}
\item[$\triangleright$] If $k \geq 2$ and if one of the $T_i$ is the one-vertex tree made up of the vertex $v$ and the edges $e_1$ and $e_2$, then $f_{\sigma}(t)=0$.
\item[$\triangleright$] If $t$ is the one-vertex tree made up of the vertex $v$ and the edges $e_1$ and $e_2$, then $f_{\sigma}(t)$ is equal to:
\begin{align*}
f_{\sigma}(t)(e_1;e_2) :t\big(s\mathbb{K}\Omega(T_{\sigma})\big)(e_1;e_2) \simeq s\mathbb{K} & \rightarrow s\mathcal{P}(\phi(e);\phi(e))\\
sv & \mapsto s\ii_{\phi(e)}\ .
\end{align*}
\item[$\triangleright$] Otherwise, $f_{\sigma}(t)= f(t)$ since $t\big(s\mathbb{K}\Omega(T_{\sigma})\big) \simeq t\big(s\mathbb{K}\Omega(T)\big)$.
\end{itemize}

\subsection{The dendroidal nerve is an infinity-operad}

\begin{lemma}(\cite[Corollary~$3.2.7$]{MoerdijkToen10})\label{lem:Descr}
A morphism of dendroidal sets $\Lambda^e[T] \rightarrow D$ is the data of dendrices $x_v \in D_{T\backslash v}$ for any external vertex $v$ and $x_a \in D_{T/a}$ for any inner edge $a$ different from $e$, which agree on common faces.
\end{lemma}

This lemma applied to the case of the dendroidal nerve of a su homotopy operads gives the following description.

\begin{cor}\label{cor:Dendrex}
For any tree $T$ and any inner edge $e$, a morphism of dendroidal sets $\Lambda^e[T] \rightarrow \mathrm{N}^{\Omega}(\mathscr{P})$ is equivalent to the data of:
\begin{itemize}
\item[$\triangleright$] a function $\phi$ from the set of edges of $T$ to the set ${C}$ and

\item[$\triangleright$]  morphisms of  $\mathbb{S}$-modules $f(t):  t\big(s\mathbb{K}\Omega(T)\big)  \rightarrow s\mathcal{P}$ over $\phi$,   for every contracted colored  sub-tree $t=T'/T_1  \cdots T_k$  along a partition $T'=T_1 \sqcup \ldots \sqcup T_k$, except for the full tree $T$ with no contraction and the tree $T/e$ where only the two-vertices sub-tree spanned by the edge $e$ is contracted,
\end{itemize}
satisfying Equation~(\ref{adjoint4}) for each of these trees.
\end{cor}

\begin{proof}
The result is a direct corollary of Lemma~\ref{lem:Descr} and the description of faces given in the previous section.
\end{proof}

\begin{thm}\label{thminftyoperad}
The dendroidal nerve of a strict unital homotopy colored operad is an $\infty$-operad.
\end{thm}

\begin{proof}
Consider a morphism of dendroidal sets $f$ from $\Lambda^e[T]$ to $\mathrm{N}^{\Omega}(\mathscr{P})$ given by a function $\phi$ from the edges of $T$ to ${C}$ and morphisms of $\mathbb{S}$-modules $f(t):  t(s\mathbb{K}\Omega(T))  \rightarrow s\mathcal{P}$ over $\phi$ for the trees $t$ described in  Corollary~\ref{cor:Dendrex}. Recall that a morphism from $\Omega[T]=\mathrm{Hom}_{\mathsf{Tree}}(-,T)$ to $\mathrm{N}^{\Omega}(\mathcal{P})$ amounts to the data of a $T$-dendrex. So to extend $f$ to a morphism from $\Omega[T]$ to $\mathrm{N}^{\Omega}(\mathcal{P})$, we have to build $f(T)$ and $f(T/e)$ so that Equation~(\ref{adjoint4}) is fulfilled for these two trees. 
Let us recall that $d_\nu$ denotes  the structural codifferential on $\mathbb{T}^c(s\mathbb{K}\Omega(T))$. We fix $f(T):=0$. Then, because of Equation~(\ref{adjoint4}) for the tree $T$, the map $f(T/e)$ must satisfy the following formula: 
\begin{align*}
f(T/e) \Big(\id \otimes \cdots \otimes \nu(e) \otimes \cdots \otimes \id \Big)=& -\sum_{a\neq e} f(T/a) \Big(\id \otimes \cdots \otimes \nu(a) \otimes \cdots \otimes \id \Big) \\ &+ \sum_{T=T_1 \sqcup \ldots \sqcup T_k} \gamma(T/T_1 \cdots T_k)  \Big( f(T_1) \otimes \cdots \otimes f(T_k) \Big)  \ ,
\end{align*}

\noindent where the first sum runs over the inner edges of the tree $T$ different from $e$ and where the second runs over all the partitions of the tree $T$ with no trivial component. Since  $\id \otimes \cdots \otimes \nu(e) \otimes \cdots \otimes \id$ is an isomorphism of $\mathbb{S}$-modules, we have built $f(T/e)$. We know that Equation~(\ref{adjoint4}) is satisfied for every tree $t=T'/T_1 \cdots T_k$ which is the contraction of a  sub-tree $T'$ of $T$ along a partition $T'=T_1 \sqcup \ldots \sqcup T_k$ except for $t=T/e$. As in the proof of Proposition~\ref{key}, we have:
$$
(Rf d_\nu)(T)= \sum_{S \subset T \atop T=T_1 \sqcup \ldots \sqcup S \sqcup \ldots \sqcup T_k} f(T_1) \otimes \cdots \otimes (f  d_{\nu})(S) \otimes \cdots \otimes f(T_k) \ ,
$$
and
$$
(d_\gamma Rf)(T)= \sum_{S \subset T \atop T=T_1 \sqcup \ldots \sqcup S \sqcup \ldots \sqcup T_k} f(T_1) \otimes \cdots \otimes (\gamma Rf)(S) \otimes \cdots \otimes f(T_k)\ .
$$
Therefore, we have
$$
(d_\gamma Rf)(T)=(Rf  d_\nu)(T)\ ,
$$
and so
$$
(f d_\nu d_\nu)(T)=0=(\gamma d_\gamma Rf)(T)=  (\gamma Rf d_\nu)(T)\ .
$$
The above equation rewrites
$$
\sum_{a} (f d_\nu)(T/a) \Big(\id \otimes \cdots \otimes \nu(a) \otimes \cdots \otimes \id \Big) =\sum_{a} (\gamma Rf)(T/a) \Big(\id \otimes \cdots \otimes \nu(a) \otimes \cdots \otimes \id \Big) \ ,
$$
where the two sums run over the inner edges of the tree $T$.
We already know that $(f d_\nu)(T/a)=\gamma Rf(T/a)$ for all the inner edges $a$ different from $e$ and that $\id \otimes \cdots \otimes {\nu}(e) \otimes \cdots \otimes \id$ is an isomorphism. Therefore, we get
$$
f d_{\nu}(T/e) = \gamma  Rf(T/e)\ . 
$$
So $f  d_{\nu}=\gamma Rf$ and the morphism $f$ induces a $T$-dendrex of $\mathrm{N}^{\Omega}(\mathscr{P})$, which extends the initial morphism $\Lambda^e[T] \rightarrow \mathrm{N}^{\Omega}(\mathscr{P})$.
\end{proof}

So, the image of $\mathrm{N}^{\Omega}(-)$ lies in the category of $\infty$-operads. Therefore, we can consider  it  as  a functor from the category of su homotopy operads to the category of $\infty$-operads: 
$$
\mathrm{N}^{\Omega}: \mathsf{suOp}_{\infty} \rightarrow \infty\textsf{-}\mathsf{Op}\ .
$$

Recall from Proposition~\ref{prop:Comparison}, that strict unital $A_{\infty}$-categories are the su homotopy operads concentrated in arity one. G. Faonte already defined in \cite{Faonte13} a simplicial nerve  $\mathrm{N}_{\mathcal{A}_{\infty}}$ for strict unital $A_{\infty}$-categories, generalizing a first construction of J. Lurie \cite{Lurie12}.
The present dendroidal nerve is actually a generalization of Faonte's simplicial nerve. 

\begin{prop}
The simplicial part of the restriction to strict unital $A_{\infty}$-categories of the dendroidal nerve is equal to Faonte's simplicial nerve:
$$\pi\left( {\mathrm{N}^{\Omega}}_{|\mathcal{A}_\infty\textsf{-}\mathsf{cat}}\right)=\mathrm{N}_{\mathcal{A}_{\infty}}\ , $$
where $\pi$ be the restriction of dendroidal sets onto simplicial sets.
\end{prop}

Finally, we show that the dendroidal nerve forgets the information contained in nonnegative degrees.

\begin{defin}
Let $tr$ be the truncation endofunctor of the category $\mathsf{dg}\text{-}\mathsf{Mod}$ which sends a chain complex of $\mathbb{K}$-module $V$ to the bounded below chain complex of $\mathbb{K}$-module $tr(V)$ defined as follows:
\begin{itemize}
\item[$\triangleright$] for any negative integer $n<0$, $tr(V)_n:=\{0\}$ 
\item[$\triangleright$] for any negative integer $n<0$, $tr(V)_n:=V_n$
\item[$\triangleright$] at degree $0$, $tr(V)_0 :=\mathrm{ker}(d:V_0 \rightarrow V_{-1})$ is the kernel of the differential. 
\end{itemize}
\end{defin}

This functor extends to an endofunctor also denoted by $tr$ of the category of su homotopy operads sending $\mathscr P = (C,\mathcal{P},\nu, \eta)$ to $(C,tr(\mathcal{P}),tr\nu, \eta)$.

\begin{prop}
The dendroidal nerve is equal to its pre-composition with the truncation functor.
$$
\mathrm{N}^{\Omega}tr=\mathrm{N}^{\Omega}
$$
\end{prop}

\begin{proof}
The proof is a straightforward consequence of the description of the nerve in Section \ref{subsecthedendroidalnerve}.
\end{proof}

\subsection{The Boardman--Vogt construction}\label{sectionboardmanvogt} We recall here the Boardman--Vogt construction for the dg operads $\mathbb{K}\Omega(T)$. For more details, we refer the reader to the original papers \cite{BergerMoerdijk07, BergerMoerdijk06, Weiss07}.

\begin{defin}[An interval in the category of chain complexes]\label{def:intervalH}
Let $H$ be the chain complex made up of two generators $h_0$ and $h_1$ in degree $0$ and one generator $h$ in degree $1$ such that the differential $d(h)$  is equal to $h_1 - h_0$. It is equipped with a symmetric product $\vee: H \otimes H \rightarrow H$ such that $h_0$ is a unit, $h_1$ is idempotent, $h$ is nilpotent, and such that $h \vee h_1=0$; it is also equipped with a map $\epsilon: H \rightarrow \mathbb{K}$ such that $\epsilon(h_i)=1_{\mathbb{K}}$ and $\epsilon(h)=0$.
\end{defin}

\begin{defin}[The Boardman--Vogt construction]\label{defin:boardman--vogtresolution}
For any tree $T$, the \textit{Boardman--Vogt construction}  $W_H(T)$ of the operad $\mathbb{K}\Omega(T)$ is the operad with the same colors (the edges of the tree $T$) and made up of colored trees whose vertices are labeled by elements of $\mathbb{K}\Omega(T)$ and whose inner edges are labeled by elements of the interval $H$. This is subject to the two following identifications.
\begin{itemize}
\item[$\triangleright$] If a vertex $v$ with one input is labeled by an identity, then the tree is identified with the same tree with the vertex $v$ removed and the two adjacent edges glued together. If the resulting edge is inner, then it is labeled  by the element of $H$ given by the product of the two elements labeling the former adjacent edges. And if the resulting edge is outer, then the tree is multiplied by the image under $\epsilon$ of the former inner adjacent edge.
\item[$\triangleright$] If an inner edge $e$ is labeled by $h_0$, then
the tree is identified with its contraction along this edge. The resulting vertex is labeled by the composition in the operad $\mathbb{K}\Omega(T)$ of the labelings of the two former adjacent vertices.
\end{itemize}
The operadic composition is given by the grafting of trees where the new inner edge is labeled by $h_1$.
\end{defin}

As an operad in the category of graded $\mathbb{K}$-modules, $W_H(T)$ is the free operad over the following $\mathbb{S}$-module $V_H(T)$.  For $T'$ a sub-tree of $T$ with $c_1, \ldots, c_n$ as inputs and $c_0$ as output, the component of the $\mathbb{S}$-module $V_H(T)$ is
$$
V_H(T)(c_1,\ldots ,c_n ; c_0):=\bigoplus_{t=T'/T_1 \cdots T_k} \mathbb{K} \, e_t\ , 
$$
where the sum is taken over the trees $t$  obtained from $T'$ by contraction with respect to one of its partition $T_1, \ldots, T_k$. The generator $e_t$ corresponds to labeling the edges of $T'$ by $h_0$ if they are inner in one of the $T_i$ and by $h$ otherwise. Its degree $|e_t|$ is equal to the number of inner edges of the tree $t$. Otherwise, we set $V_H(T)(c_0;c_1,\ldots ,c_n):= \{0\}$. The differential is given by the formula:
\begin{equation}\label{pm}
d(e_t)=\sum_a \pm \gamma(e_{ta} \otimes e_{at}) \pm e_{t/a} \ , 
\end{equation}
\noindent where the sum is taken over the inner edges $a$ of the tree $t$,  where $at$ (resp. $ta$) is the sub-tree of $t$  under (resp. above) the edge $a$, and where $\gamma(e_{ta} \otimes e_{at})$ is the composite of $e_{at}$ and $e_{ta}$ in the free operad $W_H(T)$. Furthermore, the tree $t/a$ is the contraction of $t$ along the edge $a$. 

\begin{rmk}
The signs are produced by the signed permutations of edges. This can be done coherently as follows: we first choose a planar representative of the whole tree $T$. Then, we order the inner edges of $T$ from bottom to top and from left to right. This induces a total ordering on the inner edges of $t$, thereby numbered $1$ to $|e_t|$. In this context, the first sign $\pm$ appearing in Equation~(\ref{pm}) is $(-1)^{(a-1+|e_{ta}|-a+1)|e_{at}|}$ and the second sign is $(-1)^a$.
\end{rmk}

The Boardman--Vogt construction $W_H(T)$ is functorial in $T$ as follows. 
An outer coface $T-\{v\} \rightarrow T$ (resp. an inner coface $T/e \rightarrow T$) induces canonically an injection between operads $W_H(T-\{v\}) \rightarrow W_H(T)$ (resp. $W_H(T/e) \rightarrow W_H(T)$). A codegeneracy $T \rightarrow T/v$ induces a morphism of operads which sends an element with underlying tree $T' \subseteq T$ to the similar  element with underlying  tree $T'/v \subseteq T/v$, where the inner (resp. outer) edge resulting from the removal of $v$ is labeled by the element of $H$ given by the product of the two elements labeling the adjacent edges of $v$ (resp. is labeled by the image under $\epsilon$ of the inner adjacent edge of $v$). These morphisms satisfy the same equations as the cofaces and codegeneracies. Therefore we have defined a codendroidal operad $T \mapsto W_H(T)$.

\subsection{The dendroidal nerve for operads is the homotopy coherent nerve}

\begin{defin}[The homotopy coherent nerve]
The \textit{homotopy coherent nerve} of a dg operad $\mathscr{P}$ is defined by the following dendroidal set
$$
\mathrm{hcN}(\mathscr{P})_T:=\mathrm{Hom}_{\mathsf{dg}\textsf{-}\mathsf{Op}}(W_H(T),\mathscr{P}) \ .
$$
\end{defin}

For dg operads, this construction is equal to the dendroidal nerve.

\begin{thm}
There is a canonical isomorphism
$$
\mathrm{hcN}(\mathscr{P}) \simeq \mathrm{N}^{\Omega}(\mathscr{P}) \ ,
$$
which is natural in dg colored operads $\mathscr{P}$.
\end{thm}

\begin{proof}
Let us unfold what a $T$-dendrex of the homotopy coherent nerve $\mathrm{hcN}(\mathscr{P})$ of a dg operad $\mathscr{P}=(C,\mathcal{P},\gamma,\eta)$ is: a morphism of differential graded operads from $W_H(T)$ to $\mathscr{P}$. In particular, it is a morphism of operads on graded $\mathbb{K}$-modules and, thus, is a morphism of $\mathbb{S}$-modules from $V_H(T)$ to $\mathcal{P}$, i.e. the data of a function $\phi$ from the set of edges of $T$ to the set $C$ and for any contraction $t=T'/T_1 \cdots T_k $ of a sub-tree $T'$ of $T$, an image of the generator $e_t$. Each one of these images induces a morphism:
$$
f(t):t(s \mathbb{K}\Omega(T)) \rightarrow s \mathcal{P} \ .
$$
These must be consistent with the fact that the $T$-dendrex is a morphism of differential graded operads. This condition amounts exactly the fact that the morphisms $\{f(t)\}_t$ satisfy Equation $(\ref{adjoint4})$. We have thus built a bijection from $\mathrm{hcN}(\mathscr{P})_T$ to $\mathrm{N}^{\Omega}(\mathscr{P})_T$, for every tree $T$. Moreover, these bijections are functorial with respect to the trees $T\in \mathsf{Tree}$. This concludes the proof. 

\end{proof}

\begin{rmk}
This result is related to with the bar-cobar adjunction. Indeed, according to \cite{BergerMoerdijk06}, the Boardman--Vogt construction for augmented reduced dg operads is equal to the bar-cobar construction; and morphisms (not necessarily augmented) of dg operads from the bar-cobar construction on $\mathbb{K}\Omega(T)$ are equivalent to morphisms of su homotopy operads from $\mathbb{K}\Omega(T)$, see \cite[Section~$10.5.5$]{LodayVallette12}.
\end{rmk}

\begin{rmk}\label{rmkfreesmodule}
Let $T$ be a tree. Let $t=T'/T_1 \cdots T_k$ be a contration of a subtree of $T$, canonically colored by the set of edges of $T$. Let $l'_1$, \ldots, $l'_m$; be the leaves of $T'$ and let $r'$ be its root. In the last proof, we implicitely use the fact that a graded morphism of $\mathbb{S}$-module $t(s\mathbb{K}\Omega(T)) \rightarrow (C,\mathcal{V})$ over a function $\phi$ is equivalent to the data of an element of $\mathcal{V}(\phi(l'_1),\ldots,\phi(l'_m) ; \phi(r'))$ whose degree is the number of vertices of the colored tree $t$.
\end{rmk}

\subsection{Homotopical properties of the dendroidal nerve} In this section, we explore the homotopical properties of the dendroidal nerve $\mathrm{N}^{\Omega}$. More precisely, we give a description of morphisms of su homotopy colored whose image under $\mathrm{N}^{\Omega}$ are weak equivalences (resp. fibrations) in the Cisinski--Moerdijk model structure. Then we relate these results to the model structure on nonunital $\mathcal{A}_{\infty}$-algebras existing in \cite[Theorem  1.3.3.1]{LefevreHasegawa03}. More precisely, we show that the simplicial nerve of $\mathcal{A}_{\infty}$-categories $\mathrm{N}_{\mathcal{A}_{\infty}}$ sends weak equivalences (resp. fibrations) of $\mathcal{A}_{\infty}$-algebras to weak equivalences (resp. fibrations) of the Joyal model structure on simplicial sets.

\begin{lemma}
The zero-homology-group functor $H_0: \mathsf{dg}\text{-}\mathsf{Mod} \rightarrow \mathsf{Mod}$ induces a functor also denoted $H_0$ from the category $\mathsf{suOp}_{\infty}$ to the category $\mathsf{Op}$ of colored operads enriched in $\mathbb{K}$-modules. Forgetting the many-inputs-elements, we get a functor $j^*H_0$ from su homotopy operads to categories enriched in $\mathbb{K}$-modules.
\end{lemma}

\begin{proof}
Straightforward.
\end{proof}

\begin{thm}\label{thmweakequivalence}
Let $f: \mathscr P = (C,\mathcal{P},\nu) \rightarrow \mathscr Q = (C,\mathcal{Q},\omega)$ be a morphism of su homotopy operads. The following assertions are equivalent.
\begin{enumerate}

\item \label{1} The morphism of $\infty$-operads $\mathrm{N}^{\Omega}(f)$ is a weak equivalence.
\item \label{2} The functor $j^*H_0(f)$ is an equivalence of categories and for any colors $c_1$, \ldots, $c_m$ and $c$, the first level morphism of chain complex of $\mathbb{K}$-modules $f: s\mathcal{P}(c_1,\ldots,c_m ; c) \rightarrow s\mathcal{Q}(\phi(c_1),\ldots,\allowbreak \phi(c_m) ; \phi(c))$ induces isomorphisms of homology groups of positive degrees.

\end{enumerate} 
\end{thm}

The proof of Theorem \ref{thmweakequivalence} requires the following lemmata.

\begin{lemma}\label{lemma:normalizeschaincomplex}
Let $\mathscr{P}$ be a su homotopy operad. For any colors $c$, $c_1$, \ldots, $c_m$, the simplicial set $\mathrm{N}^{\Omega}(\mathscr{P})^L(c_1,\ldots,c_m ; c)$ is a simplicial $\mathbb{K}$-module and its normalized chain complex (see \cite[3.2]{GoerssJardine09}) is isomorphic the chain complex $tr(\mathcal{P}(c_1,\ldots,c_m ; c))$.
\end{lemma}

\begin{proof}[Proof of Lemma \ref{lemma:normalizeschaincomplex}]
Let us unfold what is the simplicial set $X:=\mathrm{N}^{\Omega}(\mathscr{P})^L(c_1,\ldots,c_m;c)$. An $n$-simplex of $X$ is the data of a morphism of su homotopy operads over $(c_1,\ldots,c_m,c)$ from $\mathbb{K}\Omega(C_{m,n})$ to $\mathscr P$  and whose restriction to $\Delta[n]$ is the degeneracy of the color $c$. It is then the data of maps of graded $\mathbb{S}$-modules $t(s\mathbb{K}\Omega(C_{m,n})) \rightarrow s\mathcal{P}$ over $(c_1,\ldots,c_m,c)$ for any colored tree $t$ obtained from $C_{m,n}$ by contracting a subtree. As the restriction of the morphism of su homotopy operads to $[n]$ is given by $c$, we can restrict our attention to such colored trees $t$ which contain the corolla $C_m$. Let us number the edges of $[n] \subset C_{m,n}$ from bottom to top and by $0$ to $n$. Then, there is a bijection between the set of such contracted trees $t$ and the set of sequences of integers $0\leq i_0 < \cdots < i_k \leq n$ given by the numbers of the edges of $[n] \subset C_{m,n}$ which still appear in $t$. Therefore, according to Remark \ref{rmkfreesmodule} the map $t(s\mathbb{K}\Omega(C_{m,n})) \rightarrow s\mathcal{P}$ corresponds to an element $p_{i_0< \cdots <i_k} \in \mathcal{P}(c_1,\ldots,c_m;c)$ whose degree is the number of vertices of $t$ minus $1$, that is $k$. Relation (\ref{adjoint4}) applied to these elements give the following equation:
$$
d_{\mathcal{P}} (p_{i_0< \cdots <i_k})=\sum_{\alpha=0}^k (-1)^{\alpha} p_{i_0< \cdots <\widehat{i_\alpha} < \cdots < i_k}
$$
where $i_0< \cdots <\widehat{i_\alpha} < \cdots < i_k$ is obtained from the sequence $i_0< \cdots < i_k$ by retrieving the integer $i_\alpha$. A face $d_i$ corresponding to a coface $\delta_i$ sends the collection $\{p_{i_0< \cdots <i_k}\}_{0 \leq i_0< \cdots <i_k \leq n}$ to the collection $\{p_{ \delta_i(j_0)< \cdots <\delta_i (j_k)} \}_{0\leq j_0< \cdots <j_k \leq n-1}$. A degeneracy $s_i$ corresponding to a codegeneracy $\sigma_i$ sends the collection $\{p_{i_0< \cdots <i_k}\}_{0 \leq i_0< \cdots <i_k \leq n}$ to the collection $(p'_{j_0< \cdots <j_k})_{0 \leq j_0< \cdots <j_k \leq n+1}$ where $p'_{j_0< \cdots <j_k}=p_{\sigma_i(j_0) < \cdots <\sigma_i(j_k)}$ if we have indeed $\sigma(j_0)_i < \cdots <\sigma_i(j_k)$ and $p'_{(j_0< \cdots <j_k)}=0$ otherwise. Then, it is clear that $X=\mathrm{N}^{\Omega}(\mathscr{P})^L(c_1,\ldots,c_m;c)$ is a simplicial $\mathbb{K}$-module. Its normalized chain complex is the chain complex $N(X)$ concentrated in non-negative degrees such that $N(X)_n=\bigcap_{i=0}^{n-1} ker(d_i)$ and the differential $d:N(X)_n \rightarrow N(X)_{n-1}$ is $(-1)^n d_n$ where $d_n$ is the $n^{th}$ face map. This shows finally that $tr(\mathcal{P}(c_1,\ldots,c_m;c))$ is isomorphic to $N(X)$.
\end{proof}

\begin{lemma}\label{lemmacatlevel}
The functor $\tau_d \mathrm{N}^{\Omega}$ from su homotopy operads to set--theoretical colored operads is isomorphic to the functor $H_0$.
\end{lemma}

\begin{proof}[Proof of Lemma \ref{lemmacatlevel}]
We know from Theorem \ref{thminftyoperad} that for any su homotopy operad $\mathscr P$, $\mathrm{N}^{\Omega}(\mathscr P)$ is an $\infty$-operad. Thus Section 3.5 of \cite{Weiss07} gives us a concrete description of the operad $\tau_d\mathrm{N}^{\Omega}(\mathscr P)$: it has the same colors as $\mathscr P$ and for any such colors $c_1$, \ldots, $c_m$ and $c$, the set $\tau_d\mathrm{N}^{\Omega}(\mathscr P)(c_1,\ldots,c_m ; c)$ is the $\pi_0$ of the simplicial set $\mathrm{N}^{\Omega}(\mathscr P)(c_1,\ldots,c_m ; c)$. A straightforward computation shows that the operads $\tau_d\mathrm{N}^{\Omega}(\mathscr P)$ and $H_0(\mathscr P)$ are canonically isomorphic.
\end{proof}

\begin{proof}[Proof of Theorem \ref{thmweakequivalence}]
The theorem is a straightforward consequence of Lemma \ref{lemma:normalizeschaincomplex} and Lemma \ref{lemmacatlevel}.
\end{proof}

\begin{thm}\label{thmfibrations}
Let $f: \mathscr P = (C,\mathcal{P},\nu) \rightarrow \mathscr Q = (C,\mathcal{Q},\omega)$ be a morphism of su homotopy operads. The following assertions are equivalent.
\begin{enumerate}

\item \label{1} The morphism of $\infty$-operads $\mathrm{N}^{\Omega}(f)$ is a fibration.
\item \label{2} The functor $j^*H_0(f)$ is a categorical fibration (also called isofibraton) and for any colors $c_1$, \ldots, $c_m$ and $c$, the first level morphism of chain complex of $\mathbb{K}$-modules $f: s\mathcal{P}(c_1,\ldots,c_m ; c) \rightarrow s\mathcal{Q}(\phi(c_1),\ldots,\phi(c_m) ; \phi(c))$ is a degreewise epimorphism for degrees $n \geq 2$.

\end{enumerate} 
\end{thm}

\begin{proof}
As usual, we denote by $\phi$ the underlying function of $f$.
\begin{itemize}

\item[$(2) \Rightarrow (1)$] By Lemma \ref{lemmacatlevel} we know that $i^*\tau_d\mathrm{N}^{\Omega}(f)$ is a categorical fibration. Now, suppose that we have the following commutative diagram
$$
\xymatrix@R=25pt@C=25pt{\Lambda^e[T] \ar[r] \ar[d] & \mathrm{N}^{\Omega}(\mathscr P) \ar[d]^{\mathrm{N}^{\Omega}(f)} \\
\Omega[T] \ar[r] & \mathrm{N}^{\Omega}(\mathscr Q)}
$$
where $T$ is a tree with a root $r$ and leaves $l_1$, \ldots, $l_k$. This diagram corresponds to morphisms of graded $\mathbb{S}$-modules:
\begin{itemize}
\item[$\triangleright$] on the one hand, $h(t): t(s\mathbb{K}\Omega(T)) \rightarrow s\mathcal{Q}$ for any tree $t$ colored by the edges of $T$ and obtained by contracting a sub-tree of $T$.
\item[$\triangleright$] on the other hand, $g(t): t(s\mathbb{K}\Omega(T)) \rightarrow s\mathcal{P}$ for the same colored trees $t$ except for $t=T$ and $t=T/e$.
\end{itemize}
The underlying function from the set of edges of $T$ to $C$ (resp. $C'$) of the graded morphisms $g(t)$ (resp. $h(t)$) is denoted $\phi_g$ (resp. $\phi_h$). We have $\phi \phi_g=\phi_h$. Let $\tilde{h}(T): T(s\mathbb{K}\Omega(T)) \rightarrow s\mathcal{Q}$ be the following composition of graded morphisms:
$$
\tilde{h}(T):=h(T)- \sum_{T=T_1 \sqcup \ldots \sqcup T_k \atop k \geq 2} f(T/T_1 \cdots T_k)\Big( g(T_1) \otimes \cdots \otimes g(T_k) \Big) \ .
$$
Let $x$ be a generator of $T(s\mathbb{K}\Omega(T))(l_1,\ldots,l_k ; r)$. Its degree is the number of vertices of $T$ and so is equal or higher than $2$. The surjectivity condition for $f$ ensures us that $\tilde{h}(T)(x)$ has an antecedent through the first level map $f:s\mathcal{P}(\phi_g(l_1),\ldots,\phi_g(l_k) ; \phi_g(r)) \rightarrow s\mathcal{Q}(\phi_h(l_1),\ldots,\phi_h(l_k) ; \phi_h(r))$. By Remark \ref{rmkfreesmodule}, this antecedent corresponds to a map of graded $\mathbb{S}$-modules $g(T): T(s\mathbb{K}\Omega(T)) \rightarrow s\mathcal{P}$ over $\phi_g$. Furthermore, we have $\tilde{h}(T)=f g(T)$ and therefore, we have $h(T)=fRg(T)$.\\

As in the proof of Theorem \ref{thminftyoperad}, let us define $g(T/e)$ by
\begin{align*}
g(T/e) \Big(\id \otimes \cdots \otimes \nu(e) \otimes \cdots \otimes \id \Big) = & -\sum_{a \neq  e} g(T/a) \Big(\id \otimes \cdots \otimes \nu(a) \otimes \cdots \otimes \id \Big)\\ 
&+\sum_{T=T_1 \sqcup \ldots \sqcup T_k} \gamma(T/T_1 \cdots T_k)  \Big( g(T_1) \otimes \cdots \otimes g(T_k) \Big)\ ,
\end{align*}
\noindent where the first sum runs over the inner edges of the tree $T$ different from $e$ and where the second sum runs over all the partitions of the tree $T$ with no trivial component. The same method as in the proof of Theorem \ref{thminftyoperad} shows that the data of the maps $g(t)$ defines a $T$-dendrex $p$ of $\mathrm{N}^{\Omega}(\mathscr P)$. Let us prove that $\mathrm{N}^{\Omega}(f)(p)=q$ where $q$ is the $T$-dendrex of $\mathrm{N}^{\Omega}(\mathscr Q)$ corresponding to $h$. This amounts to prove that $h(t)=fRg(t)$ for any colored tree obtained by contracting a subtree of $T$. We already know that this is true for every such $t$ except $t=T/e$. In the spirit of Proposition \ref{key}, we prove that $RfRg(T)=Rh(T)$ for any such tree except $t=T/e$. We consider the following commutative diagram:
$$
\xymatrix@R=35pt@C=25pt{T(s\mathbb{K}\Omega(T)) \ar[r]_{Rg(T)} \ar@/^2pc/[rr]^{Rh(T)} \ar[d]_{d_{\gamma}} & \mathbb{T}(s\mathcal{P}) \ar[d]^{d_{\nu}} \ar[r]^{Rf} & \mathbb{T}(s\mathcal{Q})\ar[d]^{\omega}\\
\mathbb{T}(s\mathbb{K}\Omega(T)) \ar[r]_{Rg} & \mathbb{T}(s\mathcal{P}) \ar[r]_f & \mathcal{Q}}
$$
where $d_{\gamma}$ is the structural coderivation of the cooperad $\mathbb{T}^c(s\mathbb{K}\Omega(T))$. As $\omega Rh= h d_{\gamma}$, we have $h d_{\gamma}(T)=\omega Rh(T)=\omega RfRg(T)= f d_{\nu} Rg(T)= f Rg d_{\gamma}(T)$. Then, applying the same procedure as in the proof of Theorem \ref{thminftyoperad}, we get that $h(T/e)=fRg(T/e)$. This proves that $h=fRg$ and so that the square above has a lifting. So, by the characterization of the fibrations between fibrant dendroidal sets given in Theorem \ref{thmcisinskimoerdijk}, $\mathrm{N}^{\Omega}(f)$ is a fibration.

\item[$(1) \Rightarrow (2)$] By Lemma \ref{lemmacatlevel}, we know that $j^*H_0(f)$ is a categorical fibration. Furthermore, let   $n \geq 1$ be an integer and let $q \in \mathcal{Q}(\phi(c_1),\ldots,\phi(c_m);\phi(c))_n$ be an element of degree $n$ of $\mathscr{Q}$. We know that the tree $C_{m,n}$ is made up of a corolla with $m$ leaves above a linear tree of length $n$. Let $e$ be the lowest inner edge of $C_{m,n}$. For any tree $t$ colored by the edges of $C_{m,n}$ and which is a contraction of a subtree of $C_{m,n}$, let $sq_t$ be the element of $s\mathcal{Q}$ as follows.
\begin{itemize}
\item[$\triangleright$] If $t$ is the whole tree $C_{m,n}$, then $sq_t:=sq \in s\mathcal{Q}(\phi(c_1),\ldots,\phi(c_m);\phi(c))_{n+1}\ .$ 
\item[$\triangleright$] If $t$ is the contracted tree $C_{m,n}/e$, then $sq_t = -sdq \in s\mathcal{Q}(\phi(c_1),\ldots,\phi(c_m);\phi(c))_{n}\ .$
\item[$\triangleright$] If $t$ has only one vertex and does not contain the corolla $C_m$, then $sq_t= s\mathtt{id}_{\phi(c)} \in s\mathcal{Q}(\phi(c);\phi(c))_{1}\ .$
\item[$\triangleright$] Otherwise, $sq_t = 0 \in s\mathcal{Q}(\phi(c_1),\ldots,\phi(c_m);\phi(c))$ if $t$ contains the corolla and $sq_t = 0 \in s\mathcal{Q}(\phi(c);\phi(c))$ if not.
\end{itemize}
\noindent According to Section \ref{subsecthedendroidalnerve}, these elements $sq_t$ define a $C_{m,n}$-dendrex of the dendroidal set $\mathrm{N}^\Omega(\mathscr Q)$, i.e. a morphism $\Omega[C_{m,n}] \rightarrow \mathrm{N}^\Omega(\mathscr Q)$. For any tree $t$ colored by the edges of $C_{m,n}$ and which is a contraction of a subtree of $C_{m,n}$, except for $t=C_{m,n}$ and $t=C_{m,n}/e$, let $sp_t$ be the element of $s\mathcal{P}$ as follows.
\begin{itemize}
\item[$\triangleright$] If $t$ has only one vertex and does not contain the corolla $C_m$, then $sp_t= s\mathtt{id}_{c} \in s\mathcal{P}(c;c)_{1}\ .$
\item[$\triangleright$] Otherwise, $sp_t = 0 \in s\mathcal{P}(c_1,\ldots,c_m;c)$ if $t$ contains the corolla and $sp_t = 0 \in s\mathcal{P}(c;c)$ if not.
\end{itemize} 
According to Corollary \ref{cor:Dendrex}, these elements $p_t$ define a morphism of dendroidal sets from $\Lambda^e[C_{m,n}]$ to $\mathrm{N}^{\Omega}(\mathscr P)$. The two morphisms of dendroidal sets that we have built fit in the following commutative square of dendroidal sets.
$$
\xymatrix@R=25pt@C=25pt{\Lambda^e[T] \ar[r] \ar[d] & \mathrm{N}^{\Omega}(\mathscr P) \ar[d]^{\mathrm{N}^{\Omega}(f)} \\
\Omega[T] \ar[r] & \mathrm{N}^{\Omega}(\mathscr Q)}
$$
As $\mathrm{N}^{\Omega}(f)$ is a fibration, the square has a lifting. This provides an element $sp \in s\mathcal{P}(c_1,\ldots,c_m;c)$ such that $f(sp)=sq$. 
\end{itemize}
\end{proof}

Recall from \cite[Theorem  1.3.3.1]{LefevreHasegawa03} that, if $\mathbb{K}$ is a field, the category of nonunital $\mathcal{A}_{\infty}$-algebras admits a model structure without limits where
\begin{itemize}
\item[$\triangleright$] the weak equivalences are the morphisms $f: \mathscr A=(\mathcal{A},\gamma) \rightarrow \mathscr A'=(\mathcal{A}',\gamma')$ such that the first level map $f_1 : s\mathcal{A} \rightarrow s\mathcal{A}'$ is a quasi-isomorphism.
\item[$\triangleright$] the fibrations are the morphisms $f: \mathscr A=(\mathcal{A},\gamma) \rightarrow \mathscr A'=(\mathcal{A}',\gamma')$ such that the first level map $f_1:s\mathcal{A} \rightarrow s\mathcal{A}'$ is a degreewise epimorphism.
\item[$\triangleright$] the cofibrations are the morphisms $f: \mathscr A=(\mathcal{A},\gamma) \rightarrow \mathscr A'=(\mathcal{A}',\gamma')$ such that the first level map $f_1:s\mathcal{A} \rightarrow s\mathcal{A}'$ is a degreewise monomorphism.
\end{itemize}
We know from Section \ref{subscteionhomotopyoperads} that the category of nonunital $\mathcal{A}_{\infty}$-algebras is embedded in the category of strict unital $\mathcal{A}_{\infty}$-categories. Then, Theorem \ref{thmweakequivalence} and Theorem \ref{thmfibrations} have the following consequence.

\begin{cor}\leavevmode
The simplicial nerve $\mathrm{N}_{\mathcal{A}_{\infty}}$ sends weak equivalences (resp. fibrations) of nonunital $\mathcal{A}_{\infty}$-algebras to weak equivalences (resp. fibrations) of simplicial sets for the Joyal model structure. 
\end{cor}

\begin{proof}
This result is a straightforward consequence of Theorem \ref{thmweakequivalence} and Theorem \ref{thmfibrations}.
\end{proof}

To extend such a result to the operadic level, one would need a homotopy theory of homotopy operads. This will be the subject of another paper.

\subsection{The homotopy coherent nerve is a right Quillen functor} 
There is a pair of adjoint functors
\begin{equation}\label{adj:hcN}
\xymatrix{\mathsf{dg}\text{-}\mathsf{Op} \ar@<1ex>[r]^{\mathrm{hcN}} & \mathsf{dSet} \ar@<1ex>[l]^{W_!^{dg}}}
\end{equation}
where the functor $W_!^{dg}$,  left adjoint to the homotopy coherent nerve, is constructed as follows. For any tree $T$, we set $W_!^{dg}(\Omega[T]):=W_H(T)$ and then, since a dendroidal set is a colimit of a diagram made up of trees, the image of a dendroidal set is the corresponding colimit of the diagram made up of the images of the trees. By definition, this functor preserves colimits. In this section, we will show that this adjunction is a Quillen adjunction with respect to the model category structure on dg operads introduced in \cite{Caviglia14}, when the characteristic of the field $\mathbb{K}$ is $0$.

\begin{prop}[\cite{Caviglia14} Theorem 4.22 and Proposition 5.3]\label{prop:caviglia}
Assume that $\mathbb{K}$ is a characteristic $0$ field. The category $\mathsf{dg}\text{-}\mathsf{Op}$ admits a right proper cofibrantly generated model structure where
\begin{itemize}
\item[$\triangleright$] the weak equivalences are the morphisms $f: \mathscr P \rightarrow \mathscr Q$ such that $j^*H_0(f)$ is an essentially surjective functor and such that the morphism of chain complexes $f: \mathcal{P}(c_1,\ldots,c_m ; c) \rightarrow \mathcal{Q}(\phi(c_1),\ldots,\phi(c_m) ; \phi(c))$ is a quasi-isomorphism, for any integer $m \geq0$ and for any colors $c_1$, \ldots, $c_m$ and $c$, where $\phi$ is the function underlying $f$.
\item[$\triangleright$] the fibrations are the morphisms $f: \mathscr P \rightarrow \mathscr Q$ such that the functor $j^*H_0(f)$ is an isofibration and the morphism of chain complexes $f: \mathcal{P}(c_1,\ldots,c_m ; c) \rightarrow \mathcal{Q}(\phi(c_1),\ldots,\phi(c_m) ; \phi(c))$ is a degreewise epimorphism, for any colors $c_1$, \ldots, $c_m$ and $c$.
\end{itemize}
\end{prop}

\begin{rmk} The model structure given here may seem different from the definition given by Caviglia. It is not the case. In fact:
\begin{itemize}
\item[$\triangleright$] The two notions of weak equivalences coincide by \cite[Proposition 5.3]{Caviglia14}.
\item[$\triangleright$] Let $f: \mathscr P \rightarrow \mathscr Q$ of be a morphism of dg operads such that the morphism of chain complexes $f: \mathcal{P}(c_1,\ldots,c_m ; c) \rightarrow \mathcal{Q}(\phi(c_1),\ldots,\phi(c_m) ; \phi(c))$ is a degreewise epimorphism, for any colors $c_1$, \ldots, $c_m$ and $c$.  Then, $f$ is a fibration in the sense of \cite{Caviglia14} if and only the functor $j^*f: j^*\mathscr P \rightarrow j^*\mathscr Q$ is a fibration for the canonical model structure on dg categories introduced in \cite[Definition 1.6]{BergerMoerdijk13}. Besides, this canonical model structure coincides with the model structure introduced by Tabuada in \cite{Tabuada05}; then $f$ is a fibration if and only if $j^*f$ is a fibration in the sense of Tabuada, so if and only if it is a fibration in the sense of Proposition \ref{prop:caviglia}.
\end{itemize} 
\end{rmk}

\begin{thm}
When the characteristic of the field  $\mathbb{K}$ is $0$, the adjunction (\ref{adj:hcN}) is a Quillen adjunction.
\end{thm}
 
\begin{proof}
This is a straightforward consequence of Theorem \ref{thmweakequivalence} and Theorem \ref{thmfibrations}.
\end{proof}

\begin{rmk}
If $\mathbb{K}$ is not a characteristic $0$ field, a model structure as in Proposition \ref{prop:caviglia} exists on the category of reduced dg operads, i.e. dg operads with no elements of arity $0$. Moreover, we have a similar Quillen adjunction between reduced dg operads and reduced dendroidal sets.
\end{rmk}


\section{The big nerve of dg categories and dg colored operads}

In \cite[1.3.1]{Lurie12}, Lurie introduces another functor from dg categories to quasi-categories called the big nerve $\mathrm{N}^{\mathrm{big}}_{\mathrm{dg}}$ and he proves that it is point-wise equivalent to the homotopy coherent nerve $\mathrm{hcN}$. In this section, we extend Lurie's big nerve functor to dg colored operads and show that it is point-wise equivalent to the homotopy coherent nerve of dg operads $\mathrm{hcN}$. To do so, we have to reformulate Lurie's arguments since the Alexander--Whitney map is not symmetric, that is his formula cannot be applied mutatis mutandis on the operadic level.

\subsection{The Boardman--Vogt construction for simplicial operads}

We recall here the Boardman--Vogt construction for the simplicial operad $\Omega(T)$ for any tree $T$, see \cite{BergerMoerdijk06} for more details. \\

\paragraph{\sc Notation}For any integer $n\geq 0$, the set $\Delta[1]_n$ has $n+2$ elements that we denote $e_{0,n}=(0\, 0 \cdots 0)$, $e_{1,n}=(0\, 0\cdots 0\, 1)$, \ldots, $e_{n+1,n}=(1\, 1\cdots 1)$.\\

For any tree $T$, let $W_{\Delta[1]}(T)$ be the simplicial operad whose colors are the edges of $T$ and such that
\begin{itemize}
\item[$\triangleright$] $W_{\Delta[1]}(T)(e_1, \ldots ,e_m ; e) := \prod_{in(T')} \Delta[1]$ if there is a subtree $T'$ of $T$ whose leaves are $e_1$, \ldots, $e_m$ and whose root is $e$, where $in(T')$ is the set of inner edges of $T'$.
\item[$\triangleright$] $W_{\Delta[1]}(T)(e_1, \ldots ,e_m ; e):=\emptyset$, the initial simplicial set, otherwise.
\end{itemize}
The operadic composition is given by the grafting of trees where the new inner edge is labeled by the degeneracies of $1 \in \Delta[1]_0$.\\

As in Definition \ref{def:intervalH}, the simplicial set $\Delta[1]$ has a structure of interval given by the following morphisms.
\begin{itemize}
\item[$\triangleright$] the morphism $\Delta[1] \rightarrow \Delta[0]$
\item[$\triangleright$] the morphism $\mathrm{max}:\Delta[1]\times \Delta[1] \rightarrow \Delta[1]$ which sends the couple $(e_{i,n},e_{j,n})$ to $\mathrm{max}(e_{i,n},e_{j,n}):= e_{max(i,j),n}$.
\end{itemize}
This interval structure induces a cosimplicial structure on the mapping $T \mapsto W_{\Delta[1]}(T)$.

\subsection{The big nerve of dg operads}
Originally, Lurie defined the big nerve as follows. From any dg category $\mathscr C$, one can truncate the mapping spaces and then apply the Dold--Kan functor $\Gamma$ defined in Section \ref{section:dold-kan}. Since this last one is  monoidal, one gets a simplicial category. One can then can apply the nerve of simplicial categories to obtain a quasi-category.  In other words, the big nerve $\mathrm{N}^{\mathrm{big}}_{\mathrm{dg}}(\mathscr C)$ of a dg category $\mathscr C$ is the following simplicial set
$$
\mathrm{N}^{\mathrm{big}}_{\mathrm{dg}}(\mathscr C)_n:=\mathrm{Hom}_{\mathsf{sSet}\text{-}\mathsf{cat}}\big(W_{\Delta[1]}([n]),\Gamma (tr(\mathscr C))\big)\ , 
$$
where $\mathsf{sSet}\text{-}\mathsf{cat}$ is category of simplicial categories. However, notice that the functor $\Gamma$ from nonnegatively graded chain complexes to simplicial $\mathbb{K}$-modules is not symmetric monoidal. Therefore, the big nerve cannot be directly extended to dg operads using Lurie's formula. We first have to reformulate its definition.\\

Let $\mathbb{K}\text{-}\mathsf{sSet}\text{-}\mathsf{cat}$ be the category of categories enriched in simplicial $\mathbb{K}$-modules, and let $\mathsf{dg}\text{-}\mathsf{cat}^{\geq 0}$ be the category of categories enriched over the monoidal category $\mathsf{dg}\text{-}\mathsf{Mod}^{\geq 0}$ of nonnegatively graded chain complexes. The functors $\Gamma$, $C/D$ and $N$ are lax monoidal and so extend respectively to a functor from the category $\mathsf{dg}\text{-}\mathsf{cat}^{\geq 0}$ (resp. $\mathbb{K}\text{-}\mathsf{sSet}\text{-}\mathsf{cat}$) to the category $\mathbb{K}\text{-}\mathsf{sSet}\text{-}\mathsf{cat}$ (resp. $\mathsf{dg}\text{-}\mathsf{cat}^{\geq 0}$).

\begin{lemma}\label{lemma:adjunctiondoldkan}
The functor $C/D: \mathbb{K}\text{-}\mathsf{sSet}\text{-}\mathsf{cat} \rightarrow \mathsf{dg}\text{-}\mathsf{cat}^{\geq 0}$ is left adjoint to the functor $\Gamma: \mathsf{dg}\text{-}\mathsf{cat}^{\geq 0} \rightarrow \mathbb{K}\text{-}\mathsf{sSet}\text{-}\mathsf{cat}$
\end{lemma}

\begin{proof}
On the one hand, the endofunctor $N\Gamma$ of the category $\mathsf{dg}\text{-}\mathsf{cat}^{\geq 0}$ is exactly the identity functor. On the other hand, the functor $N: \mathbb{K}\text{-}\mathsf{sSet}\text{-}\mathsf{cat} \rightarrow \mathsf{dg}\text{-}\mathsf{cat}^{\geq 0}$ is fully faithful. So for any category $\mathscr C$ enriched in simplicial $\mathbb{K}$-modules and for any category $\mathscr D$ enriched over $\mathsf{dg}\text{-}\mathsf{Mod}^{\geq 0}$, we have:
\begin{align*}
\mathrm{Hom}_{\mathbb{K}\text{-}\mathsf{sSet}\text{-}\mathsf{cat}}\big(\mathscr C, \Gamma \mathscr D\big) &\simeq \mathrm{Hom}_{\mathsf{dg}\text{-}\mathsf{cat}^{\geq 0}}\big(N\mathscr C, N\Gamma \mathscr D\big)\\ & \simeq \mathrm{Hom}_{\mathsf{dg}\text{-}\mathsf{cat}^{\geq 0}}\big(N\mathscr C,  \mathscr D \big)\ .
\end{align*}
Since the functor $C/D:\mathbb{K}\text{-}\mathsf{sSet}\text{-}\mathsf{cat} \rightarrow \mathsf{dg}\text{-}\mathsf{cat}^{\geq 0}$ is isomorphic the functor $N$, it is also left adjoint to $\Gamma$.
\end{proof}

Subsequently, the big nerve $\mathrm{N}^{\mathrm{big}}_{\mathrm{dg}}(\mathscr C)$ can be rewritten as
$$
\mathrm{N}^{\mathrm{big}}_{\mathrm{dg}}(\mathscr C)_n \simeq \mathrm{Hom}_{\mathsf{dg}\text{-}\mathsf{cat}}\big(C/D(\mathbb{K}W_{\Delta[1]}([n])), \mathscr C \big)
$$ 
where $\mathbb{K}W_{\Delta[1]}([n])$ is the category enriched in simplicial $\mathbb{K}$-modules freely obtained from the simplicial category $W_{\Delta[1]}([n])$. Since the functor $C/D$ from simplicial $\mathbb{K}$-modules to nonnegatively graded chain complexes is symmetric monoidal, then this last formula can be extended to the  operadic level.

\begin{defin}[The big nerve of dg operads]
Let $\mathscr P$ be a dg operad. The {\it big nerve} $\mathrm{N}^{\mathrm{big}}_{\mathrm{dg}}(\mathscr P)$ of $\mathscr P$ is the following dendroidal set.
$$
\mathrm{N}^{\mathrm{big}}_{\mathrm{dg}}(\mathscr P)_T:=\mathrm{Hom}_{\mathsf{dg}\text{-}\mathsf{Op}}\big(C/D(\mathbb{K}W_{\Delta[1]}(T)), \mathscr P \big)\ .
$$
The big nerve is a functor from the category of dg colored operads to the category of dendroidal sets.
\end{defin}

As the homotopy coherent nerve, the big nerve admits a left adjoint.

\begin{prop}
Let $W_!^\Delta$ be the colimit preserving functor from the category of dendroidal sets to the category of simplicial operads such that $W_!^\Delta(\Omega[T]):=W_{\Delta[1]}(T)$ for any tree $T$. Then the functor $C/D(\mathbb{K}W_!^\Delta)$ is left adjoint to the big nerve.
\end{prop}

\begin{proof}
It follows from the fact that the functor $C/D$ preserves colimits.
\end{proof}

\begin{prop}
The big nerve of a dg colored operad is an $\infty$-operad.
\end{prop}

\begin{proof}
Let us consider a tree $T$ together with and inner edge $e$ and a morphism $f$ of dendroidal sets from $\Lambda^e[T]$ to the big nerve $\mathrm{N}^{\mathrm{big}}_{\mathrm{dg}}(\mathscr P)$ of a dg colored operad $\mathscr P$. We denote by $l_1$, \ldots, $l_m$ the leaves of the tree $T$ and by $r$ its root. The images of $l_i$ (resp. $r$) under the morphism $f$ are denoted $c_i$ (resp. $c$). The morphism $f$ and the adjunction $C/D\, \mathbb{K}W_!^{\Delta} \dashv \mathrm{N}^{\mathrm{big}}_{\mathrm{dg}}$ give a map of chain complexes from $C/D(\mathbb{K}W_!^{\Delta}(\Lambda^e[T]))(l_1,\ldots,l_m;r)$ to $\mathcal{P}(c_1,\ldots,c_m; c)$. Through the adjunction $C/D \vdash \Gamma$,  it corresponds  to a map $f'$ of simplicial sets from $W_!^{\Delta}(\Lambda^e[T])(l_1,\ldots,l_m;r) $ to $\Gamma\mathcal{P}(c_1,\ldots,c_m;c)$. Since the map $W_!^{\Delta}(\Lambda^e[T])(l_1,\ldots,l_m;r) \hookrightarrow W_!^{\Delta}(\Omega[T])(l_1,\ldots,l_m;r)$ is anodyne (see \cite[Section 7]{MoerdijkWeiss09}) and since $\Gamma\mathcal{P}(c_1,\ldots,c_m)$ is a Kan complex, there is a lifting $f''$ of the map $f'$:
$$
\xymatrix{W_!^{\Delta}(\Lambda^e[T])(l_1,\ldots,l_m;r) \ar[r]^(0.57){f'} \ar@{^(->}[d] & \Gamma\mathcal{P}(c_1,\ldots,c_m;c)\\
W_!^{\Delta}(\Omega[T])(l_1,\ldots,l_m;r)\ . \ar[ru]_{f''}}
$$
As the map $i \in \{1,\ldots,m \} \mapsto l_i$ is an injection, for any permutation $\sigma \in \mathbb{S}_m$ the structural isomorphisms $W_!^{\Delta}(\Omega[T])(l_1,\ldots,l_m;r)\simeq W_!^{\Delta}(\Omega[T])(l_{\sigma(1)},\ldots,l_{\sigma(m)};r)$ and $\Gamma\mathcal{P}(c_1,\ldots,c_m;c) \simeq \Gamma\mathcal{P}(c_{\sigma(1)},\ldots,c_{\sigma(m)};c)$ give us a morphism 
$$
W_!^{\Delta}(\Omega[T])(l_{\sigma(1)},\ldots,l_{\sigma(m)};r) \rightarrow \Gamma\mathcal{P}(c_{\sigma(1)},\ldots,c_{\sigma(m)};c)\ .
$$
Moreover, for any other inputs $e_1$, \ldots, $e_l$ and output $e_0$, the simplicial set $W_!^{\Delta}(\Omega[T])(e_1,\allowbreak \ldots,e_l;e_0)$ is exactly $W_!^{\Delta}(\Lambda^e[T])(e_1, \ldots,e_l;e_0)$. So we have maps from $W_!^{\Delta}(\Omega[T])(e_1, \ldots,e_l;e_0)$ to $\Gamma\mathcal{P}$. All these maps induce a morphism of dg colored operads from $C/D(\mathbb{K}W_!^{\Delta}(\Omega[T]))$ to $\mathscr P$ and so a morphism of dendroidal sets from $ \Omega[T]$ to $\mathrm{N}^{\mathrm{big}}_{\mathrm{dg}}(\mathscr P)$ which extends the morphism $f$.
\end{proof}

\subsection{From the big nerve to the homotopy coherent nerve}

In this section, we construct a morphism of functors $\alpha^*$ from $\mathrm{N}^{\mathrm{big}}_{\mathrm{dg}} $ to $\mathrm{hcN}$.\\

The mappings $h \mapsto e_{1,1}$, $h_0 \mapsto e_{0,0}$ and $h_1 \mapsto e_{1,0}$ induce an isomorphism from the chain complex $H$ defined in Section \ref{sectionboardmanvogt} to the chain complex $C/D(\mathbb{K}\Delta[1]):=C(\mathbb{K}\Delta[1])/D(\mathbb{K}\Delta[1])$. The functor $C/D$ is symmetric monoidal through the Eilenberg--Zilber map; see \cite[3.1]{Faonte13}. So, we get morphisms $\alpha_k$ from $H^{\otimes k}$ to $C/D(\mathbb{K}\Delta[1]^k)$ for any integer $k \geq 1$.  Let us describe them. On the one hand, we have:
$$
\alpha_k(h \otimes \cdots \otimes h)=\sum_{\sigma \in \mathbb{S}_k} \mathrm{sign}(\sigma) e_{\sigma(k),k}  \otimes \cdots \otimes e_{\sigma(1),k} \ .
$$
On the other hand, let $A \in H^{\otimes k} $ and $ B \in H^{\otimes l}$ be homogeneous elements whose cumulate degree is $n$ and such that $\alpha_{k+l}(A \otimes B)=\sum_{i \in I} A_i \otimes B_i$ where $A_i \in C/D(\mathbb{K}\Delta[1]^k)$ and $B_i \in C/D(\mathbb{K}\Delta[1]^l)$. Then we have
\begin{align*}\label{eilenberg-zilberforh}
&\alpha_{k+l+1}(A \otimes h_0 \otimes B)=\sum_{i \in I} A_i \otimes e_{0,n} \otimes B_i,\\
&\alpha_{k+l+1}(A \otimes h_1 \otimes B)=\sum_{i \in I} A_i \otimes e_{n+1,n} \otimes B_i\ .
\end{align*}
Note also that the morphism $\alpha_{k,l}$ from $C/D(\mathbb{K}\Delta[1]^k) \otimes C/D(\mathbb{K}\Delta[1]^l)$ to $C/D(\mathbb{K}\Delta[1]^{k+l})$ given by the Eilenberg--Zilber map satisfy the equation $\alpha_{k,l}(\alpha_k \otimes \alpha_l)=\alpha_{k+l}$. This follows from the fact that the Eilenberg--Zilber map is the structural map making $C/D$ into a symmetric monoidal functor.

\begin{prop}\label{prop:dgbvtosbv}
These maps $\alpha_k$ induce a morphism of dg operads $\alpha_T: W_H(T) \rightarrow C/D\mathbb{K}W_{\Delta[1]}(T)$, for any tree $T$.  Moreover, these morphisms are functorial in $T$. Subsequently, they induce
\begin{itemize}
\item[$\triangleright$] a morphism of functors $\alpha$ from $W_!^{dg}$ to $C/D\mathbb{K}W_!^{\Delta}$ \ ,
\item[$\triangleright$] and a morphism of functors $\alpha^*: \mathrm{N}^{\mathrm{big}}_{\mathrm{dg}} \rightarrow \mathrm{hcN}$\ .
\end{itemize}
\end{prop}

\begin{proof}
The maps $\alpha_k$ induce morphisms of dg $\mathbb{S}$-modules $\alpha_T: W_H(T) \rightarrow C/D\mathbb{K}W_{\Delta[1]}(T) $ for any tree $T$. We have to show that the morphisms $\alpha_T$ are morphisms of dg operads and that they are functorial with respect to the trees $T$. 
\begin{itemize}
\item[$\triangleright$] The former property follows from the fact that for any integers $k,l \geq 0$, the following square is commutative.
$$
\xymatrix@R=10pt@C=25pt{H^{\otimes k} \otimes H^{\otimes l} \ar[r]^(0.34){\alpha_k \otimes \alpha_l} \ar[dd] & C/D(\mathbb{K}\Delta[1]^k) \otimes C/D(\mathbb{K}\Delta[1]^l) \ar[d]^{\alpha_{k,l}}\\
& C/D(\mathbb{K}\Delta[1]^{k+l}) \ar[d]\\
H^{\otimes k+1+l} \ar[r]_{\alpha_{k+1+l}} & C/D(\mathbb{K}\Delta[1]^{k+1+l})\ , }
$$
where the left vertical map sends $A \otimes B \in H^{\otimes k} \otimes H^{\otimes l}$ to $A \otimes h_1 \otimes B$ and where the bottom-right vertical map is the functorial image under $C/D\, \mathbb{K}(-)$ of the morphism from $\Delta[1]^k \times \Delta[1]^l$ to $\Delta[1]^{k+1+l}$ which sends $(A,B)$ to $(A,e_{n+1,n},B)$.
\item[$\triangleright$] Let us show that the morphisms $\alpha_T$ are functorial with respect to the trees $T$. It is straightforward to show that for any coface $\delta: T \rightarrow T'$, we have $C/D\, \mathbb{K}W_{\Delta[1]}(\delta) \alpha_T =\alpha_{T'} W_{H}(\delta)$. To prove that the same equation holds for a codegeneracy, it suffices to note that the map
$$
\xymatrix{h \otimes A \in H^{\otimes k} \ar@{|->}[r]^{\alpha_k} & \sum_i \pm e_{i,n} \otimes A_i \ar@{|->}[r] & \sum_i \pm A_i \in C/D(\mathbb{K}\Delta[1]^{k-1})}
$$
is $0$.
\end{itemize}

\end{proof}

The goal of the end of Section 4 is to prove the following theorem.

\begin{thm}\label{thm:euivalencebigsmall}
For any dg colored operad $\mathscr P$, the morphism of dendroidal sets $\alpha^*(\mathscr P): \mathrm{N}^{\mathrm{big}}_{\mathrm{dg}}(\mathscr P) \rightarrow \mathrm{hcN}(\mathscr P)$ is a weak equivalence.
\end{thm}

We already know that the colors of $\mathrm{N}^{\mathrm{big}}_{\mathrm{dg}}(\mathscr P)$ and the colors $\mathrm{hcN}(\mathscr P)$ are both the colors of $\mathscr P$ and that $\alpha^*(\mathscr P)$ is the identity on these. Hence, the morphism $\alpha^*(\mathscr P)$ is essentially surjective. So, as both $\mathrm{N}^{\mathrm{big}}_{\mathrm{dg}}(\mathscr P)$ and $\mathrm{hcN}(\mathscr P)$ are $\infty$-operads, we only have to prove that $\alpha^*(\mathscr P)$ is fully faithful to prove the theorem; that is, we have to show that for any integer $m$ and for any colors $c_1$, \ldots, $c_m$ and $c$ of $\mathscr P$, the map:
$$
\alpha^*(\mathscr P)(c_1,\ldots,c_m;c):\mathrm{N}^{\mathrm{big}}_{\mathrm{dg}}(\mathscr P)^L(c_1,\ldots,c_m;c) \rightarrow \mathrm{hcN}(\mathscr P)^L(c_1,\ldots,c_m;c) 
$$
is a weak equivalence of simplicial sets for the Kan--Quillen model structure.

\subsection{The cosimplicial simplicial set Q}

We introduce here a cosimplicial simplicial set denoted $Q$ which will allow us to deal with the operations space $\mathrm{N}^{\mathrm{big}}_{\mathrm{dg}}(\mathscr P)^L(c_1,\ldots,c_m;c)$ of the big nerve. Lurie introduced in \cite[2.2.2]{Lurie09} a very similar cosimplicial simplicial set $Q^{\bullet}$ in a slightly different context and with different conventions. The purpose of this subsection is to recall some results of \cite[Section 2.2.2]{Lurie09} about $Q^\bullet$ which extend directly to $Q$.\\

For any integer $n \geq 0$, let $D_{m,n}$ be the following colimit of dendroidal sets.
$$
D_{m,n}:= C_{m,n} \coprod_{\Delta[n]} \Delta[0]
$$
The colors of  of $D_{m,n}$ are the leaves $l_1$, \ldots, $l_m$ and the root $r$ of the tree $C_{m,n}$. Since $\Delta[-]$ is a cosimplicial simplicial set, $D_{m,-}$ is a cosimplicial dendroidal set. 

\begin{defin}[The cosimplicial simplicial set $Q$]
Let $Q$ be the cosimplicial simplicial set defined by
$$
Q[n]:= W_!^{\Delta}(D_{m,n})(l_1, \ldots,l_m;r)
$$
for any integer $n \geq 0$. 
\end{defin}

\begin{prop}\label{prop:isobig}
There is an isomorphism of simplicial sets:
$$
\mathrm{N}^{\mathrm{big}}_{\mathrm{dg}}(\mathscr P)^L(c_1,\ldots,c_m;c) \simeq \mathrm{Hom}_{\mathsf{dg}\text{-}\mathsf{Mod}}\big(C/D\mathbb{K}Q[-], \mathcal{P}(c_1,\ldots,c_m;c) \big)\ .
$$
\end{prop}

\begin{proof}
A $n$-vertex of $\mathrm{N}^{\mathrm{big}}_{\mathrm{dg}}(\mathscr P)^L(c_1,\ldots,c_m;c)$ is a morphism of dendroidal sets from $D_{m,n}$ to $\mathrm{N}^{\mathrm{big}}_{\mathrm{dg}}(\mathscr P)$ which sends the colors $l_i$ to $c_i$ and $r$ to $c$. So it is a morphism of dg operads from $C/D\mathbb{K}W_!^{\Delta}(D_{m,n})$ to $\mathscr P$ which sends the colors $l_i$ to $c_i$ and $r$ to $c$ and so it is a morphism of chain complexes from $C/D\mathbb{K}(Q[n])$ to $\mathcal{P}(c_1,\ldots,c_m;c)$.
\end{proof}

The simplicial set $Q[n]$ admits the following description.
$$
Q[n] \simeq\Delta[1]^n   / \sim
$$
where $(A , e_{k+1,k}, B)  \sim (A', e_{k+1,k},B)$. We now describe the cosimplicial structure of $Q[-]$. Let $(q_1,\ldots,q_n) \in Q[n]_l$ be a $l$-vertex of $Q[n]$ represented by a $l$-vertex of $\Delta[1]^n$ and let $\delta_i: [n] \rightarrow [n+1]$ (resp. $\sigma_i: [n] \rightarrow [n-1]$) be a coface (resp. a codegeneracy).
\begin{itemize}
\item[$\triangleright$] If $i \geq 1$: 
\begin{align*}
&Q(\delta_i)(q_1,\ldots,q_n)=(q_1,\ldots,q_{i-1},e_{0,l},q_{i},\ldots,q_n)\ ,\\
&Q(\sigma_i)(q_1,\ldots,q_n)=(q_1,\ldots,\mathrm{max}(q_i,q_{i+1}),\ldots,q_n)\ .
\end{align*}

\item[$\triangleright$] If $i=0$:
\begin{align*}
&Q(\delta_i)(q_1,\ldots,q_n)=(e_{l+1,l},q_1,\ldots,q_n)\ ,\\
&Q(\sigma_i)(q_1,\ldots,q_n)=(q_2,\ldots,q_n)\ .
\end{align*}

\end{itemize}

\begin{defin}[A morphism from $Q$ to $\Delta$, \cite{Lurie09} Remark 2.2.2.6]
Let $\beta$ be the morphism of cosimplicial simplicial sets $\beta: Q[-] \rightarrow \Delta[-]$ which sends the element $e_{i_1,k} \otimes \cdots \otimes e_{i_n,k} \in Q[n]_k$ to
$$
[\mathrm{max}\{j|i_j=k+1\} \leq \mathrm{max}\{j | i_j \geq k\} \leq \cdots \leq \mathrm{max}_j\{j| i_j \geq 1\}] \in \Delta[n]_k
$$
with the convention $\mathrm{max}(\emptyset)=0$. We denote by $b$ the morphism of cosimplicial chain complexes $b:=C/D\mathbb{K}\beta: C/D\mathbb{K}Q[-] \rightarrow C/D\mathbb{K}\Delta[-]$.
\end{defin}

\begin{lemma}\label{lemma:lurie2.2.2}
The morphism of functors induced by the morphism $b$
$$
b^*: \mathrm{Hom}_{\mathsf{dg}\text{-}\mathsf{Mod}}\big(C/D\mathbb{K}\Delta[-],- \big) \rightarrow \mathrm{Hom}_{\mathsf{dg}\text{-}\mathsf{Mod}}\big(C/D\mathbb{K}Q[-],- \big)
$$
is a point-wise  equivalence, that is the morphism of simplicial sets $b^*(V)$ is a weak equivalence of simplicial sets for the Kan--Quillen model structure,  for any chain complex $V$.
\end{lemma}

\begin{proof}
It is a consequence of \cite[Proposition 2.2.2.7]{Lurie09} and \cite[Proposition 2.2.2.9]{Lurie09}.
\end{proof}

\subsection{The big nerve is equivalent to the homotopy coherent nerve}

We now prove Theorem \ref{thm:euivalencebigsmall}. For that purpose, we use the method of the proof of \cite[Proposition 1.3.1.17]{Lurie12}.

\begin{prop}\label{prop:isohc}
There is an isomorphism of simplicial sets:
$$
\mathrm{hcN}(\mathscr P)^L(c_1,\ldots,c_m;c) \simeq \mathrm{Hom}_{\mathsf{dg}\text{-}\mathsf{Mod}}\big(C/D\mathbb{K}\Delta[-], \mathcal{P}(c_1,\ldots,c_m;c) \big)\ .
$$
\end{prop}

\begin{proof}
We know that
$$
\mathrm{hcN}(\mathscr P)^L(c_1,\ldots,c_m;c) \simeq \mathrm{Hom}_{\mathsf{dg}\text{-}\mathsf{Mod}}\big(W_!^{dg}(D_{m,-})(l_1, \ldots,l_m;r), \mathcal{P}(c_1,\ldots,c_m;c)\big)\ .
$$
The chain complex $W_!^{dg}(D_{m,n})(l_1, \ldots,l_m;r)$ can be described as $H^{\otimes n}/\sim$, where $A \otimes h_1 \otimes B \sim 0$, if the degree of $A$ is different from $0$ and $A \otimes h_1 \otimes B \sim h_1 \otimes \cdots \otimes h_1 \otimes B$ otherwise. Let $A$ be an homogeneous element of $H^{\otimes n}$. It has the form $A= h_1^{\otimes i_0} \otimes A'$ where $h$ occupies the places (counted from $1$ to $n$) $i_1 < \cdots < i_k$ and $A'$ does not contain $h_1$. Then, the mapping which sends $A$ to the element $[i_0 \leq i_1 \leq\cdots \leq i_k]$ gives us an isomorphism of cosimplicial chain complexes:
$$
W_!^{dg}(D_{m,-})(l_1, \ldots,l_m;r) \simeq  C/D\mathbb{K}\Delta[-]\ .
$$
\end{proof}

\begin{defin}[From $\Delta$ to $Q$]
We have introduced a morphism of functors $\alpha$ from $W_!^{dg}$ to $ C/D\mathbb{K}W_!^{\Delta}$. If we apply it to the dendroidal sets $(D_{m,n})_{n \in \mathbb{N}}$, we get
\begin{itemize}
\item[$\triangleright$] a morphism of cosimplicial chain complexes $a$ from $C/D\mathbb{K}\Delta[-]$ to $C/D\mathbb{K}Q[-]$,
\item[$\triangleright$] and a morphism of functors
$$
a^*: \mathrm{Hom}_{\mathsf{dg}\text{-}\mathsf{Mod}}\big(C/D\mathbb{K}Q[-],- \big) \rightarrow \mathrm{Hom}_{\mathsf{dg}\text{-}\mathsf{Mod}}\big(C/D\mathbb{K}\Delta[-],- \big)\ .
$$
\end{itemize}
\end{defin}

\begin{lemma}
The endomorphism $ba$ of the cosimplicial chain complexes $C/D(\mathbb{K}\Delta[-])$ is the identity.
\end{lemma}

\begin{proof}
It follows from the description of the morphisms $a$ and $b$ given above, and a straightforward computation.
\end{proof}

\begin{proof}[Proof of Theorem \ref{thm:euivalencebigsmall}]
By Proposition \ref{prop:isobig} and Proposition \ref{prop:isohc}, the map of simplicial sets $a^*(\mathcal{P}(c_1,\ldots,\allowbreak c_m;c))$ from $\mathrm{Hom}_{\mathsf{dg}\text{-}\mathsf{Mod}}(C/D\mathbb{K}Q[-],\mathcal{P}(c_1,\ldots,c_m;c))$ to $\mathrm{Hom}_{\mathsf{dg}\text{-}\mathsf{Mod}}(C/D\mathbb{K}\Delta[-],\mathcal{P}(c_1,\ldots,c_m;c))$ can be considered as a map from $\mathrm{N}^{\mathrm{big}}_{\mathrm{dg}}(\mathscr P)^L(c_1,\ldots,c_m;c) $ to $ \mathrm{hcN}(\mathscr P)^L(c_1,\ldots,c_m;c) $. By construction, it is exactly
$$
\alpha^*(\mathscr P)(c_1,\ldots,c_m;c):\mathrm{N}^{\mathrm{big}}_{\mathrm{dg}}(\mathscr P)^L(c_1,\ldots,c_m;c) \rightarrow \mathrm{hcN}(\mathscr P)^L(c_1,\ldots,c_m;c)\ .
$$
Since its composition with the map $b^*(\mathcal{P}(c_1,\ldots,c_m;c))$ is the identity and since the map $b^*(\mathcal{P}\allowbreak(c_1,\ldots,\allowbreak c_m;c))$ is a weak equivalence, then $a^*(\mathcal{P}(c_1,\ldots,c_m;c))$ is a weak equivalence. So, for any dg colored operad $\mathscr P$, the morphism of $\infty$-operads $\alpha^*(\mathscr P)$ from $\mathrm{N}^{\mathrm{big}}_{\mathrm{dg}}(\mathscr P)$ to $\mathrm{hcN}(\mathscr P)$ is fully faithful. Since it is essentially surjective, then it is an equivalence.
\end{proof}

\begin{cor}
When the characteristic of the field $\mathbb{K}$ is zero, the adjunction $C/D\mathbb{K}W_!^{\Delta} \dashv \mathrm{N}^{\mathrm{big}}_{\mathrm{dg}}$ is a Quillen adjunction, with respect to the model structure  on the category $\mathsf{dg}\text{-}\mathsf{Op}$ of dg operads given in Proposition \ref{prop:caviglia}.
\end{cor}

\begin{proof}
It suffices to prove that the functor $\mathrm{N}^{\mathrm{big}}_{\mathrm{dg}}$ preserves weak equivalences and fibrations. Since it is equivalent to the homotopy coherent nerve which preserves weak equivalences, then it preserves weak equivalences. Let $f: \mathscr P \rightarrow \mathscr Q$ be a fibration of dg operads and let $\phi$ its underlying function between colors. On the one hand, since the functors from dg operads to set-theoretical operads  $\tau_d \mathrm{N}^{\mathrm{big}}_{\mathrm{dg}}$ and $H_0 $ are isomorphic, then the functor $i^*\tau_d \mathrm{N}^{\mathrm{big}}_{\mathrm{dg}}(f)$ is an isofibration. On the other hand, let $T$ be a tree with $m \geq 0$ leaves $l_1$, \ldots, $l_m$, a root denoted by $r$ and an inner edge $e$. Consider the following commutative square of dendroidal sets
$$
\xymatrix@R=25pt@C=25pt{\Lambda^e[T] \ar[r] \ar[d] & \mathrm{N}^{\mathrm{big}}_{\mathrm{dg}}(\mathscr P) \ \,\ar[d]^{\mathrm{N}^{\mathrm{big}}_{\mathrm{dg}} (f)}  \\
\Omega[T] \ar[r] & \mathrm{N}^{\mathrm{big}}_{\mathrm{dg}}(\mathscr Q) \ .}
$$
It induces a commutative square of chain complexes.
$$
\xymatrix@R=25pt@C=25pt{C/D\mathbb{K}W_!^{\Delta}\Lambda^e[T](l_1,\ldots,l_m;r) \ar[r] \ar[d] & \mathcal{P}(c_1,\ldots,c_m;c) \ar[d]^{f(c_1,\ldots,c_m;c)}  \\
C/D\mathbb{K}W_!^{\Delta}\Omega[T](l_1,\ldots,l_m;r) \ar[r] & \mathcal{Q}(\phi(c_1),\ldots,\phi(c_m);\phi(c)) }
$$
Since the morphism $f(c_1,\ldots,c_m;c)$ is a fibration of chain complexes and since the left vertical map is a trivial cofibration, then this square has a lifting. Since the function $i \in \{1, \ldots, m \} \mapsto l_i$ is injective, we get, for any permutation $\sigma \in \mathbb{S}_m$, a lifting of the same square but where $(l_1,\ldots,l_m)$ (resp. $(c_1,\ldots,c_m)$) is replaced by $(l_{\sigma(1)},\ldots,l_{\sigma(m)})$ (resp. $(c_{\sigma(1)},\ldots,c_{\sigma(m)})$). This gives us a lifting of the first square. By the characterization of the fibrations between fibrant dendroidal sets given in Theorem \ref{thmcisinskimoerdijk}, then $\mathrm{N}^{\mathrm{big}}_{\mathrm{dg}}(f)$ is a fibration.
\end{proof}


\appendix
\section*{Appendix}
In the literature, there are  two ways to define the fully faithful morphisms of $\infty$-operads $\allowbreak f:P \rightarrow Q$:
\begin{itemize}
\item[$\triangleright$] for any integer $m$ and any colors $c_1$, \ldots, $c_m$, $c$ of $P$, the map $ P(c_1,\ldots,c_m ; c) \rightarrow Q(f(c_1),\ldots,\allowbreak f(c_m) ; f(c))$ is a weak homotopy equivalence of simplicial sets.
\item[$\triangleright$] for any integer $m$ and any colors $c_1$, \ldots, $c_m$, $c$ of $P$, the map $\allowbreak P^L(c_1,\ldots,c_m;c) \rightarrow Q^L(f(c_1),\ldots,f(c_m);f(c))$ is a weak homotopy equivalence of simplicial sets.
\end{itemize}  
The goal of this appendix is to prove that these two definitions are equivalent that is to prove the following theorem.

\begin{thm}\label{thmfullyfaithful}
For any $\infty$-operad $P$, for any integer $m$ and any colors $c$, $c_1$, \ldots, $c_m$ of $P$, the simplicial set $P(c_1,\ldots,c_m;c)$ is homotopy equivalent to the simplicial set $P^L(c_1,\ldots,c_m;c)$. In other words, there is a chain of weak homotopy equivalences between $P(c_1,\ldots,c_m;c)$ and $P^L(c_1,\ldots,c_m;c)$. Furthermore, this construction is functorial with respect to $P$.
\end{thm}

Let $P^{(\Delta[q])}$ be the functorial Reedy fibrant resolution of $P$ defined in \cite[3.2]{CisinskiMoerdijk13a}. The two following squares are pullbacks:
$$
\xymatrix@R=30pt@C=20pt{P(c_1,\ldots,c_m;c)_q \ar[r] \ar[d] & \mathrm{Hom}_{\mathsf{dSet}}(\Omega[C_{m}], P^{(\Delta[q])}) \ar[d]\\
\Delta[0] \ar[r]_(0.25){(c,c_1,\ldots,c_m)} & \mathrm{Hom}_{\mathsf{dSet}}(\Delta[0] \sqcup\coprod_{i=1}^m \Delta[0], P^{(\Delta[q])})\ ,}
$$
$$
\xymatrix@R=30pt@C=20pt{P^L(c_1,\ldots,c_m;c)_p \ar[r] \ar[d] & \mathrm{Hom}_{\mathsf{dSet}}(\Omega[C_{m,p}], P) \ar[d]\\
\Delta[0] \ar[r]_(0.25){(c,c_1,\ldots,c_m)} & \mathrm{Hom}_{\mathsf{dSet}}(\Delta[p] \sqcup\coprod_{i=1}^m \Delta[0], P)\ .}
$$
The left vertical maps of these two diagrams are induced by cofaces maps targeting the trees $C_m$ and $C_{m,n}$. Let $\{M^P_{p,q}\}_{p,q \in \mathbb{N}}$ be the pullback of the following diagrams:
$$
\xymatrix@R=30pt@C=20pt{M^P_{p,q} \ar[rr] \ar[d] && \mathrm{Hom}_{\mathsf{dSet}}(\Omega[C_{m,p}], P^{(\Delta[q])}) \ar[d]\\
\Delta[0] \ar[rr]_(0.25){(c,c_1,\ldots,c_m)} && \mathrm{Hom}_{\mathsf{dSet}}(\Delta[p] \sqcup\coprod_{i=1}^m \Delta[0], P^{(\Delta[q])})}
$$
The collection $\{M^P_{p,q}\}_{p,q}$ has a canonical structure of a bisimplicial set. The simplicial set $M^P_{0,-}$ (resp. $M^P_{-,0}$) is equal to $P(c_1,\ldots,c_m;c)$ (resp. $P^L(c_1,\ldots,c_m;c)$). From now on, we will denote $P^{(\Delta[q])}$ simply by $P_q$.

\begin{lemma}\label{lemmaappendix1}
For any face map $d_i: P_q \rightarrow P_{q-1}$, the induced morphism of simplicial sets $M^P_{-,q} \rightarrow M^P_{-,q-1}$ is a trivial fibration.
\end{lemma}

\begin{proof}[Proof of Lemma \ref{lemmaappendix1}]
Let $k$ be an integer and suppose that we have the following diagram:
$$
\xymatrix@R=25pt@C=25pt{\partial\Delta[k] \ar[r] \ar[d] & M^P_{-,q} \ar[d]^{d_i} \\
\Delta[k] \ar[r] & M^P_{-,q-1}\ .}
$$
It corresponds to elements $b_0$, \ldots, $b_k$ in $M^P_{k-1,q}$ and $a$ in $M^P_{k,p-1}$ having coherent face relations. The color $c$ of $P$ induces a morphism $c: \Delta[k] \rightarrow P_q$. Then, $c$, $b_0$, \ldots and $b_k$ together give a map $\partial\Omega[C_{m,k}] \rightarrow P_q$ which fits into the following diagram:
$$
\xymatrix@R=25pt@C=20pt{\partial \Omega[C_{m,k}] \ar[rr]^(0.6){(b_0,\ldots,b_k,c)} \ar[d] && P_q \ar[d]^{d_i} \\
\Omega[C_{m,k}] \ar[rr] && P_{q-1}\ .}
$$ 
As the face $d_i : P_q \rightarrow P_{q-1}$ is a trivial fibration, then the former square has a lifting; therefore, the first diagram has a lifting. As this is true for any integer $k$ and for any such diagram, this proves the lemma.
\end{proof}

\begin{lemma}\label{lemmaappendix2}
let $\delta_p$ be the unique external coface $C_{m,p-1} \rightarrow C_{m,p}$ which omits the lower vertex of $C_{m,p}$. The face induced $M^P_{p,-} \rightarrow M^P_{p-1,-}$ is a trivial fibration.
\end{lemma}

\begin{proof}[Proof of Lemma \ref{lemmaappendix2}]
Let $k \geq 1$ be an integer. On the one hand, let $\lim (d_*,P)_{k-1}$ be the projective limit of the diagram of dendroidal sets made up of:
\begin{itemize}
\item[$\triangleright$] for any integer $0\leq i \leq n$ a copy of $P_{k-1}$ denoted by $(d_i,P_{k-1})$.
\item[$\triangleright$] if $k \geq 2$, for any pair of integers $(i,j)$ such that $0 \leq i < j \leq n$, a copy $(d_id_j,P_{k-2})$ of $P_{k-2}$ and maps
$$
\xymatrix{(d_i,P_k) \ar[r]^(0.45){d_{j-1}} & (d_id_j,P_k) & (d_j,P_k) \ar[l]_(0.45){d_i}}\ .
$$  
\end{itemize}
Then, the canonical map $P_k \rightarrow \lim (d_*,P)_{k-1}$ is a fibration; see \cite[4.3]{DwyerKan80} for more details. Suppose that we have the following diagram for an integer $k \geq 0$:
$$
\xymatrix{\partial\Delta[k] \ar[r] \ar[d] & M^P_{p,-} \ar[d]^{d_p} \\
\Delta[k] \ar[r] & M^P_{p-1,-}}
$$
If $k=0$ this diagram corresponds to the data of an element of $x\in M^P_{p-1,0}$. The degeneracy $s_{p-1}x$ of $x$ gives a lifting of the diagram. Suppose now that $k \geq 1$. The diagram induces an element $b \in (\lim (d_*,P)_{k-1})_{C_{m,p}}$ and an element $a \in (P_k)_{C_{m,p-1}}$. They fit in the following square:
$$
\xymatrix{\partial^{ext} \Omega[C_{m,p}] \ar[r]^{(a,c)} \ar[d] & P_k \ar[d] \\
\Omega[C_{m,p}] \ar[r]_b & \lim (d_*,P)_{k-1}\ .}
$$
The right vertical map is a fibration and the left verical map is a trivial cofibration; see \cite[Lemma 5.1]{MoerdijkWeiss09}. Then the square has a lifting which induces a lifting of the first square.
\end{proof}

\begin{proof}[Proof of Theorem \ref{thmfullyfaithful}]
By the two former lemmata, it is straightforward to prove that all the face maps $M^P_{p,-} \rightarrow M^P_{p-1,-}$ and $M^P_{-,q} \rightarrow M^P_{-,q-1}$ and degeneracies $M^P_{p,-} \rightarrow M^P_{p+1,-}$ and $M^P_{-,q} \rightarrow M^P_{-,q+1}$ are weak homotopy equivalences of simplicial sets. Then by \cite[Corollary 15.11.12]{Hirshhorn03}, we have a chain of weak homotopy equivalences:
$$
\xymatrix{P(c_1,\ldots,c_m;c) \ar[r]& diag(M^P) &P^L(x_1,\ldots,x_m;x)\ , \ar[l]}
$$
where $diag(M^P)$ is the diagonal of the bisimplicial set $M^P$. Furthermore, as the Reedy fibrant resolution $P^{(\Delta[-])}$ of $P$ is functorial in $P$, a morphism $f:P \rightarrow Q$ of $\infty$-operads induces the following diagram:
$$
\xymatrix{P(c_1,\ldots,c_m;c) \ar[d] \ar[r]& diag(M^P) \ar[d] &P^L(c_1,\ldots,c_m;c) \ar[l]\ar[d]\\
Q(f(c_1),\ldots,f(c_m);f(c)) \ar[r]& diag(M^Q) &Q^L(f(c_1),\ldots,f(c_m);f(c)) \ , \ar[l]}
$$
where the horizontal maps are weak homotopy equivalences and the vertical ones are canonically induced by $f$.
\end{proof}

\begin{cor}
The two definitions of the fully-faithful morphisms of $\infty$-operads are equivalent.
\end{cor}

\begin{proof}
Consider the former diagram. As the horizontal maps are weak homotopy equivalences, then the left vertical map is a weak equivalence if and only if the central vertical map is a weak equivalence if and only if the right vertical map is a weak equivalence.
\end{proof}


\bibliographystyle{amsalpha}
\bibliography{biblg}

\end{document}